\documentclass[12pt]{article}
\usepackage[utf8]{inputenc}
\usepackage[all]{xy}
\usepackage{amsthm, amsmath, amssymb, enumerate,graphicx}
\usepackage{microtype}
\usepackage{fullpage}
\usepackage{tikz-cd}
\usepackage{mathtools}
\usepackage{colonequals}
\usepackage{enumitem}
\usepackage{mathrsfs}
\usepackage[colorlinks=true,backref=page]{hyperref}

\begin{document}
\title{GCD and ABC inequalities for abelian varieties over function fields}
\author{Guoquan Gao}
\maketitle

\begin{abstract}
We prove two height inequalities for rational points on abelian varieties over function fields of characteristic 0. The first is the GCD inequality, which shows that the height with respect to a closed subscheme of an abelian variety of codimension at least 2 is much smaller than the usual height function. This can be viewed as a compact and geometric analogue of a theorem of Levin. The second is a special case of the (geometric) abc conjecture proposed by Vojta, which compares the truncated counting function and the usual counting function. This is a compact analogue of a theorem of Guo, Nguyen, Sun and Wang, and also resolves a conjecture of Buium, which was only known for traceless abelian varieties or constant abelian varieties previously.

The proofs of these inequalities follow a common three-step pattern: from microscopic local estimates to local estimates, and then to global estimates. Building on the canonical Xie--Yuan heights, we control discrete counting functions by continuous integrals of Betti forms. The main new tools are uniform Nevanlinna-type estimates for families of linear entire curves on abelian varieties; combined with covering arguments based on Besicovitch theory and averaging, these handle overlaps and gaps precisely.
\end{abstract}

\theoremstyle{plain}
\newtheorem{thm}{Theorem}[section]
\newtheorem{theorem}[thm]{Theorem}
\newtheorem{cor}[thm]{Corollary}
\newtheorem{corollary}[thm]{Corollary}
\newtheorem{lem}[thm]{Lemma}
\newtheorem{lemma}[thm]{Lemma}
\newtheorem{pro}[thm]{Proposition}
\newtheorem{proposition}[thm]{Proposition}
\newtheorem{prop}[thm]{Proposition}
\newtheorem{definition}[thm]{Definition}
\newtheorem{assumption}[thm]{Assumption}
\newtheorem{conjecture}[thm]{Conjecture}
\newtheorem{sublemma}[thm]{Sub-lemma}

\theoremstyle{remark}
\newtheorem{remark}[thm]{Remark}
\newtheorem{example}[thm]{Example}
\newtheorem{remarks}[thm]{Remarks}
\newtheorem{problem}[thm]{Problem}
\newtheorem{exercise}[thm]{Exercise}
\newtheorem{situation}[thm]{Situation}
\newtheorem{Question}[thm]{Question}
\newtheorem{caution}[thm]{Caution}

\numberwithin{equation}{subsection}

\newcommand{\ZZ}{\mathbb{Z}}
\newcommand{\CC}{\mathbb{C}}
\newcommand{\QQ}{\mathbb{Q}}
\newcommand{\RR}{\mathbb{R}}
\newcommand{\HH}{\mathcal{H}}     

\newcommand{\ad}{\mathrm{ad}}            
\newcommand{\NT}{\mathrm{NT}}
\newcommand{\nonsplit}{\mathrm{nonsplit}}
\newcommand{\Pet}{\mathrm{Pet}}
\newcommand{\Fal}{\mathrm{Fal}}

\newcommand{\cs}{{\mathrm{cs}}}

\newcommand{\ZZn}{\mathbb{Z}[1/n]}
\newcommand{\ZZN}{\mathbb{Z}[1/N]}

\newcommand{\pair}[1]{\langle {#1} \rangle}
\newcommand{\wpair}[1]{\left\{{#1}\right\}}
\newcommand{\wh}{\widehat}
\newcommand{\wt}{\widetilde}

\newcommand\Spf{\mathrm{Spf}}

\newcommand{\lra}{{\longrightarrow}}

\newcommand{\matrixx}[4]
{\left( \begin{array}{cc}
  #1 &  #2  \\
  #3 &  #4  \\
 \end{array}\right)}        


\newcommand{\BA}{{\mathbb {A}}}
\newcommand{\BB}{{\mathbb {B}}}
\newcommand{\BC}{{\mathbb {C}}}
\newcommand{\BD}{{\mathbb {D}}}
\newcommand{\BE}{{\mathbb {E}}}
\newcommand{\BF}{{\mathbb {F}}}
\newcommand{\BG}{{\mathbb {G}}}
\newcommand{\BH}{{\mathbb {H}}}
\newcommand{\BI}{{\mathbb {I}}}
\newcommand{\BJ}{{\mathbb {J}}}
\newcommand{\BK}{{\mathbb {K}}}
\newcommand{\BL}{{\mathbb {L}}}
\newcommand{\BM}{{\mathbb {M}}}
\newcommand{\BN}{{\mathbb {N}}}
\newcommand{\BO}{{\mathbb {O}}}
\newcommand{\BP}{{\mathbb {P}}}
\newcommand{\BQ}{{\mathbb {Q}}}
\newcommand{\BR}{{\mathbb {R}}}
\newcommand{\BS}{{\mathbb {S}}}
\newcommand{\BT}{{\mathbb {T}}}
\newcommand{\BU}{{\mathbb {U}}}
\newcommand{\BV}{{\mathbb {V}}}
\newcommand{\BW}{{\mathbb {W}}}
\newcommand{\BX}{{\mathbb {X}}}
\newcommand{\BY}{{\mathbb {Y}}}
\newcommand{\BZ}{{\mathbb {Z}}}

\newcommand{\CA}{{\mathcal {A}}}
\newcommand{\CB}{{\mathcal {B}}}
\newcommand{\CD}{{\mathcal{D}}}
\newcommand{\CE}{{\mathcal {E}}}
\newcommand{\CF}{{\mathcal {F}}}
\newcommand{\CG}{{\mathcal {G}}}
\newcommand{\CH}{{\mathcal {H}}}
\newcommand{\CI}{{\mathcal {I}}}
\newcommand{\CJ}{{\mathcal {J}}}
\newcommand{\CK}{{\mathcal {K}}}
\newcommand{\CL}{{\mathcal {L}}}
\newcommand{\CM}{{\mathcal {M}}}
\newcommand{\CN}{{\mathcal {N}}}
\newcommand{\CO}{{\mathcal {O}}}
\newcommand{\CP}{{\mathcal {P}}}
\newcommand{\CQ}{{\mathcal {Q}}}
\newcommand{\CR }{{\mathcal {R}}}
\newcommand{\CS}{{\mathcal {S}}}
\newcommand{\CT}{{\mathcal {T}}}
\newcommand{\CU}{{\mathcal {U}}}
\newcommand{\CV}{{\mathcal {V}}}
\newcommand{\CW}{{\mathcal {W}}}
\newcommand{\CX}{{\mathcal {X}}}
\newcommand{\CY}{{\mathcal {Y}}}
\newcommand{\CZ}{{\mathcal {Z}}}

\newcommand{\ab}{{\mathrm{ab}}}
\newcommand{\Ad}{{\mathrm{Ad}}}
\newcommand{\an}{{\mathrm{an}}}
\newcommand{\Aut}{{\mathrm{Aut}}}

\newcommand{\Br}{{\mathrm{Br}}}
\newcommand{\bs}{\backslash}
\newcommand{\bbs}{\|\cdot\|}

\newcommand{\cod}{{\mathrm{cod}}}
\newcommand{\cont}{{\mathrm{cont}}}
\newcommand{\cl}{{\mathrm{cl}}}
\newcommand{\criso}{{\mathrm{criso}}}

\newcommand{\dR}{{\mathrm{dR}}}
\newcommand{\disc}{{\mathrm{disc}}}
\newcommand{\Div}{{\mathrm{Div}}}
\renewcommand{\div}{{\mathrm{div}}}

\newcommand{\Eis}{{\mathrm{Eis}}}
\newcommand{\End}{{\mathrm{End}}}

\newcommand{\Frob}{{\mathrm{Frob}}}

\newcommand{\Gal}{{\mathrm{Gal}}}
\newcommand{\GL}{{\mathrm{GL}}}
\newcommand{\GO}{{\mathrm{GO}}}
\newcommand{\GSO}{{\mathrm{GSO}}}
\newcommand{\GSp}{{\mathrm{GSp}}}
\newcommand{\GSpin}{{\mathrm{GSpin}}}
\newcommand{\GU}{{\mathrm{GU}}}
\newcommand{\BGU}{{\mathbb{GU}}}

\newcommand{\Hom}{{\mathrm{Hom}}}
\newcommand{\Hol}{{\mathrm{Hol}}}
\newcommand{\HC}{{\mathrm{HC}}}

\renewcommand{\Im}{{\mathrm{Im}}}
\newcommand{\Ind}{{\mathrm{Ind}}}
\newcommand{\inv}{{\mathrm{inv}}}
\newcommand{\Isom}{{\mathrm{Isom}}}

\newcommand{\Jac}{{\mathrm{Jac}}}
\newcommand{\JL}{{\mathrm{JL}}}

\newcommand{\Ker}{{\mathrm{Ker}}}
\newcommand{\KS}{{\mathrm{KS}}}

\newcommand{\Lie}{{\mathrm{Lie}}}

\newcommand{\new}{{\mathrm{new}}}
\newcommand{\NS}{{\mathrm{NS}}}

\newcommand{\ord}{{\mathrm{ord}}}
\newcommand{\ol}{\overline}

\newcommand{\rank}{{\mathrm{rank}}}

\newcommand{\PGL}{{\mathrm{PGL}}}
\newcommand{\PSL}{{\mathrm{PSL}}}
\newcommand{\Pic}{\mathrm{Pic}}
\newcommand{\Prep}{\mathrm{Prep}}
\newcommand{\Proj}{\mathrm{Proj}}

\newcommand{\Picc}{\mathcal{P}ic}

\renewcommand{\Re}{{\mathrm{Re}}}
\newcommand{\red}{{\mathrm{red}}}
\newcommand{\sm}{{\mathrm{sm}}}
\newcommand{\sing}{{\mathrm{sing}}}
\newcommand{\reg}{{\mathrm{reg}}}

\newcommand{\tor}{{\mathrm{tor}}}
\newcommand{\tr}{{\mathrm{tr}}}

\newcommand{\ur}{{\mathrm{ur}}}

\newcommand{\vol}{{\mathrm{vol}}}

\newcommand{\ds}{\displaystyle}

\tableofcontents

\section{Introduction}\label{1}
The main goal of this article is to prove GCD and abc inequalities on abelian varieties over function fields. We mainly work on the function fields of complex curves. The case of function fields of one variable over an arbitrary field $k$ of characteristic 0 then follows from the Lefschetz principle. At the beginning, we list the main theorems (as well as important corollaries). The relevant background and definitions (particularly several height functions) are presented in \S\ref{1.1} and \S\ref{1.2}.

The first is the GCD inequality, which compares the height function associated to a higher-codimensional subscheme (i.e. GCD height) with the usual height function.

\begin{theorem}[Geometric GCD for abelian varieties]\label{main}
Let $K$ be the function field of a smooth projective complex curve $B$. Let $A$ be an abelian variety over $K$. Let $L$ be an ample line bundle on $A$ and let $Y\subseteq A$ be a closed subscheme of $A$ of codimension $>1$. Fix height functions $h_{L}$ and $h_{Y}$ associated to $L$ and $Y$ respectively. Then for any $\varepsilon>0$, there is a constant $c>0$ such that the inequality
$$
h_{Y}(x)\leq\varepsilon h_{L}(x)+c
$$
holds for all rational points $x$ of $A$ outside a proper Zariski closed subset.
\end{theorem}
As predicted by Vojta's conjecture, the exceptional set in the theorem can in fact be given a geometric description, and is therefore stable under finite field extensions $K'/K$; see \S\ref{4.6} for more discussions.

Combining with the Mordell--Lang theorem, we can derive the following corollary.
\begin{corollary}\label{main'}
Let $K$ be the function field of a smooth projective complex curve $B$. Let $A$ be an abelian variety over $K$. Let $L$ be an ample line bundle on $A$ and let $Y\subseteq A$ be a closed subscheme of $A$ of codimension $>1$. Let $\Gamma\subseteq A(\ol{K})$ be a finitely generated subgroup. Fix the height functions $h_{L}$ and $h_{Y}$ associated to $L$ and $Y$ respectively. Then for any $\varepsilon>0$, there is a constant $c>0$ and a finite union $Z$ of translates of proper abelian subvarieties of $A$ such that
$$
h_{Y}(x)\leq\varepsilon h_{L}(x)+c
$$
holds for all $x\in \Gamma\backslash Z\subseteq A(\ol{K})$.
\end{corollary}
Note that $\Gamma$ is contained in $A(K')$ for some finite extension $K'/K$. Therefore (after base change everything to $K'$) we can indeed apply Theorem \ref{main} to this setting. This is a geometric and compact analogue of a theorem of Levin; see \S\ref{1.2} for the statement and other GCD-type results for tori or abelian varieties. \medskip\medskip

\noindent Now we turn to an abc inequality, which compares the truncated counting function with the usual height function.

\begin{theorem}[Geometric ABC for rational points on abelian varieties]\label{main1}
Let $K$ be the function field of a smooth projective complex curve $B$. Let $A$ be an abelian variety over $K$. Let $D$ be an effective \textbf{reduced} Cartier divisor on $A$. Fix a finite subset $S\subset B(\CC)$, a height function $h_{D}$ associated to $D$, and a truncated counting function $N^{(1)}_{D,S}\colon (A\backslash{\rm Supp}\,D)(K)\to\BN$. Then for any $\varepsilon>0$, there is a constant $c>0$ such that the inequality
$$
N^{(1)}_{D,S}(x)\geq (1-\varepsilon)h_{D}(x)-c
$$
holds for all rational points $x$ of $A$ outside a proper Zariski closed subset.
\end{theorem}

This proves Vojta's (geometric) abc conjecture for rational points on abelian varieties over $\CC(B)$; see Conjecture \ref{vojta} below for a review of the conjecture. It also implies the geometric Lang--Vojta conjecture for affine open subvarieties of abelian varieties. The toric analogue of Theorem \ref{main1} was established by Guo, Nguyen, Sun and Wang \cite{GNSW}. See also \S\ref{1.2} for more related works on these conjectures for tori or abelian varieties.

Moreover, by using the same method, we can also prove the following conjecture of Buium (see the second paragraph after \cite[Theorem 1]{zbMATH00703972}, or \cite[Conjecture 2.7]{zbMATH01270843}).

\begin{proposition}[Buium's conjecture]\label{buium}
Let $A$ be an abelian variety over $K=\CC(B)$ , let $D$ be an effective divisor
on $A$. Fix a height function $h_{D}$ associated to $D$, a finite subset $S\subseteq B(\CC)$ and a truncated counting function $N^{(1)}_{D,S}\colon (A\backslash{\rm Supp}\,D)(K)\to\BN$. Then there are constants $c_{1},c_{2}>0$ such that $N^{(1)}_{D,S}(x)\geq c_{1}h_{D}(x)-c_{2}$
for all $x\in (A\backslash{\rm Supp}\,D)(K)$.
\end{proposition}
Note that this inequality for general $x\in A(K)$ (outside a proper Zariski closed subset) is just a direct corollary of Theorem \ref{main1}. In fact, take an integer $M$ such that $D\leq M\cdot D_{\rm red}$. Then we can take $c_{1}$ to be any constant $0<c_{1}<\frac{1}{M}$, and then take appropriate $c_{2}>0$ and an exceptional subset. To prove Buium's conjecture for all points $x\in (A\backslash{\rm Supp}\,D)(K)$, we then need to make more effort on the exceptional locus.

Buium himself proved this conjecture when $A$ has trivial $K/\CC$-trace in \cite[Theorem 1]{zbMATH00703972}. Noguchi and Winkelmann \cite[Main Theorem]{NW} proved the constant case. The general case remained open previously.

Both Theorem \ref{main1} and Proposition \ref{buium} can be generalized to the case when $S$ is (Euclidean) closed of measure zero (with respect to a K\"{a}hler form on $B(\CC)$). See Remark \ref{jiaqiang} in \S\ref{4.5}.

\subsection{Classical GCD and ABC problems}\label{1.1}
Both GCD and abc inequalities originate from inequalities in elementary number theory. We briefly review them.

In 2003, Bugeaud, Corvaja and Zannier \cite{zbMATH01922486} began a new line of results related to the inequalities for greatest common divisors:
\begin{theorem}[{\cite[Theorem 1]{zbMATH01922486}}]\label{bcz}
Let $a,b\in\BZ$ be multiplicatively independent integers. Then for every $\varepsilon>0$,
\begin{align*}
\log\gcd(a^{n}-1,b^{n}-1)\leq\varepsilon n
\end{align*}
for all but finitely many positive integers $n$.
\end{theorem}
Theorem \ref{bcz} was then generalized by Corvaja and Zannier \cite{CZ03} and, independently, Hern\'{a}ndez and Luca \cite{HL}, allowing $\gcd(a^{n}-1,b^{n}-1)$ to be replaced by $\gcd(u-1,v-1)$ for multiplicatively independent $S$-unit integers $u$ and $v$. Corvaja and Zannier \cite{CZ05} further generalized the theorem to the setting where $u$ and $v$ lie in a fixed finitely generated subgroup of $\ol{\QQ}^{*}$, and they replaced $u-1$ and $v-1$ by more general pairs of polynomials in $u$ and $v$. Building on these works, Levin \cite{zbMATH07023767} extended the result to the multivariate setting, yielding the following significant theorem.

\begin{theorem}[{\cite[Theorem 1.1]{zbMATH07023767}}]\label{levin-1}
Let $n$ be a positive integer, $\Gamma\subseteq\BG_{m}^{n}(\ol{\QQ})$ a finitely generated group, and $f(x_{1},\dots,x_{n}),g(x_{1},\dots,x_{n})\in\ol{\QQ}[x_{1},\dots,x_{n}]$ nonconstant coprime polynomials such that not both of them vanish at $(0,0,\dots,0)$. Let $h(\alpha)$ denote the (absolute logarithmic) height of an algebraic number $\alpha$. For all $\varepsilon>0$, there exists a finite union $Z$ of translates of proper algebraic subgroups of $\BG_{m}^{n}$ such that
\begin{align*}
\log\gcd(f(u_{1},\dots,u_{n}),g(u_{1},\dots,u_{n}))<\varepsilon\max\{h(u_{1}),\dots,h(u_{n})\}
\end{align*}
for all $(u_{1},\dots,u_{n})\in\Gamma\backslash Z$.
\end{theorem}

Here for two algebraic numbers $\alpha,\beta\in\ol{\QQ}$, not both zero, define
$$
\log\gcd(\alpha,\beta)\colonequals -\sum\limits_{v\in M_{K}}\log^{-}\max\{|\alpha|_{v},|\beta|_{v}\},
$$
where $K$ is a number field containing both $\alpha$ and $\beta$, $M_{K}$ is the set of places of $K$, and $\log^{-}x\colonequals \min\{0,\log x\}$.

\medskip\medskip

\noindent The famous abc conjecture was proposed by Oesterl\'{e} and Masser in 1985, which remains open. The statement is as follows.

\begin{conjecture}[abc conjecture]\label{abc}
For every $\varepsilon>0$, there exist only finitely many triples $(a, b, c)$ of coprime positive integers, with $a + b = c$, such that
$$
\log c>(1+\varepsilon)\log({\rm rad}(abc)).
$$
Here ${\rm rad}(abc)$ is the product of all different prime divisors of $abc$.
\end{conjecture}

In the next subsection, we will reformulate these classical inequalities using height functions, which can be carried out on projective varieties over either number fields or function fields.
\subsection{Height functions and reformulation of the problems}\label{1.2}
In this subsection we recall the definition of height functions (as well as some related functions). For simplicity, we focus on the case of function fields. All of these can be defined similarly over number fields.\medskip

\noindent\textbf{Convention.} In this article, a variety is an integral separated scheme of finite type over a field. A curve is a 1-dimensional variety. We write tensor products of line bundles additively. \medskip

Let $K$ be the function field of a smooth projective curve $B$ over a field $k$. Denote by $|B|$ the set of all closed points on $B$. Let $X$ be a projective variety over $K$, and let $D$ be an effective Cartier divisor on $X$.

Fix any integral model $\CX$ of $X$ over $B$, i.e. a projective flat morphism $\pi\colon\CX\to B$ whose generic fibre is isomorphic to $X$. Let $\CD$ be an effective Cartier divisor on $\CX$ extending $D$. For a rational point $x\in X(K)$, by taking Zariski closure of $x$ in $\CX$, we can view $x$ as a section $x\colon B\to\CX$ of $\pi$. Then for any $x\in(X\backslash{\rm Supp}\,D)(K)$ and any closed point $b\in |B|$, we define
\begin{itemize}
\item Local height function (or local Weil function):\quad $\lambda_{D,b}(x)\colonequals {\rm ord}_{b}x^{*}\CD\cdot[k(b)\colon k]$.

\end{itemize}
For a subset $S\subseteq|B|$, we can then define the following functions:
\begin{itemize}
\item Height function:\quad $h_{D}(x)\colonequals \sum\limits_{b\in |B|}\lambda_{D,b}(x)=\sum\limits_{b\in |B|}{\rm ord}_{b}x^{*}\CD\cdot[k(b)\colon k]$.

\item Proximity function:\quad $m_{D,S}(x)\colonequals \sum\limits_{b\in S}\lambda_{D,b}(x)=\sum\limits_{b\in S}{\rm ord}_{b}x^{*}\CD\cdot[k(b)\colon k]$.

\item Counting function:\quad $N_{D,S}(x)\colonequals \sum\limits_{b\in |B|\backslash S}\lambda_{D,b}(x)=\sum\limits_{b\in |B|\backslash S}{\rm ord}_{b}x^{*}\CD\cdot[k(b)\colon k]$.

\item Truncated counting function: \quad $N^{(m)}_{D,S}(x)\colonequals \sum\limits_{b\in |B|\backslash S}\min\{m,{\rm ord}_{b}x^{*}\CD\}\cdot[k(b)\colon k]$,\\
    where $m\geq1$ is an integer.
\end{itemize}
More generally, for a point $x\in (X\backslash{\rm Supp}\,D)(\ol{K})$, let $K(x)$ be the residue field of $x$ and let $f\colon B'\to B$ be the finite map between smooth projective curves over $k$ which corresponds to the finite extension $K(x)/K$. Then $x$ corresponds to a morphism $B'\to\CX$, which we still denote by $x$. For a closed point $b\in|B|$, define
\begin{align*}
\lambda_{D,b}(x)\colonequals \frac{1}{\deg(x)}\sum\limits_{b'\in f^{-1}(b)}{\rm ord}_{b'}x^{*}\CD\cdot[k(b')\colon k].
\end{align*}
Then we can define the four functions above for $x$ in the same way.

All the functions above depend on the choice of $\CX$ and $\CD$. However, it can be shown that they are unique up to a bounded function on $(X\backslash{\rm Supp}\,D)(\ol{K})$. Moreover, it is obvious from the definition that $h_{D}=m_{D,S}+N_{D,S}+O(1)$. This is called the \emph{first main theorem} in Vojta's dictionary. Finally, note that for the fixed model $(\CX,\CD)$ of $(X,D)$, the height function $h_{D}$ is just the intersection number
$$
\deg_{B}x^{*}\CD=\deg_{B}x^{*}\CO_{\CX}(\CD).
$$
Thus up to a bounded function, $h_{D}$ depends only on the line bundle $L=\CO_{X}(D)$ associated to $D$. In particular, it can be extended to a function on $X(\ol{K})$.

\begin{remark}\label{S}
As a definition, the subset $S\subseteq B(\CC)$ can be arbitrary, and the functions $m$ and $N$ actually play symmetric roles (by replacing $S$ by $B(\CC)\backslash S$). However, in practice, we always take $S$ to be ``small'' and therefore $m$ seems much smaller than $N$.

In fact, $S$ is typically taken to be finite (in the number field case, it is further required to contain all archimedean places). In our applications, we may take $S$ to be infinite, but it is still ``small'' (in the sense that its measure (with respect to a K\"{a}hler form on $B(\CC)$) is very small compared to $B(\CC)$). See, for example, Remark \ref{jiaqiang}.
\end{remark}

The abc conjecture can then be formulated in the following form, which is due to Vojta.

\begin{conjecture}[{\cite[Conjecture 2.3]{zbMATH01270843}}]\label{vojta}
Let $K$ be a number field or a function field of characteristic $0$. Let $X$ be a projective variety over $K$. Let $D$ be a normal crossing divisor on $X$. Let $K_{X}$ be the canonical divisor on $X$ and let $A$ be a big Cartier divisor on $X$. Let $S$ be a finite set of places of $K$ containing all archimedean places. Then for any $\varepsilon>0$, there is a proper closed subset $Z$ of $X$ such that
$$
N^{(1)}_{D,S}(x)\geq h_{K_{X}+D}(x)-\varepsilon h_{A}(x)-O(1)
$$
for all $x\in (X\backslash Z)(K)$.
\end{conjecture}
Conjecture \ref{abc} can be deduced from Conjecture \ref{vojta} by taking $K=\QQ$, $S=\{\infty\}$, $X=\BP^{1}$ and $D=[0]+[1]+[\infty]$ on $\BP^{1}$. See the discussion in \cite{zbMATH01270843} after Conjecture 2.4.

Vojta actually formulated the conjecture for all algebraic points with fixed degree (with an additional term of discriminant on the left hand side); we only consider the rational points in this article. This conjecture can be viewed as a (conjectural) second main theorem. Our main Theorem \ref{main1} confirms this conjecture for abelian varieties over function fields (with $D$ only required to be reduced). Note that this conjecture also has a non-truncated version, in which one uses $N_{D,S}$ instead of $N^{(1)}_{D,S}$; this is the classical version of Vojta's conjecture; see \cite[Conjecture 2.1]{zbMATH01270843}. We will first prove this non-truncated version for abelian varieties over function fields (with $D$ only required to be effective), and then pass to the truncated version. These two versions correspond respectively to Proposition \ref{V}(1), (2) below.

Note that the non-truncated version for abelian varieties over number fields was already proved by Faltings \cite{Fal}. Subsequently, Buium--Voloch \cite{BV} proved the analogue over function fields, and Vojta \cite{Voj99} generalized it to semi-abelian varieties. Nevertheless, our proof of the non-truncated version for abelian varieties over function fields still exhibits a new phenomenon: unlike previous Diophantine approximation arguments, it uses complex-analytic methods to obtain finer distribution. In particular, we can weaken the assumption ``$S\subseteq B(\CC)$ is finite'' to ``$S$ is a Euclidean closed subset of measure 0''; see Remark \ref{jiaqiang}.

Conjecture \ref{vojta} also implies the (geometric) Lang--Vojta conjecture, where the geometric version asserts that if $K_{X}+D$ is of log general type, then there is a proper Zariski closed subset $Z$ of $X$ and a constant $c>0$ such that for any finite subset $S$ of $B(\CC)$ and any $(D,S)$-integral point $x$ outside $Z$, the inequality
$$
h_{D}(x)\leq c\max\{1,2g(B)-2+\sharp S\}
$$
holds for $x$. Here $g(B)$ denotes the genus of $B$. In fact, fix any finite subset $S_{0}$ of $B(\CC)$. Then $x$ is a $(D,S)$-integral point implies that $N^{(1)}_{D,S_{0}}(x)\leq \sharp (S
\backslash S_{0})\leq\sharp S$. Applying Conjecture \ref{vojta} to $S_{0}$, we see that when $g(B)\geq1$, then we can just take $c=1+\varepsilon$ for any $\varepsilon>0$. When $g(B)=0$, then we have $\sharp S\leq 3(2g(B)-2+\sharp S)$ for $\sharp S\geq 3$. For those $\sharp S\leq 2$, Conjecture \ref{vojta} also implies that $h_{D}(x)$ is bounded by some constant $c'>0$ outside a proper closed subset. Then take $c=\max\{3,c'\}$.

Similarly, Buium's conjecture (Proposition \ref{buium}) also implies the geometric Lang--Vojta conjecture when $X$ is an abelian variety. The previous known cases of Buium's conjecture are listed before \S\ref{1.1}. One can also consult \cite{Ph,Dol} for more discussions. None of these results yield the optimal constant $c>0$ above, and consequently fail to establish the abc conjecture (Theorem \ref{main1}), where we need $c$ to be $1+\varepsilon$ for any $\varepsilon>0$. There are also several results on Vojta's abc conjecture and Lang--Vojta conjecture for tori over function fields; see \cite{Voj00,CZ08,CZ13,Tur,CT,GGW}. In particular, \cite[Theorem 1]{GNSW} is a full analogue of our Theorem \ref{main1}.

Besides, (the algebraic points version of) Conjecture \ref{vojta} was proved for curves over function fields by McQuillan \cite{MQ} and Yamanoi \cite[Theorem 5]{Yam04b} independently; see also \cite{Gas} for a survey.
\medskip\medskip

\noindent\textbf{GCD height function}\medskip

\noindent All the functions above are associated to a Cartier divisor $D$ on $X$. In fact, they can be generalized to the functions associated to a closed subscheme $Y$ of $X$ in the same way. This was first done by Silverman \cite{Sil}. Namely, let $\CY$ be a closed subscheme of $\CX$ extending $Y$ (for example, we can take $\CY$ to be the Zariski closure of $Y$ in $\CX$). Then for a point $x\in (X\backslash Y)(K)$, define the local height function (or local Weil function)
\begin{align*}
\lambda_{Y,b}(x)\colonequals {\rm ord}_{b}x^{*}\CY\cdot[k(b)\colon k],\quad \forall\, b\in|B|
\end{align*}
and the height function
\begin{align*}
h_{Y}(x)\colonequals \sum\limits_{b\in |B|}\lambda_{Y,b}(x)=\sum\limits_{b\in |B|}{\rm ord}_{b}x^{*}\CY\cdot[k(b)\colon k].
\end{align*}
Here ${\rm ord}_{b}x^{*}\CY$ is defined as follows: Write $\CY=\CD_{1}\cap\dots\cap \CD_{l}$ as an intersection of effective Cartier divisors $\CD_{i}$. Then let ${\rm ord}_{b}x^{*}\CY\colonequals\min_{1\leq i\leq l}{\rm ord}_{b}x^{*}\CD_{i}$, which is finite since $x(B)\nsubseteq \CY$. This definition is independent of the choice of $\CD_{i}$. Moreover, up to a bounded function, the function $h_{Y}$ is independent of the choice of $\CX$ and $\CY$.

The other three functions associated to $Y$ are defined similarly. These functions can also be defined for all algebraic points on $X\backslash Y$ similarly as before. We omit the details here. When the codimension of $Y$ in $X$ is at least 2, the function $h_{Y}$ is called the \emph{GCD height function}. By using this function (in the number field case), we can reformulate Levin's theorem as follows.
\begin{theorem}[{\cite[Theorem 1.16]{zbMATH07023767}}]\label{levin-2}
Let $Y$ be a closed subscheme of $\BP^{n}$ of codimension at least $2$ defined over a number field. Suppose that $Y$ does not contain any of the points $P_{0}=[1:0:\dots0],\dots,P_{n}=[0:0:\dots:0:1]$. Let $L=\CO_{\BP^{n}}(1)$. Let $\Gamma\subseteq \BG_{m}^{n}(\ol{\QQ})$ be a finitely generated subgroup. Identify $\BG_{m}^{n}$ with $\BP^{n}\backslash\cup_{i=0}^{n}H_{i}$, where $H_{0},H_{1},\dots,H_{n}$ are the $(n+1)$ coordinate hyperplanes. Then for all $\varepsilon>0$, there exists a finite union $Z$ of translates of proper algebraic subgroups of $\BG_{m}^{n}$ such that
$$
h_{Y}(P)\leq\varepsilon h_{L}(P)
$$
for all $P\in\Gamma\backslash Z\subseteq\BP^{n}(\ol{\QQ})$.
\end{theorem}

Theorem \ref{levin-2} is slightly weaker than Theorem \ref{levin-1}, but still powerful; see \cite{zbMATH07023767} for a detailed explanation of their relationship. This theorem also has several generalizations; see \cite{WY,HL22,HLX,WX}. See also \cite{Ya25,Ya} for more GCD-type inequalities.

There are some function field analogues of Levin's theorem. Ailon--Rudnick \cite{AR} proved that for two multiplicatively independent polynomials $f,g\in\CC[t]$, there is a polynomial $h\in\CC[t]$ such that $\gcd(f^{n}-1,g^{n}-1)$ divides $h$ for all $n\geq1$. This is analogous to Theorem \ref{bcz}, with a stronger conclusion. Corvaja--Zannier \cite[Corollary 2.3]{CZ08} proved an explicit GCD inequality on $\BG_{m}^{2}$ over function fields, which implies that the GCD height is bounded by a quantity of order at most the $\frac{2}{3}$-th power of the usual height, which is also stronger than the linear order $\varepsilon h$ we consider.

Our Corollary \ref{main'} can be viewed as a compact and geometric analogue of Theorem \ref{levin-2}. Let us mention some previous progress on this case. Let $A$ be an abelian variety of dimension at least 2 over a field $K$, which is either a number field or a function field. Silverman \cite{Sil04,Sil05} proposed some conjectures concerning the GCD height function associated to the point $0\in A$ on a rank-one Zariski dense subgroup $\BZ\cdot P$ of $A(K)$ and proved the case where $A$ is the product of two isotrivial elliptic curves over a function field. The inequality in part (a) of \cite[Conjecture 1]{Sil04} has the same form as our Theorem \ref{main}. In fact, Silverman proved in \cite[Proposition 9]{Sil05} the arithmetic version of our Theorem \ref{main} under Vojta's conjecture for blow up of $A$ at $0$, and then deduced part (a) of \cite[Conjecture 1]{Sil04} from this arithmetic version. Other conclusions of these conjectures are stronger and can not be deduced from the present article. Silverman's result over function fields was later generalized by Ghioca--Hsia--Tucker \cite[Proposition 4.3]{GHT} to products of elliptic curves and by Barroero--Capuano--Turchet \cite[Theorem 1.3]{BCT} to split semi-abelian varieties.

Recently, Corvaja--Zannier--Zucconi \cite[Theorem 1.3]{CZZ} also proved the GCD inequality for constant abelian surfaces that are either of Picard rank one or products of elliptic surfaces (with some mild extra assumptions) by using algebraic geometry of surfaces. They also give a refined bound of GCD height for all algebraic points with at most one exception.

By contrast to the above results, our method is complex analytic in nature, and is based on the notion of Xie--Yuan height\footnote{This notion was introduced in \cite{XY1}, where it was called \emph{partial height}. In this article, we refer to it as the Xie--Yuan height, after the two authors of that paper.} below. Our Theorem \ref{main} requires neither additional assumptions on the abelian variety $A$ nor that the closed subscheme $Y$ be a single point. We consider general rational points on $A$ without restricting to rank-one subgroups, but we do not treat general algebraic points.

\medskip\medskip
\noindent\textbf{Integral expression of the height and the Xie--Yuan height}\medskip

\noindent Here we recall the notion of Xie--Yuan height defined in \cite{XY1}, which plays an important role in the proof. We assume $K=\CC(B)$ in the following.

The key observation is that the height function above can also be expressed by integration. Explicitly, we have
\begin{align*}
h_{D}(x)\colonequals \deg_{B}x^{*}\CD=\int_{B(\CC)}x^{*}\omega_{\CD}.
\end{align*}
Here $\omega_{\CD}$ is a real $(1,1)$-form representing the first Chern class $c_{1}(\CD)$ of $\CD$. This integral expression can be viewed as a continuous analogue of the summation expression above. Using this, we can naturally define the Xie--Yuan height as follows, which turns out to be a continuous analogue of the proximity function. Fix a real $(1,1)$-form $\omega_{\CD}$ representing the first Chern class $c_{1}(\CD)$ of $\CD$. For any measurable subset $S$ of $B(\CC)$, define
\begin{align*}
h_{D,S}(x)\colonequals \int_{S}x^{*}\omega_{\CD}.
\end{align*}
We sometimes also denote this Xie--Yuan height by $h_{\omega_{\CD},S}$ or $h_{L,S}$, where $L=\CO_{X}(D)$. Note that in \cite[\S 2.1]{XY1}, they denote it by $h_{(S,\omega_{\CD})}$; here we instead use the above notation to be consistent with the form of the proximity and counting functions.

The continuity is one of the advantages of Xie--Yuan height. For example, the Xie--Yuan height with respect to a subset $S$ of measure 0 is always zero, while the proximity function with respect to a subset $S$ with only one point may be very large. We will use the estimate of Xie--Yuan height on abelian varieties (see Proposition \ref{jiben-}) to estimate the proximity function. These two functions can also be viewed as an analogue of the characteristic function and the counting function in Nevanlinna theory (see \S\ref{2.3}).
\subsection{Ideas of the proof}\label{1.3}
Now we introduce the idea to prove the main theorems. We can reduce the problems to the following proposition. Recall that an infinite sequence $\{x_{n}\}_{n\geq1}$ of rational points on a variety $X$ is called \emph{generic} if every infinite subsequence of it is Zariski dense in $X$.
\begin{proposition}\label{V}
Let $K$ be the function field of a smooth projective complex curve $B$. Let $A$ be an abelian variety over $K$, and let $T$ be a projective variety over $\CC$. Let $D$ be an effective Cartier divisor on $A\times T_{K}$ which is relatively ample over $T_{K}$. Let $\{(s_{n},t_{n})\}_{n\geq1}\subseteq A(K)\times T_{K}(K)$ be a generic sequence of rational points on $A\times T_{K}$ such that the height of $s_{n}$ (with respect to an ample line bundle on $A$) tends to infinity as $n\to\infty$ and $t_{n}\in{\rm Im}(T(\CC)\to T_{K}(K))$ for all $n$. Let $\varepsilon>0$ be a constant. Then

\begin{enumerate}[label=(\arabic*),leftmargin=*]
\item For any subset $U\subseteq B(\CC)$ whose boundary $\partial U$ has measure $0$, the inequality
\begin{align*}
(1-\varepsilon)\hat{h}_{D^{\rm sym},U}((s_{n},t_{n}))\leq m_{D,U}((s_{n},t_{n}))\leq \hat{h}_{D^{\rm sym},U}((s_{n},t_{n}))+\varepsilon \hat{h}_{D^{\rm sym}}((s_{n},t_{n}))
\end{align*}
holds for $n$ large enough. Here $D^{\rm sym}=\frac{1}{2}(D+([-1]_{A}\times{\rm id}_{T_{K}})^{*}D)$ is the symmetrization of $D$ over $T_{K}$, and $\hat{h}_{D^{\rm sym}},\hat{h}_{D^{\rm sym},U}$ are the canonical (Xie--Yuan) height functions induced by the Betti form associated to $D^{\rm sym}$, see \S\ref{betti} for more details. When $U$ does not have measure zero, we can strengthen the right-hand side to $(1+\varepsilon)\hat{h}_{D^{\rm sym},U}((s_{n},t_{n}))$.

Moreover, when $U$ is (Euclidean) open (resp. closed), then the left-hand (resp. right-hand) inequality above also holds for $n$ large enough.

\item If moreover $D$ is reduced, then we can replace the term $m_{D,U}((s_{n},t_{n}))=N_{D,B(\CC)\backslash U}((s_{n},t_{n}))$ above by the truncated counting function $N^{(1)}_{D,B(\CC)\backslash U}((s_{n},t_{n}))$.

\item Let $Y$ be a closed subscheme of $A\times T_{K}$ of codimension $>1$. Then
\begin{align*}
h_{Y}((s_{n},t_{n}))\leq\varepsilon h_{D}((s_{n},t_{n}))
\end{align*}
holds for $n$ large enough.
\end{enumerate}
\end{proposition}
Here the role of $T$ is the same as in \cite[Proposition 1.4]{arXiv:2511.17010}: it serves as the ``constant part''. As explained in \cite[Remark 1.5(1)]{arXiv:2511.17010}, even if we are only interested in rational points on abelian varieties (i.e. the case where $T$ is trivial), when the $K/\CC$-trace of $A$ is nontrivial, it is sometimes necessary to reduce the problem to the case where $T$ is nontrivial. See the beginning of \S\ref{4.1} for the detailed reduction process.

The symmetrization $D^{\rm sym}$ here is performed because the canonical Xie--Yuan height was only defined for symmetric line bundles previously via Betti forms. It can actually be generalized to all line bundles on abelian varieties (see Remark \ref{adelic}), and in this way we can formulate Proposition \ref{V} in a more natural and unified way and avoid symmetrization. However, since this approach requires the development of a more complete theory, we avoid using it here for simplicity. Instead, we reduce everything to the symmetric setting. Since $D-D^{\rm sym}$ is anti-symmetric over $T_{K}$, the symmetrization process has no effect on the asymptotic growth of the height functions; see Proposition \ref{bijiao-} for an explicit estimate. See also Proposition \ref{bijiao} for a similar estimate.

We also note that in (1) and (2), for every ``good'' subset $U$ of $B(\CC)$ we have the corresponding estimates comparing the asymptotic behaviour of discrete sums of local heights with that of the continuous Xie--Yuan heights. In particular, this enables us to obtain an equidistribution of the Dirac measures associated to the finite subsets $S_{n}\colonequals \{b\in B(\CC)\big|\, \lambda_{D,b}((s_{n},t_{n}))>0\}$ of $B(\CC)$; see Remark \ref{jianjin}.

In particular, if we take $U$ to be the complement of a Euclidean closed subset $S\subseteq B(\CC)$ of Lebesgue measure 0, then the canonical Xie--Yuan height $\hat{h}_{D^{\rm sym},U}$ equals the usual canonical height $\hat{h}_{D^{\rm sym}}$, which has the same asymptotic growth as any Weil height function $h_{D}$ associated to $D$ (see Proposition \ref{bijiao-}). Moreover, the function $m_{D,U}$ (resp. $N^{(1)}_{D,B(\CC)\backslash U}$) equals $N_{D,S}$ (resp. $N^{(1)}_{D,S}$). Thus we see that
$$
N_{D,S}((s_{n},t_{n}))\geq (1-\varepsilon)h_{D}((s_{n},t_{n}))
$$
for $n$ large enough. And if $D$ is reduced, we can replace the left hand side by the truncated counting function $N^{(1)}_{D,S}((s_{n},t_{n}))$.

Note that the condition ``$S$ is Euclidean closed'' is necessary. In fact, let $S_{n}\colonequals \{b\in B(\CC)\big|\, \lambda_{D,b}((s_{n},t_{n}))>0\}$ as above. Since $\{(s_{n},t_{n})\}_{n\geq1}$ is generic, the set $S_{n}$ is finite for $n$ large enough. If we take $S$ to be the union of all $S_{n}\,(n\gg1)$, then $S$ is countable and thus has measure 0. But in this case we have $N_{D,S}((s_{n},t_{n}))=0$ for all $n\gg1$, which contradicts the inequality above.

The deduction of Theorems \ref{main} and \ref{main1} from Proposition \ref{V} will be presented in \S\ref{4.5}. Now we focus on the proof of Proposition \ref{V}.

We first fix some notations. Let $\CA$ be an integral model of $A$ over $B$, then $\CA\times_{\CC}T$ is an integral model of $A\times T_{K}$ over $B$. Let $\CD$ be a Cartier divisor on $\CA\times_{\CC}T$ extending $D$ and let $\CY$ be a closed subscheme of $\CA\times_{\CC}T$ extending $Y$. View $(s_{n},t_{n})\in A(K)\times T(\CC)$ as a section $B\to\CA\times_{\CC}T$.

We need to compare the intersection numbers $(s_{n},t_{n})(B)\cap\CD$ (or $(s_{n},t_{n})(B)\cap\CY$) with the Xie--Yuan height of $(s_{n},t_{n})$ (which is an integration of Betti form associated to $D$).

\medskip

\noindent\textbf{Step 1: Microscopic local estimates.} We shall first prove that for all points $b\in B(\CC)$ outside a subset $S$ of measure 0 (which we call ``bad locus'') and all $n$ sufficiently large (depending on $b$), we can choose an open disc $\BD_{n,b}$ centred at $b$ whose radius tends to 0 as $n\to\infty$, on which the desired inequalities hold. This step is a combination of the reparameterization technique of Xie–Yuan (see Lemma \ref{linear} for a review) and the estimates for a single linear entire curve established in \S\ref{2.3}; the latter are derived from Yamanoi’s second main theorem, the only additional point being the passage from weighted to unweighted counting functions.

At this stage the estimates are only ``microscopic'' because the radius of the disc shrinks with $n$; they are not yet valid on any fixed neighbourhood of $b$.
\medskip\medskip

\noindent\textbf{Step 2: Local estimates.} The microscopic estimates alone are insufficient: to prove Proposition \ref{V} one needs inequalities that hold on a fixed subset of $B(\CC)$ (rather than depending on $n$). A natural strategy is to try to fill such a fixed subset, for each large $n$, by the tiny discs $\BD_{n,b}$ supplied by the microscopic step, and then piece together the estimates on those discs. Carrying this out meets two principal difficulties: uniformity in the parameter $b$, and the fact that a family of tiny discs can never exhaust a given open set without both overlaps and gaps.

The first difficulty is uniformity. In the microscopic step, both the required size of $n$ and the radius of $\BD_{n,b}$ may depend on the point $b$. We need a uniform $n$ for all ``good'' points $b$ in order to glue the microscopic local estimates together. To do this, we make some uniform estimates for linear entire curves. In fact, the linear entire curves that appear all share essentially the same direction, but their starting points vary. One therefore needs estimates that are uniform with respect to the starting point; these are obtained in \S\ref{3} and summarized in Proposition \ref{conclusion}, which is based on the estimates in \S\ref{2.3} for a single linear entire curve. The main effect of this uniformity is that for all sufficiently large $n$, the microscopic estimates hold simultaneously for every $b$ outside a new exceptional set (which we still call the ``bad locus''). In other words, one finds a common $n$ that works for all the relevant points $b$. As a by-product, the radii of the discs can also be taken to be independent of $b$ for each such $n$; this additional uniformity will be essential later when we handle overlaps via an averaging argument. The new exceptional set may be slightly larger than the original one and may depend on $n$, but its measure can still be controlled to be arbitrarily small; this is precisely the content of Lemma \ref{juli} (and of its refinement Lemma \ref{juli-} for the truncated case).

The concrete realizations of the above uniform estimates are presented in (\ref{key--5}), (\ref{key-5}, \ref{key-6}) and (\ref{abc-5}) respectively. Namely, we show that for each ``good point'' $b\in B(\CC)$, there is an open disc $\BD_{b}$ containing $b$ such that for each $n$ large enough and each ``good point'' $b'\in\BD_{b}$, there is an open disc $\BD_{n,b'}$ containing $b'$ inside $\BD_{b}$ on which the expected inequalities hold. Moreover, the radius of $\BD_{n,b'}$ is independent of $b'$.

With the uniformity issue resolved, the problem reduces to understanding how to (almost) fill a fixed neighbourhood with discs of a common radius while controlling overlaps and gaps.\medskip

\noindent\emph{Overlaps}. For the overlaps, the Besicovitch covering theorem allows one to extract a sub-collection of any cover (by open discs) with bounded multiplicity. For the GCD inequality (Proposition \ref{V}(3)) a bounded multiplicity only changes the constant $\varepsilon$ by a multiplicative factor, which is harmless. For the (truncated) counting functions appearing in (1) and (2), however, the factor $(1-\varepsilon)$ would be multiplied by the multiplicity, so a mere bound on the multiplicity is not sufficient. Instead, one needs an asymptotically uniform arrangement: almost every point of the neighbourhood should be covered with essentially the same multiplicity. Because the radii are already independent of $b$ for fixed $n$, this can be achieved by an averaging argument; see Lemma \ref{refine--}. We emphasize that a disjoint placement of equal-sized discs in a larger disc can cover only a fixed proportion of the area regardless of the radius; therefore a uniform covering inevitably has large overlaps, but with almost constant multiplicity.\medskip

\noindent\emph{Gaps}. Gaps in the cover arise from the presence of the bad locus and cannot be entirely avoided. The canonical Xie--Yuan height is an integration term which enjoys good continuity, so a gap of small measure affects it only negligibly (this is justified by the estimates for canonical Xie--Yuan heights in \S\ref{4.2}). The sum of local heights, however, can be seriously affected by even a single omitted point. When one needs a lower bound for the sum (the left-hand inequalities in (1) and (2)), a small gap is harmless because a partial sum already provides the required estimate. Upper bounds for the sum (the right-hand inequalities in (1) and (2), and the GCD estimate (3)) require more work. The right-hand inequalities in (1) and (2) can be deduced directly from the left-hand ones applied to the complement.

For the GCD upper bound (3), we bound the sum over the gaps by the corresponding sum for an ample effective divisor $D$ containing $Y$; the latter sum equals the total height associated to $D$ minus the sum over the good locus (which we called the ``good $D$-sum''), whose lower bound is provided by (1). However, the gap set depends on $n$, whereas the upper bound in (1) for the sum over a fixed subset is guaranteed only when $n$ is large. One cannot apply (1) directly to a family of subsets that varies with $n$. The remedy is to prove the upper bound for the good GCD sum and the lower bound for the good $D$-sum \textbf{simultaneously} on the \textbf{same} good part of the cover, so that the good and bad sets coincide for both estimates. This simultaneous estimate is the content of Proposition \ref{key}, whose proof relies on a strengthened covering lemma (Lemma \ref{refine}) that combines the Besicovitch multiplicity bound with the uniform averaging technique.
\medskip

Combining the solutions to all the difficulties mentioned above, we can provide local estimates, as shown in Proposition \ref{key-}, \ref{key} and \ref{key-abc} respectively. These propositions imply that for each ``good'' point $b\in B(\CC)$, there is an open disc $\BD_{b}$ containing $b$ such that for any open disc $\BD'_{b}$ inside $\BD_{b}$, the desired inequality holds on $\BD'_{b}$ for $n$ large enough (depending on $\BD'_{b}$). This gives the estimates on some fixed subset of $B(\CC)$, but not all subsets. In fact, the open disc $\BD_{b}$ above may be very small, so this step is called ``local estimation''.
\medskip\medskip

\noindent\textbf{Step 3: Global estimates}. The final step is to glue all local estimates to a global one. The idea is still to (almost) cover the whole $B(\CC)$ (or the given subset $U$ in (1) and (2)) with the open discs used in the local estimates, but this step is easier than the process from ``microscopic local'' estimation to ``local estimation'' in Step 2, because at this step, we can allow the cover to have a small omitted part (whose measure has an upper bound depending on $\varepsilon$) and only need to ensure that there is no overlap. Since the canonical Xie--Yuan height associated to the omitted part is very small (see Proposition \ref{jiben-}), summing all these local estimates yields the global estimate.

Two further points should be noted. First, ensuring the absence of overlap is easy: at the end of the previous step we stated that every sub-disc $\BD'_{b}$ of each disc $\BD_{b}$ admits a corresponding local estimate, and by taking the discs $\BD'_{b}$ sufficiently small we can avoid overlap. Second, we can choose the above (almost) cover to be finite (with the complement of this cover still having small measure); then, since the estimate on each disc holds for $n$ large enough, the global estimate also holds for $n$ large enough.
\medskip\medskip

\noindent At the end of Introduction, we discuss the structure of the article. In \S\ref{canonical}, we recall the definition of canonical Xie--Yuan heights introduced by Xie--Yuan and estimate their lower and upper bounds. In \S\ref{3-} we recall some facts about linear entire curves and do uniform estimates on them. The proof of Proposition \ref{V} and the main theorems will be given in \S\ref{4}.

\section{Canonical height and Betti form}\label{canonical}
The goal of this section is to estimate the canonical Xie--Yuan height on abelian varieties defined in \cite{XY1}. The first three subsections recall various related definitions and basic properties. The main estimates are given in \S\ref{4.2}. We also derive a comparison of height functions on abelian schemes (see Proposition \ref{bijiao-} below).

\subsection{Canonical height on abelian varieties}\label{3.1}
We first recall the classical canonical height on abelian varieties, which is also called the N\'{e}ron--Tate height.

Let $B$ be a smooth projective curve over $\CC$ and let $K=\CC(B)$ be its function field. Let $A$ be an abelian variety over $K$ and let $L$ be a $\QQ$-line bundle on $A$. Then we can choose a Weil height function $h_{L}$ on $A(\ol{K})$ associated to $L$ as in \S\ref{1.2}, which is unique up to a bounded function. In the case of abelian varieties, however, we have a canonical choice of this height function by Tate's limiting process.

More precisely, when $L$ is symmetric, we have $[m]^{*}L\simeq m^{2}L$ for any integer $m$. Based on this, one can prove that the limit
$$
\lim\limits_{n\to\infty}\frac{h_{L}([2^{n}]x)}{4^{n}}
$$
exists for all $x\in A(\ol{K})$ and is independent of the choice of $h_{L}$. Moreover, the difference between this limit and $h_{L}$ is a bounded function on $A(\ol{K})$. Therefore this limit is a (canonical) Weil height function associated to $L$. We denote it by $\hat{h}_{L}$.

An important property of $\hat{h}_{L}$ is that it is a quadratic form on $A(\ol{K})$. Moreover, it is invariant under the translation action of $A^{(K/\CC)}(\CC)$ on $A(\ol{K})$. Thus it induces a quadratic form on the real vector space $V(A,K)\colonequals (A(K)/A^{(K/\CC)}(\CC))\otimes_{\BZ}\RR$, which is finite-dimensional by Lang--N\'{e}ron theorem. When $L$ is ample symmetric, $\hat{h}_{L}$ is a positive definite quadratic form on $V(A,K)$.

When $L$ is anti-symmetric (i.e. $[-1]^{*}L\simeq -L$), we have $[m]^{*}L\simeq mL$ for any integer $m$. Completely similarly as above, we can define the canonical function $\hat{h}_{L}$ as the limit
$$
\lim\limits_{n\to\infty}\frac{h_{L}([2^{n}]x)}{2^{n}},\quad\quad x\in A(\ol{K}).
$$
In this case, $\hat{h}_{L}$ induces a linear function on $V(A,K)$.

In general, we can write $L=L_{1}+L_{2}$, where $L_{1}=\frac{1}{2}(L+[-1]^{*}L)$ is a symmetric $\QQ$-line bundle and $L_{2}=\frac{1}{2}(L-[-1]^{*}L)$ is an anti-symmetric $\QQ$-line bundle. Then define $\hat{h}_{L}=\hat{h}_{L_{1}}+\hat{h}_{L_{2}}$.

\medskip\medskip

\noindent\textbf{Relative setting}\medskip

\noindent We also need to consider the relative setting due to our purpose. The first is the following height inequality.
\begin{proposition}\label{uniform}
Let $\pi\colon A\to S$ be an abelian scheme, where $A,S$ are projective varieties over $K$. Let $e\colon S\to A$ be the zero section. For an integer $m$, denote by $[m]$ the multiplication by $m$ on the abelian scheme $A$ over $S$. Let $L$ be a line bundle on $A$. Fix a Weil height function $h_{L}$ on $A(\ol{K})$. For each point $s\in S$, denote by $L_{s}$ the restriction of $L$ to the abelian subvariety $A_{s}$, and by $\hat{h}_{L_{s}}$ the canonical height function on $A(\ol{K})\times\{s\}$ associated to $L_{s}$. Let $h$ be a Weil height function on $S(\ol{K})$ associated to the line bundle $e^{*}L$ on $S$. Then for any $c>0$, the difference $|h_{L}(x)-\hat{h}_{L_{s}}(x)|$ is uniformly bounded for all $s\in S(\ol{K})$ with $|h(s)|<c$ and all $x\in A_{s}(\ol{K})$.
\end{proposition}
\begin{proof}
First consider the case when $L$ is symmetric, i.e. $[-1]^{*}L\simeq L$. In this case, the restriction of $[2]^{*}L$ and $4L$ to any fibre $A_{s}$ is isomorphic. Hence there is a line bundle $L_{S}$ on $S$ such that $[2]^{*}L\simeq 4L+\pi^{*}L_{S}$. Pulling back both sides via the zero map $[0]=e\circ\pi\colon A\to A$, we see that
$$
[0]^{*}L=[0]^{*}[2]^{*}L\simeq 4[0]^{*}L+[0]^{*}\pi^{*}L_{S}=4[0]^{*}L+\pi^{*}L_{S},
$$
which implies that $\pi^{*}L_{S}\simeq -3[0]^{*}L$. Therefore we have $[2]^{*}L\simeq 4L-3[0]^{*}L$.

By projection formula of the height, we have
$$
[2]^{*}h_{L}=h_{[2]^{*}L}+O(1)=4h_{L}-3h_{[0]^{*}L}+O(1)
$$
Note that $[0]^{*}L=\pi^{*}(e^{*}L)$, thus for any $x\in A(\ol{K})$, we have
$$
h_{L}([2]x)=4h_{L}(x)-3h(\pi(x))+O(1).
$$
Thus there is a constant $C>0$ such that $|h_{L}([2]x)-4h_{L}(x)|<C$ for all $x\in A(\ol{K})$ with $|h(\pi(x))|<c$. Fix one such $x$. Replace $x$ by $[2^{n}]x$, and divide both sides by $4^{n}$, we get
$$
\left|\frac{h_{L}([2^{n+1}]x)}{4^{n+1}}-\frac{h_{L}([2^{n}]x)}{4^{n}}\right|<\frac{C}{4^{n+1}}
$$
for all $n\geq0$. Recall that $\hat{h}_{L_{\pi(x)}}(x)=\lim_{n\to\infty}\frac{h_{L}([2^{n}]x)}{4^{n}}$, thus we have $|h_{L}(x)-\hat{h}_{L_{\pi(x)}}(x)|<C\sum_{n=0}^{\infty}\frac{1}{4^{n+1}}<C$ for all $x\in A(\ol{K})$ with $|h(\pi(x))|<c$. This proves the case when $L$ is symmetric.

When $L$ is anti-symmetric, i.e. $[-1]^{*}L\simeq -L$, the proof is completely similar as above (replace all `4' above by `2'). For general $L$, write $L=L_{1}+L_{2}$, where $L_{1}=\frac{1}{2}(L+[-1]^{*}L)$ is a symmetric $\QQ$-line bundle and $L_{2}=\frac{1}{2}(L-[-1]^{*}L)$ is an anti-symmetric $\QQ$-line bundle. The proposition then follows from the additivity of height functions.
\end{proof}
Now we consider the following estimate on a special kind of abelian schemes. This will be used in the proof of the main theorems.
\begin{proposition}\label{bijiao-}
Let $A$ be an abelian variety over $K$ and let $T$ be a projective variety over $\CC$. Let $L$ be a $\QQ$-line bundle on $A\times T_{K}$ which is relatively ample over $T_{K}$. Let $L^{\rm sym}=\frac{1}{2}(L+([-1]_{A}\times{\rm id}_{T_{K}})^{*}L)$ be the symmetrization of $L$ over $T_{K}$. For each point $t\in T(\CC)$, let $L^{\rm sym}_{t}$ be the restriction of $L^{\rm sym}$ to the fibre $A\times\{t\}$ and let $\hat{h}_{L^{\rm sym}_{t}}$ be the canonical height function on $A(K)\times\{t\}$. Then for all $(s,t)\in A(K)\times T(\CC)$, we have
$$
h_{L}((s,t))=\hat{h}_{L^{\rm sym}_{t}}(s)+O\left(1+\sqrt{\hat{h}_{L^{\rm sym}_{t}}(s)}\right).
$$
\end{proposition}
\begin{proof}
First note that the points in $T(\CC)\subseteq T_{K}(K)$ have bounded height with respect to any ample line bundle on $T_{K}$. Moreover, the function $\hat{h}_{L^{\rm sym}_{t}}(s)$ is independent of the choice of $t\in T(\CC)$. This is because all symmetric $\QQ$-line bundles $L^{\rm sym}_{t}$ are in the same numerical class, and hence they are isomorphic up to a torsion line bundle (they are in fact isomorphic).

Let $L_{t}$ be the restriction of $L$ to the fibre $A\times\{t\}$ and let $\hat{h}_{L_{t}}$ be the canonical height function on $A(K)\times\{t\}$. Then by Proposition \ref{uniform} above, we have
\begin{align*}
h_{L}((s,t))=\hat{h}_{L_{t}}(s)+O(1)
\end{align*}
for all $(s,t)\in A(K)\times T(\CC)$. Let $L'=L^{\rm sym}-L$, which is anti-symmetric over $T_{K}$. Define $L'_{t},\hat{h}_{L'_{t}}$ similarly as above. We claim that
\begin{align*}
\hat{h}_{L'_{t}}(s)=O\left(1+\sqrt{\hat{h}_{L^{\rm sym}_{t}}(s)}\right)
\end{align*}
for all $(s,t)\in A(K)\times T(\CC)$. If the claim is true, the proposition then follows from the above two inequalities.

Now we prove the claim. We first choose a Weil height function $h_{L'}$ on $A(K)\times T_{K}(K)$ as follows. Take an integral model $\CA$ of $A$ over $B$. Then $\CA\times_{\CC}T$ is an integral model of $A\times T_{K}$ over $B$. Extend $L'$ to a $\QQ$-line bundle $\CL'$ on $\CA\times_{\CC}T$. For a point $(s,t)\in A(K)\times T_{K}(K)$, view it as a section $(s,t)\colon B\to \CA\times_{\CC}T$. Then define
$$
h_{L'}((s,t))\colonequals \deg_{B}(s,t)^{*}\CL'=\int_{B(\CC)}(s,t)^{*}\omega',
$$
where $\omega'$ is a real $(1,1)$-form on $\CA\times_{\CC}T$ representing the first Chern class $c_{1}(\CL')$ of $\CL'$. If $t\in T(\CC)$, denote by $\omega'_{t}$ the restriction of $\omega'$ to the fibre $\CA\times\{t\}\simeq\CA$. Then $t\mapsto \omega'_{t}$ is a continuous map from $T(\CC)$ to the space of real continuous $(1,1)$-forms on $\CA$. Hence for fixed $s\in A(K)$, the map $t\mapsto h_{L'}((s,t))$ from $T(\CC)$ to $\RR$ is continuous. In particular, the compactness of $T(\CC)$ implies that the height $h_{L'}((s,t))\,(t\in T(\CC))$ is bounded for any fixed $s\in A(K)$.

By Proposition \ref{uniform} again, we have $h_{L'}((s,t))=\hat{h}_{L'_{t}}(s)+O(1)$ for all $(s,t)\in A(K)\times T(\CC)$. Combining with the previous paragraph, we see that for any fixed $s\in A(K)$, the canonical height $\hat{h}_{L'_{t}}(s)$ is bounded for all $t\in T(\CC)$. Note that $\hat{h}_{L'_{t}}(s)$ is invariant under the translation action of $A^{(K/\CC)}(\CC)$, and thus descends to a function on the finite-dimensional $\RR$-vector space $V(A,K)$, which is linear since $L'_{t}$ is anti-symmetric. Choose an $\RR$-basis $s_{1},\dots,s_{r}$ of $V(A,K)$. Choose $c>0$ such that $|\hat{h}_{L'_{t}}(s_{i})|\leq c$ for all $1\leq i\leq r$ and all $t\in T(\CC)$. Then for any $s=a_{1}s_{1}+\dots+a_{r}s_{r}\in V(A,K)$ and any $t\in T(\CC)$, we have
$$
\left|\hat{h}_{L'_{t}}(s)\right|=\left|\sum\limits_{i=1}^{r}a_{i}\hat{h}_{L'_{t}}(s_{i})\right|\leq c\sqrt{r}\cdot\sqrt{\sum_{i=1}^{r}a_{i}^{2}}.
$$
Finally, as we have mentioned at the beginning of the proof, the height $\hat{h}_{L_{t}}(s)$ depends only on $s$. Moreover, it induces a positive definite quadratic form on $V(A,K)$ since $L$ is relatively ample over $T_{K}$. Therefore the right hand side above is exactly $O\left(1+\sqrt{\hat{h}_{L^{\rm sym}_{t}}(s)}\right)$. This ends the proof.
\end{proof}
\subsection{Betti form and the canonical Xie--Yuan height}\label{betti}
The notion of Betti form was first introduced by Mok \cite{Mok} to study the Mordell--Weil group of abelian varieties over complex function fields. We briefly review it here.

Let $T$ be a projective variety over $\CC$. Let $\pi\colon\CA\to T$ be an abelian scheme with zero section $e\colon T\to\CA$. For an integer $m\geq0$, denote by $[m]$ the multiplication by $m$ map on $\CA$ over $T$. Let $\CL$ be a $\QQ$-line bundle on $\CA$. For each point $t\in T(\CC)$, there is a unique translation-invariant closed $(1,1)$-form $\omega_{t}$ on $\CA_{t}$ representing $c_{1}(\CL\big|_{\CA_{t}})$. The \emph{Betti form} on $\CA$ associated to $\CL$ is the unique $(1,1)$-form $\omega$ on $\CA$ satisfying the following two properties:
\begin{enumerate}[label=(\arabic*),leftmargin=*]
\item For any $t\in T(\CC)$, the restriction of $\omega$ to the fibre $\CA_{t}$ equals the translation-invariant form $\omega_{t}$.

\item \emph{Quadraticity}. For an integer $m$, we have $[m]^{*}\omega=m^{2}\omega$.
\end{enumerate}
We can deduce some basic facts from these two properties.
\begin{enumerate}[label=(\arabic*),leftmargin=*]
\item The pullback of $\omega$ via the zero section $e\colon T\to\CA$ is zero.

\item If $\CL$ is anti-symmetric, i.e. $[-1]^{*}\CL\simeq -\CL$, then $\omega=0$. \label{(2)}

\item \emph{Positivity}. When $\CL$ is relatively ample over $T$, the form $\omega$ is semi-positive.

\item \emph{Additivity}. For two $\QQ$-line bundles $\CL_{1},\CL_{2}$ on $\CA$, let $\omega_{1},\omega_{2}$ be their associated Betti forms. Then the Betti form associated to the $\QQ$-line bundle $\CL_{1}+\CL_{2}$ equals $\omega_{1}+\omega_{2}$. \label{(4)}

\item Let $A$ be the generic fibre of $\pi\colon\CA\to T$, which is an abelian variety over $K=\CC(T)$. Let $L$ be a $\QQ$-line bundle on $A$ and let $\CL$ be a $\QQ$-line bundle on $\CA$ extending $L$. Then the Betti form $\omega$ on $\CA$ associated to $\CL$ is independent of the choice of $\CL$, since different choices of $\CL$ only differ by a vertical divisor of $\CA$ and give isomorphic $\CL\big|_{\CA_{t}}$ for any $t\in T(\CC)$. We call this $\omega$ the \emph{Betti form on $\CA$ associated to $L$}. \label{(5)}
\end{enumerate}
Combining \ref{(2)}, \ref{(4)} and \ref{(5)}, we see that the Betti form on $\CA$ associated to $L\in {\rm Pic}(A)$ depends only on the numerical class of $L$.\medskip\medskip

\noindent\textbf{Canonical height as integration of Betti form}\medskip

\noindent Now let $T=B_{0}$ be a smooth quasi-projective curve over $\CC$. Denote $K=\CC(B_{0})$. Then the set $\CA(B_{0})$ of sections of $\pi\colon\CA\to B_{0}$ can be canonically identified with the set $A(K)$ of rational points of $A$. From now on, we always use $s$ instead of $\sigma$ to denote a section in $\CA(B_{0})$. Since we no longer consider a general base $S$, this will not cause any confusion. Moreover, we sometimes also use $B$ to denote a compactification of $B_{0}$ and use $S$ to denote a subset of $B(\CC)$.

Now Let $L$ be a symmetric $\QQ$-line bundle on $A$. Note that the set $\CA(B_{0})$ of sections of $\pi\colon\CA\to B_{0}$ can be canonically identified with the set $A(K)$ of rational points of $A$. Under this identification, we have a canonical height function $\hat{h}_{L}$ on $A(K)=\CA(B_{0})$. Let $\omega$ be the Betti form on $\CA$ associated to $L$ (see fact \ref{(5)} above). We have the following interpretation of the canonical height in terms of Betti form proved by Gauthier--Vigny \cite[Theorem B]{GV}.

\begin{theorem}\label{gv}
For any section $s\colon B_{0}\to\CA$, we have
$$
\hat{h}_{L}(s)=\int_{B_{0}(\CC)}s^{*}\omega.
$$
\end{theorem}

\begin{remark}\label{adelic}
A new interpretation (and proof) of Theorem \ref{gv} is given in \cite{Guo}, using the theory of adelic line bundles developed in \cite{YZ}. In this way, we can also generalize this theorem to arbitrary line bundles on general polarized algebraic dynamics. This approach is similar to the Tate limiting process. We do not adopt it in this article, so we omit the details and refer the reader to \cite{Guo} and \cite[\S 6.1, 6.4]{YZ}. Below, for a general ($\QQ$)-line bundle $L$ on an abelian variety $A$, we give a way of establishing Theorem \ref{gv} that does not use adelic line bundles. Of course, proving that these two approaches are equivalent is also nontrivial, and we likewise omit it.

Theorem \ref{gv} already gives the conclusion when $L$ is symmetric. When $L$ is an anti-symmetric, its associated Betti form is zero, whereas its corresponding canonical height $\hat{h}_{L}$ is not zero in general; hence in this case the Betti form is not enough. Instead, let $A^{\vee}$ be the dual abelian variety of $A$ and let $P$ be the Poincar\'{e} line bundle on $A\times A^{\vee}$ (which is symmetric). By the universal property of $P$, there is a point $y\in A^{\vee}(K)$ such that the restriction of $P$ to $A\times\{y\}$ is isomorphic to $L$ via the natural identification $A\overset{\simeq}\to A\times\{y\}$. Let $\CA^{\vee}$ be the dual abelian scheme of $\CA$ over $B_{0}$, whose generic fibre is isomorphic to $A^{\vee}$. Then we can define $\omega$ to be the restriction of the Betti form on $\CA\times_{B_{0}}\CA^{\vee}$ associated to $P$ to $\CA$ via $\CA\overset{({\rm id},y\circ\pi)}\lra\CA\times_{B_{0}}\CA^{\vee}$. For a general $L$, write $L$ as the sum of a symmetric $\QQ$-line bundle and an anti-symmetric $\QQ$-line bundle, and extend by linearity.

After establishing Theorem \ref{gv} for a general $L$, we can then define the canonical Xie--Yuan height associated to $L$ similarly as below. In this way, we can avoid introducing the symmetrization $D^{\rm sym}$ in Proposition \ref{V}, and hence the approach is more natural. However, this approach requires us to put effort into developing a more complete theory, which is unnecessary in this article; therefore we do not adopt it either.
\end{remark}
\medskip\medskip

\noindent\textbf{Canonical Xie--Yuan height}\medskip

\noindent The canonical Xie--Yuan height was introduced by Xie--Yuan \cite[\S 3.2]{XY1} based on the Theorem \ref{gv} above. Namely, with the previous notations, for a measurable subset $S$ of $B_{0}(\CC)$, we have the Xie--Yuan height
$$
h_{\omega,S}(s)\colonequals \int_{S\cap B_{0}(\CC)}s^{*}\omega.
$$
We call this the \emph{canonical Xie--Yuan height} of $s$ with respect to $(L,S)$, and also denote it by $\hat{h}_{L,S}$ or $\hat{h}_{D,S}$, where $D$ is a $\QQ$-Cartier divisor with $\CO_{A}(D)\simeq L$.

If $B$ is a smooth compactification of $B_{0}$ and $S\subseteq B(\CC)$ is a measurable subset, we also define $\hat{h}_{L,S}(s)\colonequals \hat{h}_{L,S\cap B_{0}(\CC)}(s)$. Note that this definition depends only on $A,L,s$, and does not depend on the choice of $B_{0}$, $\CA$ and $\CL$.
\medskip\medskip

\noindent\textbf{Quadraticity}\medskip

\noindent Similar to the N\'{e}ron--Tate height, Xie and Yuan proved in \cite[Proposition 3.2, Theorem 3.5]{XY1} that the canonical Xie--Yuan height function $\hat{h}_{L,S}\colon A(K)\to\RR$ is quadratic and invariant under the translation action of $A^{(K/\CC)}(\CC)$ on $A(K)$. Thus it induces a quadratic form on $V(A,K)=(A(K)/A^{(K/\CC)}(\CC))\otimes_{\BZ}\RR$, which we still denote by $\hat{h}_{L,S}$. Moreover, when $L$ is ample symmetric and $S$ is an open disc in $B(\CC)$, they further showed that this quadratic form is positive definite. This is called the \emph{non-degeneracy} of canonical Xie--Yuan height. See Proposition \ref{feituihua} for more details.

Moreover, Xie and Yuan actually proved that the pullback $s\mapsto s^{*}\omega$ is quadratic on $s\in A(K)$ and is invariant under translation by $A^{(K/\CC)}(\CC)$. Hence we can extend the definition of this pullback form to all $s\in V(A,K)$ and still denote it by $s^{*}\omega$. In this way, the canonical Xie--Yuan height $\hat{h}_{L,S}(s)$ equals the integration $\int_{S\cap B_{0}(\CC)}s^{*}\omega$ for all $s\in V(A,K)$.

\medskip\medskip

\noindent\textbf{Relative setting}\medskip

\noindent Let $A$ be an abelian variety over $K=\CC(B)$ as above. Let $X$ be a projective variety over $K$. Let $L$ be a $\QQ$-line bundle on $A\times X$. Assume that $L$ is symmetric over $X$, i.e. $([-1]_{A}\times{\rm id}_{X})^{*}L\simeq L$. Then for any $x\in X(K)$, the restriction $L_{x}$ of $L$ to $A$ via the embedding $A\hookrightarrow A\times X,\, a\mapsto (a,x)$ is a symmetric $\QQ$-line bundle on $A$ for any $x\in X(K)$. For a measurable subset $S$ of $B(\CC)$ and a rational point $(a,x)\in A(K)\times X(K)$, define the canonical Xie--Yuan height of $(a,x)$ associated to $(L,S)$ as
$$
\hat{h}_{L,S}((a,x))\colonequals \hat{h}_{L_{x},S}(a).
$$

Note that all $\QQ$-line bundles $L_{x}\,(x\in X(K))$ are symmetric and in the same numerical class. Therefore they are isomorphic up to a torsion line bundle (in fact they are all isomorphic). This implies that the height $\hat{h}_{L,S}((a,x))$ actually depends only on $a$.\medskip

\noindent Alternatively, we can interpret $\hat{h}_{L,S}$ by using Betti form on some integral model of $A\times X$. Explicitly, Let $\CA,\CX$ be respectively integral models of $A,X$ over $B$. Then $\CA\times_{B}\CX$ is an integral model of $A\times X$ over $B$. Take a nonempty Zariski open subset $B_{0}$ of $B$ such that $\CA_{B_{0}}\to B_{0}$ is an abelian scheme. Then $\CA_{B_{0}}\times_{B_{0}}\CX_{B_{0}}$ is an abelian scheme over $\CX_{B_{0}}$. Let $\omega$ be the Betti form on $\CA_{B_{0}}\times_{B_{0}}\CX_{B_{0}}$ associated to $L$ (see fact \ref{(5)} above). Then $\omega$ can be described as follows: By the previous paragraph, the $\QQ$-line bundles $L_{x}\,(x\in X(K))$ are symmetric and in the same numerical class. Thus their associated Betti forms on $\CA_{B_{0}}$ are the same. Denote this Betti form by $\omega_{0}$. Then $\omega$ is exactly the pullback of $\omega_{0}$ via the projection $\CA_{B_{0}}\times_{B_{0}}\CX_{B_{0}}\to\CA_{B_{0}}$.

For a rational point $(a,x)\in A(K)\times X(K)$, view it as a section $B_{0}\to\CA_{B_{0}}\times_{B_{0}}\CX_{B_{0}}$. Then we have
$$
\hat{h}_{L,S}((a,x))=\int_{S\cap B_{0}(\CC)}(a,x)^{*}\omega=\int_{S\cap B_{0}(\CC)}a^{*}\omega_{0}.
$$

More generally, if $L$ is a (not necessarily symmetric) $\QQ$-line bundle on $A\times X$, let $L^{\rm sym}=\frac{1}{2}(L+([-1]_{A}\times{\rm id}_{X})^{*}L)$ be the symmetrization of $L$ over $X$. Then for any $x\in X$, the difference $L_{x}-L^{\rm sym}_{x}$ is anti-symmetric on $A$. This means that the Betti forms on $\CA_{B_{0}}\times_{B_{0}}\CX_{B_{0}}$ associated to $L$ and $L^{\rm sym}$ are the same.

Finally, we sometimes also denote $\hat{h}_{L,S}$ by $h_{\omega,S}$ or $\hat{h}_{D,S}$, where $D$ is a $\QQ$-Cartier divisor on $A\times X$ with $\CO_{A\times X}(D)\simeq L$.
\subsection{Estimations of the canonical Xie--Yuan heights}\label{4.2}
In this subsection we give some estimations of the Xie--Yuan height. Let $A$ be an abelian variety over $K=\CC(B)$ as above. Let $L$ be an ample symmetric line bundle on $A$. Fix a K\"{a}hler form $\omega_{B}$ on $B(\CC)$, which induces a measure $\mu_{B}$ on it.

The first is the following non-degeneracy property proved by Xie--Yuan \cite[Theorem 3.5]{XY1}.
\begin{proposition}[Non-degeneracy]\label{feituihua}
Fix a measurable subset $S$ of $B(\CC)$ with $\mu_{B}(S)>0$. Then there is a constant $\varepsilon_{S}>0$ such that $\hat{h}_{L,S}(s)\geq \varepsilon_{S}\hat{h}_{L}(s)$ for all $s\in A(K)$.
\end{proposition}
In fact, as we have mentioned in \S\ref{betti}, Xie and Yuan proved that $\hat{h}_{L,S}$ induces a positive definite quadratic form on $V(A,K)$ when $S$ is an open disc in $B(\CC)$. The only property of $S$ used in their proof is that $\mu_{B}(S)>0$. Therefore the induced quadratic form $\hat{h}_{L,S}$ on $V(A,K)$ is also positive definite in this case. The proposition then follows from the fact that both $\hat{h}_{L,S}$ and $\hat{h}_{L}$ are positive definite quadratic forms on the finite-dimensional $\RR$-vector space $V(A,K)$.\medskip

On the other hand, we can also control the upper bound of canonical Xie--Yuan height.

\begin{proposition}[Smallness of Xie--Yuan height]\label{jiben-}
For any $\varepsilon>0$, there is a constant $\delta>0$ such that for every rational point $s\in A(K)$ and every measurable subset $S$ of $B(\CC)$ with $\mu_{B}(S)<\delta$, we have $h_{\omega,S}(s)\leq \varepsilon \int_{B_{0}}s^{*}\omega$.
\end{proposition}

\begin{proof}
Fix a Euclidean inner product on the finitely dimensional $\RR$-vector space $V(A,K)$ and let $\BS$ be the unit sphere with respect to this inner product. Since $\hat{h}_{L,S}$ is quadratic for any measurable subset $S\subseteq B(\CC)$, we only need to ensure the inequality
$$
h_{\omega,S}(s)\leq \varepsilon \int_{B_{0}}s^{*}\omega
$$
for all $s\in \BS$ and all measurable subsets $S$ of $B(\CC)$ with $\mu_{B}(S)<\delta$.

By definition, we have
$$
\hat{h}_{L,S}(s)=\int_{S\cap B_{0}(\CC)}s^{*}\omega,
$$
Now let $S_{0}=(B\backslash B_{0})(\CC)$, which is a finite subset of $B(\CC)$ containing the bad reduction locus of $\pi\colon\CA\to B$. Fix a point $s_{0}\in\BS$. Choose a Euclidean open neighbourhood $U$ of $S_{0}$ in $B(\CC)$ such that $\hat{h}_{L,U}(s_{0})=\int_{U\backslash S_{0}}s_{0}^{*}\omega<\frac{\varepsilon}{2}\int_{B_{0}(\CC)}s_{0}^{*}\omega$. Then there is an open neighbourhood of $s_{0}$ in $\BS$ such that
$$
\hat{h}_{L,U}(s)\leq\frac{\varepsilon}{2}\hat{h}_{L}(s)
$$
for all $s$ in this neighbourhood.

On the other hand, the compactness of $B(\CC)\backslash U\subseteq B_{0}(\CC)$ and the continuity of the map $s\mapsto s^{*}\omega$ from $\BS$ to $\Omega^{1,1}(B_{0})$ imply that there is a constant $c>0$ such that for all measurable subsets $S\subseteq B(\CC)\backslash U$ and all $s\in\BS$ sufficiently close to $s_{0}$, the inequality
$$
\int_{S}s^{*}\omega\leq c\mu_{B}(S)
$$
holds. In particular, for any $S\subseteq B(\CC)\backslash U$ with $\mu_{B}(S)<\frac{\varepsilon\hat{h}_{L}(s_{0})}{3c}$ and any $s\in\BS$ sufficiently close to $s_{0}$, we have $\hat{h}_{L,S}(s)\leq c\cdot\frac{\varepsilon\hat{h}_{L}(s_{0})}{3c}\leq\frac{\varepsilon}{2}\hat{h}_{L}(s)$. In conclusion, for any $S\subseteq B(\CC)$ with $\mu_{B}(S)<\frac{\varepsilon\hat{h}_{L}(s_{0})}{3c}$ and any $s\in\BS$ sufficiently close to $s_{0}$, we have the estimate
\begin{align*}
\hat{h}_{L,S}(s)&=\hat{h}_{L,S\cap U}(s)+\hat{h}_{L,S\backslash U}(s)\leq \hat{h}_{L,U}(s)+c\mu_{B}(S\backslash U)\\
&\leq \frac{\varepsilon}{2}\hat{h}_{L}(s)+c\mu_{B}(S)\leq \frac{\varepsilon}{2}\hat{h}_{L}(s)+\frac{\varepsilon}{2}\hat{h}_{L}(s)=\varepsilon \hat{h}_{L}(s)
\end{align*}
Now for each $s_{0}\in\BS$, we have chosen a constant $\delta_{s_{0}}\colonequals \frac{\varepsilon\hat{h}_{L}(s_{0})}{3c}>0$ and an open neighbourhood of $s_{0}$ in $\BS$. By compactness of $\BS$, we can find a finite cover of $\BS$ by these open neighbourhoods. The proposition then follows by setting $\delta$ to be the minimum of all $\delta_{s_{0}}$ corresponding to some open neighbourhood in this finite covering.
\end{proof}
At the end of this subsection, we compare the canonical Xie--Yuan height functions for pullback ($\QQ$-)line bundles. The estimate is similarly to Proposition \ref{bijiao-}.
\begin{lemma}\label{bijiao}
Let $\phi\colon A'\to A$ be a morphism between $K$-abelian varieties, and let $f\colon X'\to X$ be a morphism between $K$-projective varieties. Let $L$ be a $\QQ$-line bundle on $A\times X$ which is relatively ample over $X$, and take $L'=(\phi\times f)^{*}L$. Let $L^{\rm sym}=\frac{1}{2}\left(L+([-1]_{A}\times{\rm id}_{X})^{*}L\right)$ be the symmetrization of $L$ over $X$, and similarly for $(L')^{\rm sym}$. Fix a measurable subset $S\subseteq B(\CC)$. Then for all $(s',x')\in A'(K)\times X'(K)$, we have
$$
\hat{h}_{(L')^{\rm sym},S}((s',x'))=\hat{h}_{L^{\rm sym},S}((\phi(s'),f(x')))+O\left(1+\sqrt{\hat{h}_{L^{\rm sym},S}((\phi(s'),f(x')))}\right).
$$
\end{lemma}
\begin{proof}
When $\mu_{B}(S)=0$, both sides are 0. Now we assume that $\mu_{B}(S)>0$. We may also assume that $X'(K)\neq\emptyset$, and thus $X(K)\neq\emptyset$. For all points $x\in X(K)$, the restriction of $L^{\rm sym}$ to $A$ via the embedding $A\hookrightarrow A\times X,\, s\mapsto (s,x)$ are in the same numerical class. Let $L_{0}\in {\rm Pic}(A)_{\QQ}$ be a representative in this numerical class. Then for any $x'\in X'(K)$, the restriction of $(L')^{\rm sym}$ to $A'$ via the embedding $A'\hookrightarrow A'\times X',\, s'\mapsto (s',x')$ is numerically equivalent to $\phi^{*}L_{0}$.

Take integral models $\CA,\CA'$ of $A,A'$ over $B$. Let $B_{0}$ be a Zariski open subset of $B$ such that $\phi\colon A'\to A$ extends to a $B_{0}$-morphism $\CA'_{B_{0}}\to\CA_{B_{0}}$, which we still denote by $\phi$. Let $\omega_{0}$ be the Betti form on $\CA_{B_{0}}$ associated to $L_{0}$. Then for any $(s,x)\in A(K)\times X(K)$, by the explanation at the end of \S\ref{betti}, we have
$$
\hat{h}_{L^{\rm sym},S}((s,x))=\int_{S\cap B_{0}(\CC)}s^{*}\omega_{0}.
$$
Similarly, let $\omega'_{0}$ be the Betti form on $\CA'_{B_{0}}$ associated to $\phi^{*}L_{0}$. Then for any $(s',x')\in A'(K)\times X'(K)$, we have
$$
\hat{h}_{(L')^{\rm sym},S}((s',x'))=\int_{S\cap B_{0}(\CC)}(s')^{*}\omega'_{0}.
$$

Factoring $\phi\colon A'\to A$ as $t_{a}\circ\phi_{0}$, where $\phi_{0}\colon A'\to A$ is a homomorphism, and $t_{a}\colon A\to A$ is the translation by some point $a\in A(K)$. Then $\phi\colon\CA'_{B_{0}}\to\CA_{B_{0}}$ also factors as $\CA'_{B_{0}}\overset{\phi_{0}}\to\CA_{B_{0}}\overset{t_{a}}\to\CA_{B_{0}}$. By the functoriality of Betti forms, we see that $\omega'_{0}$ equals $\phi_{0}^{*}\omega_{0}$. Hence for any $(s',x')\in A'(K)\times X'(K)$, we have
\begin{align*}
\hat{h}_{L^{\rm sym},S}((\phi(s'),f(x')))-\hat{h}_{(L')^{\rm sym},S}((s',x'))&=\int_{S\cap B_{0}(\CC)}\phi(s')^{*}\omega_{0}-\int_{S\cap B_{0}(\CC)}s'^{*}\phi_{0}^{*}\omega_{0}\\
&=\int_{S\cap B_{0}(\CC)}((\phi_{0}(s')+a)^{*}\omega_{0}-\phi_{0}(s')^{*}\omega_{0})\\
&=h_{\omega_{0},S}(\phi_{0}(s')+a)-h_{\omega_{0},S}(\phi_{0}(s')).
\end{align*}
Note that $a\in A(K)$ is fixed. By quadraticity, the term $h_{\omega_{0},S}(s+a)-h_{\omega_{0},S}(s)-h_{\omega_{0},S}(a)$ is linear in $s\in V(A,K)$. By Proposition \ref{feituihua}, the term $\hat{h}_{L^{\rm sym},S}((s,x))=h_{\omega_{0},S}(s)$ is a positive definite quadratic form on $V(A,K)$. Thus we have
\begin{align*}
h_{\omega_{0},S}(s+a)-h_{\omega_{0},S}(s)&=(h_{\omega_{0},S}(s+a)-h_{\omega_{0},S}(s)-h_{\omega_{0},S}(a))+h_{\omega_{0},S}(a)\\
&=O\left(1+\sqrt{h_{\omega_{0},S}(s)}\right)
\end{align*}
for all $s\in A(K)$. This completes the proof.
\end{proof}
\section{Uniform estimations for linear entire curves}\label{3-}
\subsection{Linear entire curves on complex abelian varieties}
Let $A$ be a complex abelian variety. Denote by ${\rm Lie}(A)$ the group of translation-invariant holomorphic derivations on $\CO_{A}$. Let $x\in A$ be a point, and $v\in{\rm Lie}(A)$ be a nonzero vector in the Lie algebra. By restriction, we have a canonical isomorphism ${\rm Lie}(A)\simeq T_{x}A$. By definition, $v$ defines a translation-invariant vector field on $A$, and thus it is a foliation $\CF(v)$ on $A$. The leaf $\CF(v)_{x}=\CF(x, v)$ of this foliation through $x$ is called the \emph{linear entire curve through $x$ in the direction $v$}. By translation, we have
$$
\CF(x, v)=\CF(0, v)+x.
$$

Alternatively, write $A={\rm Lie}(A)/ H_{1}(A, \ZZ)$. Denote by $\tilde{x}$ a lifting of $x$ in ${\rm Lie}(A)$. In the complex vector space ${\rm Lie}(A)$, there is a linear subset given by
$$
\tilde{\phi}_{(\tilde{x}, v)}\colon \CC\ \lra\ {\rm Lie}(A),\qquad z\mapsto\tilde{x}+zv.
$$
Composing with ${\rm Lie}(A)\to A$, we obtain an entire curve
$$
\phi_{(x, v)}\colon\CC\ \lra\ A.
$$
The entire curve is independent of the choice of $\tilde{x}$, and its image is exactly the linear leaf $\CF(x, v)$. We call $\phi_{(x, v)}\colon\CC\to A$ the \emph{linear entire curve through $x$ in the direction $v$}.

\medskip\medskip

\noindent\textbf{Zariski closures of linear entire curves}\medskip

\noindent Linear entire curves are generally not algebraic. We recall a clear description in \cite{arXiv:2308.08117} of Zariski closures of linear leaves in fibres of abelian schemes.

Let $B$ be a smooth quasi-projective curve over $\CC$, and let $\pi\colon\CA\to B$ be an abelian scheme.  Let $K=\CC(B)$ be the function field of $B$, and $A$ be the generic fibre of $\pi\colon\CA\to B$. Denote $V(A, K)=(A(K)/A^{(K/\CC)}(\CC))_{\RR}$, which is a finite-dimensional $\RR$-vector space by the Lang--N\'{e}ron theorem. Here $A^{(K/\CC)}$ is the $K/\CC$-trace of $A$. For an abelian subvariety $H$ of $A$, there is a natural inclusion
$$
V(H, K)=(H(K)/H^{(K/\CC)}(\CC))_{\RR}=(H(K)/H(K)\cap A^{(K/\CC)}(\CC))_{\RR}\subseteq V(A, K).
$$
Moreover, Xie--Yuan proved in \cite[Lemma 2.4]{arXiv:2308.08117} that for every $s\in V(A, K)$, there exists a \emph{unique minimal abelian subvariety} $G(s)$ of $A$ with $s\in V(G(s), K)$.

The following is the main proposition cited in this subsection.
\begin{proposition}[{\cite[Proposition 2.6]{arXiv:2308.08117}}]\label{closure}
Let $B$ be a smooth quasi-projective curve over $\CC$, and let $\pi\colon\CA\to B$ be an abelian scheme. Let $K=\CC(B)$ be the function field of $B$, and $A$ be the generic fibre of $\pi\colon\CA\to B$. Let $s$ be a nonzero element of $V(A, K)$. Denote by $\CG$ the Zariski closure of $G(s)$ in $\CA$. Then there is a subset $S(A)$ of $B(\CC)$ of measure $0$ such that for any $b\in B(\CC)\backslash S(A)$, the Zariski closure of the linear leaf $\CF(0, \delta(s, v_{b}))$ in $\CA_{b}$ equals $\CG_{b}$. Here $v_{b}\in T_{b}B$ is a nonzero vector, and the transfer map (associated to $s$) $\delta(s,\cdot)\colon T_{b}B\to{\rm Lie}(\CA_{b})$ is defined in \cite[\S 2.2]{arXiv:2308.08117} using the Betti map.
\end{proposition}
In \cite[Proposition 2.6]{arXiv:2308.08117}, Xie--Yuan made an extra assumption that the natural map ${\rm End}_{K}(A)\to\rm{End}_{\ol{K}}(A_{\ol{K}})$ is an isomorphism. This can be removed by taking a finite base change $B'\to B$ and using the functoriality of $G(s)$ and $\delta(s,v_{b})$.
\subsection{Approximating sections of abelian schemes by linear entire curves}
\begin{lemma}\label{linear}
Let $B$ be a smooth quasi-projective curve over $\CC$, and let $\pi\colon\CA\to B$ be an abelian scheme. Let $K=\CC(B)$ be the function field of $B$, and $A$ be the generic fibre of $\pi\colon\CA\to B$. Let $\{s_{n}\}_{n\geq1}$ be a sequence in $\CA(B)$ and $\{\ell_{n}\}_{n\geq1}$ be a sequence of positive real numbers tending to infinity such that the image of $\ell_{n}^{-1}s_{n}$ in $V(A, K)$ converges to a nonzero element $s_{\infty}$ in $V(A, K)$.

Fix a point $b\in B(\CC)$, an open unit disc $\BD$ in $B(\CC)$ with centre $b$ and a distance function $d$ on $\CA_{\BD}$. Let $z$ be the standard coordinate of $\BD$. Take the tangent vector $v_{b}\in T_{b}B$ to be the one represented by $v_{\rm{st}}=\frac{\rm{d}}{\rm{d}z}\in T_{0}\BD$. Then for any $R>0$ and $\varepsilon>0$, there are constants $0<\varepsilon'<1$ and $n_{0}>0$ satisfying the following property: for each $b'\in\BD_{\varepsilon'}$, each $z\in\BD_{R}$ and each $x\in\CA_{b}=\CA_{0}$, if $n>n_{0}$ and $d(s_{n}(b'),x)\leq\varepsilon'$, then $b'+\ell_{n}^{-1}z\in\BD$, and
$$
d(s_{n}(b'+\ell_{n}^{-1}z),\phi_{(x,\delta(s_{\infty},v_{b}))}(z))<\varepsilon.
$$
\end{lemma}

\begin{proof}
Let $d_{0}$ be the metric on $\CA_{\BD}$ defined before Lemma 3.3 in \cite{arXiv:2308.08117}. Fix a Euclidean norm $\left|\cdot\right|$ on $V(A,K)$. By the proof of \cite[Lemma 3.3]{arXiv:2308.08117}, we see that there are constants $C_{0},C>0$ depending only on $\CA,\BD$ and $s_{\infty}$ such that for any $b'\in\BD$ and $z\in\CC$ with $|b'|+\ell_{n}^{-1}|z|<0.9$, the inequality
$$
d_{0}(s_{n}(b'+\ell_{n}^{-1}z),\phi_{(x,\delta(s_{\infty},v_{b}))}(z))\leq d_{0}(s_{n}(b'),x)+C|z|\cdot|u_{n}|/\ell_{n}+C_{0}\ell_{n}^{-1}|z|^{2}+C|s_{\infty}||b'z|
$$
holds. Here $u_{n}=s_{n}-\ell_{n}s_{\infty}\in V(A,K)$. By our assumption $|u_{n}|/\ell_{n}\to 0$ as $n\to\infty$.

Since $d$ and $d_{0}$ induce the same topology on the compact subset $\CA_{\ol{\BD}_{0.9}}$, there is a constant $\delta>0$ such that for any $x_{1},x_{2}\in\CA_{\ol{\BD}_{0.9}}$, if $d_{0}(x_{1},x_{2})<\delta$, then $d(x_{1},x_{2})<\varepsilon$. Similarly, there is $\varepsilon_{1}>0$ such that if $d(x_{1},x_{2})<\varepsilon_{1}$, then $d_{0}(x_{1},x_{2})<\frac{\delta}{3}$. Choose $0<\varepsilon'<\min\{\varepsilon_{1},\frac{\delta}{3CR|s_{\infty}|},0.8\}$. Finally, choose $n_{0}$ such that for all $n>n_{0}$, we have $\ell_{n}^{-1}R<0.1$ and $CR\cdot|u_{n}|/\ell_{n}+C_{0}\ell_{n}^{-1}R^{2}<\frac{\delta}{3}$. We prove that $\varepsilon'$ and $n_{0}$ satisfy the lemma.

In fact, for any $b'\in\BD_{\varepsilon'}$, $z\in\BD_{R}$ and $n>n_{0}$ with $d(s_{n}(b'),x)<\varepsilon'$, we have $|b'|+\ell_{n}^{-1}|z|<0.8+0.1=0.9$ and $d_{0}(s_{n}(b'),x)<\frac{\delta}{3}$. Hence the right hand side above is less than
$$
\frac{\delta}{3}+CR\cdot|u_{n}|/\ell_{n}+C_{0}\ell_{n}^{-1}R^{2}+C|s_{\infty}|R\varepsilon'<\delta.
$$
This implies the lemma by our choice of $\delta$.
\end{proof}

\subsection{Estimations of a single linear entire curve}\label{2.3}
The goal of this subsection is to prove Corollary \ref{rm} below, which includes two inequalities for every (Zariski dense) linear entire curve. This corollary can be seen as an unweighted version of the second main theorem (as well as the GCD inequality) for abelian varieties in Nevanlinna theory, which was established by Yamanoi (see Theorem \ref{Nevanlinna} below). We will actually deduce the corollary from Yamanoi's theorem.

We first recall some basic Nevanlinna theory.

\medskip\medskip
\noindent\textbf{Counting function}\medskip

\noindent Let $X$ be a projective variety over $\CC$ and let $Y\subseteq X$ be a closed subscheme. Let $f\colon \CC\to X$ be a holomorphic map such that $f(\CC)\nsubseteq Y$. We recall the definition of $N(r,f,Y)$ in \cite{zbMATH02144170}.

Let $D_{1},\dots,D_{l}$ be effective Cartier divisors such that $Y=D_{1}\cap\dots\cap D_{l}$. This means that $\CI_{Y}=\CI_{D_{1}}+\dots+\CI_{D_{l}}$, where $\CI_{Y},\CI_{D_{1}},\dots,\CI_{D_{l}}$ are the defining ideal sheaves in $\CO_{X}$. Set
$$
{\rm ord}_{z}f^{*}Y\colonequals \min\{{\rm ord}_{z}f^{*}D_{1},\dots,{\rm ord}_{z}f^{*}D_{l}\}.
$$
When $f(\CC)\subseteq {\rm Supp}(D_{i})$, we set ${\rm ord}_{z}f^{*}D_{i}=+\infty$. The definition of ${\rm ord}_{z}f^{*}Y$ does not depend on the choice of the Cartier divisors $D_{1},\dots,D_{l}$. We define the \emph{counting function} by
$$
N(r,f,Y)\colonequals \int_{1}^{r}\left(\sum\limits_{z\in\BD_{t}}{\rm ord}_{z}f^{*}Y\right)\frac{{\rm d}t}{t},
$$
$$
N^{(k)}(r,f,Y)\colonequals \int_{1}^{r}\left(\sum\limits_{z\in\BD_{t}}\min\{k,{\rm ord}_{z}f^{*}Y\}\right)\frac{{\rm d}t}{t}.
$$

Here $\BD_{t}=\{z\in\CC\big| |z|<t\}$ is the open disc of radius $t$ with centre 0. \medskip

\medskip\medskip

\noindent\textbf{Characteristic function and proximity function}\medskip

\noindent Let $X$ be a projective variety over $\CC$. Let $L$ be a line bundle on $X$ with a (fixed) smooth Hermitian metric $\left\|\cdot\right\|$ on it. Let $c_{1}(L,\left\|\cdot\right\|)$ be the Chern form. Let $f\colon \CC\to X$ be a holomorphic curve. Define the \emph{characteristic function} by
$$
T(r,f,L)\colonequals \int_{1}^{r}\left(\int_{\BD_{t}}f^{*}c_{1}(L,\left\|\cdot\right\|)\right)\frac{{\rm d}t}{t}+O(1),
$$
which is independent of the choice of $\left\|\cdot\right\|$ up to a bounded function.

For an effective Cartier divisor $D$ on $X$ associated to a section $s_{D}\in \Gamma(X,L)$, We also set $T(r,f,D)\colonequals T(r,f,L)$. If $f(\CC)\nsubseteq D$, define the \emph{proximity function} by
$$
m(r,f,D)\colonequals -\int_{0}^{2\pi}\log\left\|s_{D}(f(r{\rm e}^{{\rm i}\theta}))\right\|\frac{{\rm d\theta}}{2\pi}+O(1),
$$
which is also independent of the choice of $\left\|\cdot\right\|$ up to a bounded function. As $X$ is compact, the function $x\mapsto\left\|s_{D}(x)\right\|$ is bounded above, and thus $m(r,f,D)$ is bounded below.

\begin{theorem}[First Main Theorem]\label{fmt}
$X,L,D$ as above. Then for any holomorphic map $f\colon \BD_{R}\to X$ with $f(\BD_{R})\nsubseteq D$ and any $0<r<R$, we have
 $$
 T(r,f,D)=m(r,f,D)+N(r,f,D)+O(1).
 $$
\end{theorem}

The next theorem is due to Yamanoi. Part (1) of the theorem is called the Second Main Theorem (for abelian varieties), which is an analytic analogue of our abc inequality (Theorem \ref{main1}). Part (2) is an analytic analogue of our GCD inequality (Theorem \ref{main}).

\begin{theorem}\label{Nevanlinna}
Let $A$ be an abelian variety over $\CC$ and let $L$ be an ample line bundle on $A$. Let $f\colon \CC\to A$ be a holomorphic curve such that the image of $f$ in $A$ is Zariski dense.

\begin{enumerate}[label=\normalfont(\arabic*),leftmargin=*]
\item {\rm (Second Main Theorem, \cite[Theorem 3.1.1]{zbMATH02144170})} If $D$ is a reduced effective divisor on $A$, then
$$
T(r,f,D)\leq N^{(1)}(r,f,D)+\varepsilon T(r,f,L)\quad\quad \big|\big|_{\varepsilon}\quad \forall \varepsilon>0.
$$
Here $\big|\big|$ means that the estimate holds for $r>0$ outside some exceptional set of finite Lebesgue measure. $T(r,f,D), T(r,f,L)$ are with respect to any smooth Hermitian metrics on $\CO_{A}(D), L$ respectively. Similarly for \emph{(2)} below.

\item {\rm (\cite[Corollary 2.5.8]{zbMATH02144170})} If $Y\subseteq A$ is a closed subscheme whose support has codimension greater than 1, then
$$
N(r,f,Y)\leq\varepsilon T(r,f,L)\quad\quad \big|\big|_{\varepsilon}\quad \forall \varepsilon>0.
$$
\end{enumerate}
\end{theorem}
Now we can deduce the following corollary. Actually, we will use Theorem \ref{Nevanlinna} only for linear entire curves. In this case there is a simplified proof in \cite[pp. 255-257]{zbMATH06448080}. For Theorem \ref{Nevanlinna} (1) we only use its non-truncated version, which implies the defect $\delta(f,D)=\liminf_{r\to\infty}m(r,f,D)/T(r,f,D)$ is zero, as proved in \cite{SY97}. The non-truncated version was also generalized to semi-abelian varieties in \cite{NWY02}.

\begin{corollary}\label{rm}
Let $A$ be an abelian variety over $\CC$ and let $D$ be an ample effective divisor on $A$. Let $f\colon \CC\to A$ be a linear entire curve such that the image of $f$ in $A$ is Zariski dense. Fix $\varepsilon>0$ and a smooth Hermitian metric $\left\|\cdot\right\|$ on $L=\CO_{A}(D)$. Then

\begin{enumerate}[label=\normalfont(\arabic*),leftmargin=*]
\item For $r$ large enough, we have
$$
\sum\limits_{z\in\BD_{r}}{\rm ord}_{z}f^{*}D\geq (1-\varepsilon) \int_{\BD_{r}}f^{*}c_{1}(L,\left\|\cdot\right\|).
$$

\item If $Y\subseteq A$ is a closed subscheme of codimension greater than $1$, then for $r$ large enough, we have
$$
\sum\limits_{z\in\BD_{r}}{\rm ord}_{z}f^{*}Y\leq \varepsilon \int_{\BD_{r}}f^{*}c_{1}(L,\left\|\cdot\right\|).
$$
\end{enumerate}
\end{corollary}
\begin{proof}
For all $r>0$, let $
\tau_{L}(r)\colonequals \int_{\BD_{r}}f^{*}c_{1}(L,\left\|\cdot\right\|)$. We first show that both $\frac{\tau_{L}(r)}{r^{2}}$ and $\frac{T(r,f,L)}{r^{2}}$ have limits. In fact, let $\omega$ be the unique translation-invariant $(1,1)$-form representing the Chern class $c_{1}(L)$. Then since $f$ is linear, there is a constant $c>0$ such that
$$
\int_{\BD_{r}}f^{*}\omega=cr^{2}
$$
for all $r>0$, which also implies that
$$
T(r,f,L)=\int_{1}^{r}\left(\int_{\BD_{t}}f^{*}\omega\right)\frac{{\rm d}t}{t}+O(1)=\frac{1}{2}cr^{2}+O(1)
$$
for all $r>0$. This proves $\lim_{r\to\infty}\frac{T(r,f,L)}{r^{2}}=\frac{c}{2}$.

Now let $\alpha$ be the 1-form on $A$ such that $c_{1}(L,\left\|\cdot\right\|)-\omega={\rm d}\alpha$. Then we have
\begin{align*}
\left|\int_{\BD_{r}}f^{*}c_{1}(L,\left\|\cdot\right\|)-\int_{\BD_{r}}f^{*}\omega\right|=\left|\int_{\partial\BD_{r}}f^{*}\alpha\right|
\end{align*}
By compactness, the right hand side is bounded above by a linear term of $r$. This shows that $\lim_{r\to\infty}\frac{\tau_{L}(r)}{r^{2}}=c$.

Next we claim that $\lim_{r\to\infty}\frac{N(r,f,D)}{r^{2}}=\frac{c}{2}$ and $\lim_{r\to\infty}\frac{N(r,f,Y)}{r^{2}}=0$. Once the claims are true, apply the Lemma \ref{shufen} below to the monotonically increasing functions $n_{D}(r)\colonequals \sum\limits_{z\in\BD_{r}}{\rm ord}_{z}f^{*}D$ and $n_{Y}(r)\colonequals \sum\limits_{z\in\BD_{r}}{\rm ord}_{z}f^{*}Y$, we get $n_{D}(r)\sim cr^{2}\sim\tau_{L}(r)$ and $n_{Y}(r)=o(r^{2})=o(\tau_{L}(r))$, which are exactly the inequalities in the proposition. It remains to prove the claims.

By Theorem \ref{Nevanlinna}, for any $\varepsilon>0$, there is a subset $\Sigma_{\varepsilon}$ of $\RR_{>0}$ of finite Lebesgue measure such that $N(r,f,D)\geq (1-\varepsilon)T(r,f,L)$ and $N(r,f,Y)\leq\varepsilon T(r,f,L)$ for all $r\in\RR_{>0}\backslash\Sigma_{\varepsilon}$. Let $m$ be the standard Lebesgue measure on $\RR$ with $m([0,1])=1$. Then for any $r>1+m(\Sigma_{\varepsilon})$ large enough, there are real numbers $r',r''>0$ with $r-1-m(\Sigma_{\varepsilon})<r'<r$ and $r<r''<r+1+m(\Sigma_{\varepsilon})$ such that $N(r',f,D)\geq (1-\varepsilon)T(r',f,L)$ and $N(r'',f,Y)\leq\varepsilon T(r'',f,L)$. Hence we have
$$
N(r,f,D)\geq N(r',f,D)\geq (1-\varepsilon)T(r',f,L)\geq (1-\varepsilon)T(r-1-m(\Sigma_{\varepsilon}),f,L)
$$
and
$$
N(r,f,Y)\leq N(r'',f,Y)\leq \varepsilon T(r'',f,L)\leq \varepsilon T(r+1+m(\Sigma_{\varepsilon}),f,L)
$$
For fixed $\varepsilon>0$, both the terms $T(r-1-m(\Sigma_{\varepsilon}),f,L)$ and $T(r+1+m(\Sigma_{\varepsilon}),f,L)$ equal to $\frac{c}{2}r^{2}+o(r^{2})$, and thus $\liminf_{r\to\infty}\frac{N(r,f,D)}{r^{2}}\geq (1-\varepsilon)\frac{c}{2}$ and $\limsup_{r\to\infty}\frac{N(r,f,Y)}{r^{2}}\leq \varepsilon \frac{c}{2}$. Since $\varepsilon$ is arbitrary, we have $\liminf_{r\to\infty}\frac{N(r,f,D)}{r^{2}}\geq \frac{c}{2}$ and $\limsup_{r\to\infty}\frac{N(r,f,Y)}{r^{2}}\leq 0$. On the other hand, the term $N(r,f,Y)$ is always nonnegative, and the First Main Theorem \ref{fmt} implies that $N(r,f,D)\leq T(r,f,L)+O(1)=\frac{c}{2}r^{2}+O(1)$. This proves the claims.
\end{proof}

\begin{lemma}\label{shufen}
Let $n\colon \RR_{>0}\to\RR$ be a monotone increasing function. For $r>0$, let $N(r)=\int_{1}^{r}\frac{n(t)}{t}{\rm d}t$. If $\lim_{r\to\infty}\frac{N(r)}{r^{2}}=a\in\RR$, then $\lim_{r\to\infty}\frac{n(r)}{r^{2}}=2a$.
\end{lemma}
\begin{proof}
For each $r>0$ and $0<c<1$, as $n$ is monotonically increasing, we have
\begin{align*}
n(cr)\int_{cr}^{r}\frac{{\rm d}t}{t}\leq\int_{cr}^{r}n(t)\frac{{\rm d}t}{t}\leq n(r)\int_{cr}^{r}\frac{{\rm d}t}{t}.
\end{align*}
Note that $\int_{cr}^{r}\frac{{\rm d}t}{t}=\log\frac{1}{c}$ and $\int_{cr}^{r}n(t)\frac{{\rm d}t}{t}=N(r)-N(cr)$. Thus we have
\begin{align*}
n(cr)\log\frac{1}{c}\leq N(r)-N(cr)\leq n(r)\log\frac{1}{c}.
\end{align*}
Now fix $0<c<1$. As $N(r)/r^{2}\to a$, we have $N(cr)/r^{2}\to ac^{2}$. Dividing the left half of the inequality above by $(cr)^{2}$, and taking the upper limit as $r\to\infty$, we get
$$
\limsup\limits_{r\to\infty}\frac{n(r)}{r^{2}}=\limsup\limits_{r\to\infty}\frac{n(cr)}{(cr)^{2}}\leq\lim\limits_{r\to\infty}\frac{N(r)-N(cr)}{c^{2}r^{2}\log(1/c)}=\frac{a(1-c^{2})}{c^{2}\log(1/c)}.
$$
Similarly, we also have
$$
\liminf\limits_{r\to\infty}\frac{n(r)}{r^{2}}\geq\lim\limits_{r\to\infty}\frac{N(r)-N(cr)}{r^{2}\log(1/c)}=\frac{a(1-c^{2})}{\log(1/c)}.
$$
Since $0<c<1$ is arbitrary, let $c\to 1^{-}$, we then get
$$
2a=\lim_{c\to 1^{-}}\frac{a(1-c^{2})}{c^{2}\log(1/c)}\geq \limsup\limits_{r\to\infty}\frac{n(r)}{r^{2}}\geq \liminf\limits_{r\to\infty}\frac{n(r)}{r^{2}}\geq \lim_{c\to 1^{-}}\frac{a(1-c^{2})}{\log(1/c)}=2a.
$$
This just means that $\lim_{r\to\infty}\frac{n(r)}{r^{2}}=2a$.
\end{proof}

\begin{remark}\label{jq}
Corollary \ref{rm}(2) can be strengthened to assert that for any effective Cartier divisor $D$ whose support contains $Y$, we have
$$
\sum\limits_{z\in\BD_{r}\cap f^{-1}(Y)}{\rm ord}_{z}f^{*}D\leq\varepsilon\int_{\BD_{r}}f^{*}c_{1}(L,\left\|\cdot\right\|)
$$
for $r$ large enough. In fact, we claim that the order ${\rm ord}_{z}f^{*}D\,(z\in\CC)$ is bounded for any fixed $D$. Let $f=\phi_{(x_{0},v)}$, where $x_{0}\in A(\CC)$ is the starting point and $v\in{\rm Lie}(A)$ is the direction. Then ${\rm ord}_{z}f^{*}D={\rm ord}_{0}\phi_{(f(z),v)}^{*}D$ for all $z\in\CC$, and thus we only need to bound the term ${\rm ord}_{0}\phi_{(x,v)}^{*}D\,(x\in A(\CC))$.

As $f(\CC)$ is Zariski dense in $A$, the linear entire curve $\phi_{(x,v)}$ is Zariski dense for any $x\in A(\CC)$. Thus for each $x\in A(\CC)$, the order ${\rm ord}_{0}\phi_{(x,v)}^{*}D$ is finite. By Rouch\'{e}'s theorem, there is a Euclidean open neighbourhood $U_{x}$ of $x$ in $A(\CC)$ such that ${\rm ord}_{0}\phi_{(x',v)}^{*}D\leq {\rm ord}_{0}\phi_{(x,v)}^{*}D$ for any $x'\in U_{x}$. The claim then follows from the compactness of $A(\CC)$.

By the claim, for any fixed $D$, choose a constant $M>0$ such that ${\rm ord}_{z}f^{*}D\leq M$ for all $z\in\CC$. Then we have ${\rm ord}_{z}f^{*}D\leq M\cdot {\rm ord}_{z}f^{*}Y$ for all $z\in f^{-1}(Y)$.  The strengthened version follows by applying Corollary \ref{rm}(2) to $f,Y$ and $\frac{\varepsilon}{M}$.
\end{remark}
\subsection{Uniform estimations}\label{3}
The goal of this subsection is to establish, uniformly for a family of linear entire curves, the two inequalities in Corollary \ref{rm}, which play a fundamental role in the proof of the main theorems in \S\ref{4}. The main results are summarized in Proposition \ref{conclusion}, while the intermediate Proposition \ref{rongchi} will also be used later in Lemma \ref{juli}.

In the following, we always fix a complex abelian variety and a direction $v\in{\rm Lie}(A)$. We also fix a divisor $D$ and a closed subscheme $Y$ of $A$. The variable is the starting point $x\in A(\CC)$ of the linear entire curve $\phi_{(x,v)}$ in $A$. We want to make the inequalities in Corollary \ref{rm} hold for certain linear entire curves $\phi_{(x,v)}$, where $x$ belongs to a fixed subset $Z$ of $A(\CC)$. Note that in general we can not take $Z=A(\CC)$, since the linear leaf $\CF(x,v)$ may be contained in ${\rm Supp}\,D$ or $Y$ for some $x\in A(\CC)$. Instead, we will take $Z$ to be a compact subset of $A(\CC)$ with `good' properties. Most arguments in this subsection are based on compactness. We first prove the following basic lemma.

\begin{lemma}\label{c}
Let $A$ be an abelian variety over $\CC$ and let $Y$ be a proper closed subset of $A$. Let $v\in{\rm Lie}(A)$ be a fixed vector. Let $Z$ be a (Euclidean) closed subset of $A(\CC)$ such that for any point $x\in Z$, the linear leaf $\CF(x,v)\subseteq A$ is not contained in $Y$. Then for any $c>0$, there is a (Euclidean) open subset $V$ in $A(\CC)$ containing $Y$ such that for any point $x\in Z$, any $z_{0}\in\CC$, there exists $0<r<c$ such that $\phi_{(x,v)}(z_{0}+z)\notin V$ for all $|z|=r$.
\end{lemma}
\begin{proof}
Denote by $T\colonequals \ol{\CF(0,v)}$ the Euclidean closure of the linear leaf $\CF(0,v)$ in $A(\CC)$. Then for any $x\in A(\CC)$, we have $\ol{\CF(x,v)}=x+T$. Let $Z'=Z+T\colonequals \{x_{1}+x_{2}\big|\,x_{1}\in Z,x_{2}\in T\}$, which is also a (Euclidean) closed subset of $A(\CC)$. Moreover, for any $x=x_{1}+x_{2}\in Z'$ with $x_{1}\in Z$ and $x_{2}\in T$, we have $\ol{\CF(x,v)}=x+T=x_{1}+T$, which is not contained in $Y$ by our assumption. As a consequence, $\CF(x,v)$ is not contained in $Y$. Therefore, after replacing $Z$ by $Z'=Z+T$, we may assume that $Z=Z+T$.

Now fix a translation-invariant distance function $d$ on $A(\CC)$ and a constant $c>0$. For any $x\in Z$, as $\CF(x,v)$ is not contained in $Y$, there exists $0<r_{x}<c$ such that $\phi_{(x,v)}(z)\notin Y$ for any $|z|=r_{x}$. Choose $\delta_{x}>0$ such that $\BB(Y,\delta_{x})\colonequals \{x'\in A(\CC)\big|\,d(x',Y)<\delta_{x}\}$ is disjoint from $\phi_{(x,v)}(\partial\BD_{r_{x}})$. Then for any $x'\in \BB(x,\frac{\delta_{x}}{2})\colonequals \{x''\in A(\CC)\big|\,d(x'',x)<\frac{\delta_{x}}{2}\}$ and any $|z|=r_{x}$, we have
$$
d(\phi_{(x',v)}(z),Y)\geq d(\phi_{(x,v)}(z),Y)-d(\phi_{(x',v)}(z),\phi_{(x,v)}(z))\geq\delta_{x}-\frac{\delta_{x}}{2}=\frac{\delta_{x}}{2}
$$
In other words, for all $x'\in \BB(x,\frac{\delta_{x}}{2})$, we have $\phi_{(x',v)}(\partial\BD_{r_{x}})\cap \BB(Y,\frac{\delta_{x}}{2})=\emptyset$.

As $Z$ is compact and the open subsets $\BB(x,\frac{\delta_{x}}{2})\,(x\in Z)$ cover $Z$, there are finitely many such subsets $\BB(x_{i},\frac{\delta_{x_{i}}}{2})\,(x_{i}\in Z)$ cover $Z$. Now let $V=\BB(Y,\min_{1\leq i\leq m}\frac{\delta_{x_{i}}}{2})$. Then for any point $x\in Z$, any $z_{0}\in\CC$, we have $x'=\phi_{(x,v)}(z_{0})\in x+T\subseteq Z$ by our assumption. Choose $1\leq i\leq m$ such that $x'\in \BB(x_{i},\frac{\delta_{x_{i}}}{2})$. Then by the previous paragraph, we have $0<r_{x_{i}}<c$ and $\phi_{(x,v)}(z_{0}+z)=\phi_{(x',v)}(z)\notin V$ for all $|z|=r_{x_{i}}$. This proves the lemma.
\end{proof}

Now we establish the uniform estimate for Corollary \ref{rm}(1).

\begin{proposition}\label{rongchi}
Let $A$ be an abelian variety over $\CC$, which is also a closed subvariety of some projective variety $X$ over $\CC$. Let $D$ be an effective Cartier divisor on $X$ whose support does not contain $A$. Let $v\in{\rm Lie}(A)$ be a nonzero element. Let $Z$ be a (Euclidean) closed subset of $A(\CC)$ such that for any point $x\in Z$, the linear leaf $\CF(x,v)\subseteq A$ is not contained in ${\rm Supp}\, D$. Fix a distance function $d$ on $X(\CC)$. Then for any $c>0$, there is a constant $\delta_{c}>0$ satisfying the following condition: For any $r>0$, any point $x\in Z$ and any holomorphic map $\phi\colon \BD_{r+c}\to X$, if $d(\phi(z), \phi_{(x,v)}(z))<\delta_{c}$ for all $z\in\BD_{r+c}$, then
$$
\sum\limits_{z\in\BD_{r+c}}{\rm ord}_{z}\phi^{*}D\geq\sum\limits_{z\in\BD_{r}}{\rm ord}_{z}\phi_{(x,v)}^{*}D.
$$
\end{proposition}

\begin{proof}
Choose a finite affine cover $\{U_{i}\}_{i=1}^{m}$ of $X$ such that $D\big|_{U_{i}}$ is generated by a section $f_{i}\in\Gamma(U_{i},\CK^{*})$ for each $1\leq i\leq m$. Then there is a constant $\delta_{1}>0$ such that for any $x\in X(\CC)$, the set $\BB(x,\delta_{1})\colonequals \{x'\in X(\CC)\big|\, d(x,x')<\delta_{1}\}$ is contained in some $U_{i}$. In fact, if this is not true, there is a sequence $\{x_{n}\}_{n\geq1}\subseteq X(\CC)$ such that $\BB(x_{n},\frac{1}{n})$ is not contained in any $U_{i}$. Let $x_{\infty}$ be a limit point of $\{x_{n}\}_{n\geq1}$. Choose $1\leq i\leq m$ and $\delta_{0}>0$ such that $\BB(x_{\infty},\delta_{0})\subseteq U_{i}$. Then for $n>\frac{2}{\delta_{0}}$ with $d(x_{n},x_{\infty})<\frac{\delta_{0}}{2}$, we have (by triangle inequality)
$$
\BB(x_{n},\frac{1}{n})\subseteq \BB(x_{\infty},\frac{1}{n}+d(x_{n},x_{\infty}))\subseteq \BB(x_{\infty},\frac{\delta_{0}}{2}+\frac{\delta_{0}}{2})=\BB(x_{\infty},\delta_{0})\subseteq U_{i}.
$$
which contradicts our assumption that $\BB(x_{n},\frac{1}{n})$ is not contained in $U_{i}$.

Let $Z_{i}\subseteq U_{i}$ be the set of all points $x\in U_{i}(\CC)$ such that $\BB(x,\delta_{1})\subseteq U_{i}$ (or equivalently, $Z_{i}=\{x\in X(\CC)\big|\,d(x,X\backslash U_{i})\geq\delta_{1}\}$). Then $Z_{i}$ is (Euclidean) closed in $U_{i}(\CC)$ and $\{Z_{i}\}_{i=1}^{m}$ covers $X(\CC)$ by our assumption in the previous paragraph. Let $Z_{i}'\colonequals \{x\in X(\CC)\big|\,d(x,Z_{i})\leq\frac{\delta_{1}}{2}\}$. Then $Z_{i}'$ is a (Euclidean) compact subset of $U_{i}(\CC)$ containing $Z_{i}$.

By abuse of notations, we identify $D$ with ${\rm Supp}\,D$ so that we can view it as a subset of $X$. Choose $0<c'<c$ such that for any point $x\in A(\CC)$ and any $z\in\BD_{c'}$, we have $d(x,\phi_{(x,v)}(z))<\frac{\delta_{1}}{4}$. As the linear leaf $\CF(x,v)$ is not contained in $D$ for all $x\in Z$, by Lemma \ref{c}, there is an (Euclidean) open subset $V$ of $X$ containing $D$ such that for any point $x\in Z$ and any $z_{0}\in\CC$, there exists $0<r_{x,z_{0}}<c'$ such that $\phi_{(x,v)}(z_{0}+z)\notin V$ for all $|z|=r_{x,z_{0}}$. Choose $\delta_{2}>0$ such that $\BB(D,\delta_{2})\colonequals \{x\in X(\CC)\big|\, d(x,D)<\delta_{2}\}$ is contained in $V$.

As $D$ is effective, we can view $f_{i}$ as a holomorphic function on $U_{i}$ for each $1\leq i\leq m$. Consider its restriction to the compact subset $Z_{i}'\backslash \BB(D,\frac{\delta_{2}}{2})$, which is identically nonzero. Therefore $|f_{i}(x)|\,(x\in Z_{i}'\backslash \BB(D,\frac{\delta_{2}}{2}))$ has a minimum $c_{i}>0$. On the other hand, by the uniform continuity, there is a constant $\delta_{3,i}>0$ such that for all $x,x'\in Z_{i}'\backslash \BB(D,\frac{\delta_{2}}{2})$ with $d(x,x')<\delta_{3,i}$, we have $|f_{i}(x)-f_{i}(x')|<c_{i}$. Now let $\delta_{3}\colonequals\min_{1\leq i\leq m}\delta_{3,i}$ and let $\delta_{c}=\min\{\frac{\delta_{1}}{4},\frac{\delta_{2}}{3},\delta_{3}\}$.

Fix an $r>0$, a point $x\in Z$ and a holomorphic map $\phi\colon \BD_{r+c}\to X$ such that $d(\phi(z), \phi_{(x,v)}(z))<\delta_{c}$ for all $z\in\BD_{r+c}$. For each $z\in\ol{\BD}_{r}$, consider the open ball $\BD(z,r_{x,z})\colonequals \{z'\in\CC\big|\,|z-z'|<r_{x,z}\}$. As all these open balls cover $\ol{\BD}_{r}$, there exist finitely many such balls $\BD(z_{j},r_{x,z_{j}})_{j=1}^{k}\,(z_{j}\in\ol{\BD_{r}})$ that cover $\ol{\BD_{r}}$.

Now on any intersection $\BD_{j_{1},\dots,j_{l}}\colonequals\cap_{i=1}^{l}\BD(z_{j_{i}},r_{x,z_{j_{i}}})\subseteq \BD_{r+c}$ (where $1\leq j_{1}<\dots<j_{l}\leq k$, denote by $\Gamma_{j_{1},\dots,j_{l}}$ its boundary, which is contained in the union $\cup_{i=1}^{l}\partial\BD(z_{j_{i}},r_{x,z_{j_{i}}})$. By our choice, for each $1\leq j\leq k$, $d(\phi_{(x,v)}(z_{j}+z),D)\geq\delta_{2}$ for all $|z|= r_{x,z_{j}}$. Or equivalently, $d(\phi_{(x,v)}(z),D)\geq\delta_{2}$ for all $z\in\partial\BD(z_{j_{i}},r_{x,z_{j_{i}}})$.  Hence we have $d(\phi_{(x,v)}(\Gamma_{j_{1},\dots,j_{l}}),D)\geq \delta_{2}$. Moreover, as $d(\phi(z), \phi_{(x,v)}(z))<\delta_{c}$ for all $z\in\BD_{r+c}$, we have
$$
d(\phi(z), D)\geq d(\phi_{(x,v)}(z),D)- d(\phi(z), \phi_{(x,v)}(z))\geq\delta_{2}-\delta_{c}\geq\frac{2\delta_{2}}{3}>\frac{\delta_{2}}{2}
$$
for all  $z\in\Gamma_{j_{1},\dots,j_{l}}$.

Denote $x'=\phi_{(x,v)}(z_{j_{1}})$. Choose $1\leq i\leq m$ such that $x'\in Z_{i}$. Then since $r_{x,z_{j_{1}}}<c'$, we have $d(x', \phi_{(x,v)}(z_{j_{1}}+z))=d(x',\phi_{(x',v)}(z))<\frac{\delta_{1}}{4}$ for all $|z|=r_{x,z_{j_{1}}}$. Or equivalently, $ d(x',\phi_{(x,v)}(z))<\frac{\delta_{1}}{4}$ for all $z\in \ol{\BD(z_{j_{1}},r_{x,z_{j_{1}}})}$. Hence
$$
d(x',\phi(z))\leq d(x',\phi_{(x,v)}(z))+d(\phi_{(x,v)}(z),\phi(z))<\frac{\delta_{1}}{4}+\delta_{c}\leq \frac{\delta_{1}}{2}
$$
for all $z\in \ol{\BD(z_{j_{1}},r_{x,z_{j_{1}}})}$. This shows that both $\phi_{(x,v)}(\ol{\BD(z_{j_{1}},r_{x,z_{j_{1}}})})$ and $\phi(\ol{\BD(z_{j_{1}},r_{x,z_{j_{1}}})})$ are contained in $Z_{i}'$. As $\Gamma_{j_{1},\dots,j_{l}}\subseteq \ol{\BD(z_{j_{1}},r_{x,z_{j_{1}}})}$, for all $z\in \Gamma_{j_{1},\dots,j_{l}}$, we have $\phi(z),\phi_{(x,v)}(z)\in Z'_{i}\backslash \BB(D,\frac{\delta_{2}}{2})$ and $d(\phi(z),\phi_{(x,v)}(z))<\delta_{c}\leq\delta_{3}$, and thus
$$
|f_{i}(\phi(z))-f_{i}(\phi_{(x,v)}(z))|<c_{i}
$$
and
$$
|f_{i}(\phi_{(x,v)}(z))|\geq c_{i}.
$$
Hence by Rouch\'{e}'s theorem, the holomorphic functions $f_{i}\circ\phi$ and $f_{i}\circ\phi_{(x,v)}$ have the same number of zeroes on the intersection $\BD_{j_{1},\dots,j_{l}}$. Equivalently, we have

$$
\sum\limits_{z\in \BD_{j_{1},\dots,j_{l}}}{\rm ord}_{z}\phi^{*}D=\sum\limits_{z\in \BD_{j_{1},\dots,j_{l}}}{\rm ord}_{z}\phi_{(x,v)}^{*}D.
$$

Since $1\leq j_{1}<\dots<j_{l}\leq k$ are arbitrary, we deduce from the inclusion-exclusion principle that
$$
\sum\limits_{z\in W}{\rm ord}_{z}\phi^{*}D=\sum\limits_{z\in W}{\rm ord}_{z}\phi_{(x,v)}^{*}D.
$$
where $W=\cup_{j=1}^{k}\BD(z_{j},r_{x,z_{j}})$ is an open subset of $\BD_{r+c}$ containing $\BD_{r}$. This implies the proposition.
\end{proof}

\begin{corollary}\label{rm-}
Let $A$ be an abelian variety over $\CC$ and let $D$ be an ample effective Cartier divisor on $A$. Let $v\in{\rm Lie}(A)$ be a nonzero element. Let $Z$ be a (Euclidean) closed subset of $A(\CC)$ such that for any point $x\in Z$, the linear leaf $\CF(x,v)\subseteq A$ is not contained in ${\rm Supp}\, D$. Fix $\varepsilon>0$ and a smooth Hermitian metric $\left\|\cdot\right\|$ on $L=\CO_{A}(D)$. Then there is a constant $r_{0}>0$ such that for each $r>r_{0}$ and each point $x\in Z$, we have
$$
\sum\limits_{z\in\BD_{r}}{\rm ord}_{z}\phi_{(x,v)}^{*}D\geq (1-\varepsilon)\int_{\BD_{r}}\phi_{(x,v)}^{*}c_{1}(L,\left\|\cdot\right\|).
$$
\end{corollary}

\begin{proof}
By abuse of notation, we view $D$ as a subset of $A$ by identifying it with ${\rm Supp}\,D$. Denote by $G$ the Zariski closure of the linear leaf $\CF(0,v)$ in $A$. Then for any $x\in A(\CC)$, the leaf $\CF(x,v)$ is contained in $D$ if and only if $x+G\subseteq D$. Let $S\subseteq A$ be the set of all $x\in D$ such that $\CF(x,v)\subseteq D$. Then $S$ is Zariski closed and $S\cap Z=\emptyset$. Hence there exists an (Euclidean) open neighbourhood $V$ of $S(\CC)$ which is disjoint from $Z$ with $A(\CC)\backslash V$ connected. After replacing $Z$ by $A(\CC)\backslash V$, we may assume that $Z$ is connected.

Now fix a translation-invariant distance function $d$ on $A(\CC)$ and a point $x\in Z$. Applying Corollary \ref{rm}(1) to $\phi_{(x,v)}\colon \CC\to (x+G)$ and the ample divisor $D\big|_{(x+G)}$ on $(x+G)$, we see that

\begin{align}\label{rm-1}
\sum\limits_{z\in\BD_{r}}{\rm ord}_{z}\phi_{(x,v)}^{*}D\geq (1-\frac{\varepsilon}{2})\int_{\BD_{r}}\phi_{(x,v)}^{*}c_{1}(L,\left\|\cdot\right\|)
\end{align}
for $r$ large enough.

Applying Proposition \ref{rongchi} to $X=A,D,Z,v$ and $c=1$, we get a constant $\delta=\delta_{1}>0$ satisfying the conclusion of the proposition. By compactness and connectedness of $Z$, there is an integer $C>0$ such that for any point $x'\in Z$, there exists $n\leq C$ and a sequence of points $x=x_{0},x_{1},\dots,x_{n}=x'$ with $d(x_{i},x_{i+1})<\delta$ for all $0\leq i\leq n-1$. As $d(\phi_{(x,v)}(z),\phi_{(x',v)}(z))=d(x,x')$ for all $x'\in A(\CC)$ and all $z\in\CC$, by Proposition \ref{rongchi}, for each $r>0$ and each $0\leq i\leq n-1$, we have
\begin{align*}
\sum\limits_{z\in\BD_{r+1}}{\rm ord}_{z}\phi_{(x_{i+1},v)}^{*}D\geq\sum\limits_{z\in\BD_{r}}{\rm ord}_{z}\phi_{(x_{i},v)}^{*}D.
\end{align*}
Therefore we deduce that
\begin{align}\label{gun}
\sum\limits_{z\in\BD_{r+C}}{\rm ord}_{z}\phi_{(x',v)}^{*}D\geq\sum\limits_{z\in\BD_{r+n}}{\rm ord}_{z}\phi_{(x',v)}^{*}D=\sum\limits_{z\in\BD_{r+n}}{\rm ord}_{z}\phi_{(x_{n},v)}^{*}D\geq \sum\limits_{z\in\BD_{r+n-1}}{\rm ord}_{z}\phi_{(x_{n-1},v)}^{*}D\notag\\
\geq\dots\geq\sum\limits_{z\in\BD_{r}}{\rm ord}_{z}\phi_{(x_{0},v)}^{*}D=\sum\limits_{z\in\BD_{r}}{\rm ord}_{z}\phi_{(x,v)}^{*}D.
\end{align}
for any $r>0$ and any $x'\in Z$.

Let $\omega$ be the unique translation-invariant $(1,1)$-form representing the Chern class $c_{1}(L)$. Then there is a constant $c_{1}>0$ such that
\begin{align*}
\int_{\BD_{r}}\phi_{(x,v)}^{*}\omega=c_{1}r^{2}
\end{align*}
for any $x\in A(\CC)$ and any $r>0$. Let $\alpha$ be a 1-form on $A$ such that $c_{1}(D,\left\|\cdot\right\|)-\omega={\rm d}\alpha$. Then for any $x\in A(\CC)$ and any $r>0$,
\begin{align*}
\left|\int_{\BD_{r}}\phi_{(x,v)}^{*}c_{1}(D,\left\|\cdot\right\|)-\int_{\BD_{r}}\phi_{(x,v)}^{*}\omega\right|=\left|\int_{\partial\BD_{r}}\phi_{(x,v)}^{*}\alpha\right|
\end{align*}
By compactness, there is a constant $c_{2}>0$ (independent of $x$ and $r$) such that the right hand side is bounded above by $c_{2}r$. Therefore we have
\begin{align}\label{omega}
c_{1}r^{2}-c_{2}r\leq\int_{\BD_{r}}\phi_{(x,v)}^{*}c_{1}(D,\left\|\cdot\right\|)\leq c_{1}r^{2}+c_{2}r
\end{align}
for all $x\in A(\CC)$ and all $r>0$.

Now take $r_{0}>C$ such that $c_{1}(r-C)^{2}-c_{2}(r-C)\geq (1-\frac{\varepsilon}{2})(c_{1}r^{2}+c_{2}r)$ for all $r>r_{0}$. Enlarge $r_{0}$ such that (\ref{rm-1}) holds for all $r>r_{0}-C$. Combining (\ref{rm-1}), (\ref{gun}) and (\ref{omega}), we get for all $x'\in Z$ and all $r>r_{0}$,
\begin{align*}
\sum\limits_{z\in\BD_{r}}{\rm ord}_{z}\phi_{(x',v)}^{*}D&\geq\sum\limits_{z\in\BD_{r-C}}{\rm ord}_{z}\phi_{(x,v)}^{*}D\geq (1-\frac{\varepsilon}{2})\int_{\BD_{r-C}}\phi_{(x,v)}^{*}c_{1}(L,\left\|\cdot\right\|)\\
&\geq (1-\frac{\varepsilon}{2})(c_{1}(r-C)^{2}-c_{2}(r-C))\geq (1-\frac{\varepsilon}{2})^{2}(c_{1}r^{2}+c_{2}r)\\
&\geq (1-\varepsilon)\int_{\BD_{r}}\phi_{(x',v)}^{*}c_{1}(L,\left\|\cdot\right\|). \tag*{\qedhere}
\end{align*}
\end{proof}

Then we turn to uniform estimate for Proposition \ref{rm}(2).

\begin{proposition}\label{Y}
Let $A$ be an abelian variety over $\CC$ and let $Y$ be a closed subscheme of $A$ of codimension greater than $1$. Let $v\in{\rm Lie}(A)$ be a nonzero element and let $G$ be the Zariski closure of the linear leaf $\CF(0,v)$ in $A$ (which is an abelian subvariety of $A$). Let $Z$ be a (Euclidean) closed subset of $A(\CC)$ such that for any point $x\in Z$, $(x+G)\cap Y$ has codimension greater than $1$ in $(x+G)$. Fix a constant $\varepsilon>0$, an ample line bundle $L$ on $A$ and a smooth Hermitian metric $\left\|\cdot\right\|$ on $L$. Then there is a constant $R_{0}>0$ such that for any $R>R_{0}$ and any point $x\in Z$, we have
$$
\sum\limits_{z\in\BD_{R}}{\rm ord}_{z}\phi_{(x,v)}^{*}Y\leq \varepsilon \int_{\BD_{R}}\phi_{(x,v)}^{*}c_{1}(L,\left\|\cdot\right\|).
$$
Furthermore, for any effective divisor $D$ on $A$ whose support contains $Y$ but does not contain any $(x+G)\,(x\in Z)$, we can further ensure that
$$
\sum\limits_{z\in\BD_{R}\cap\phi_{(x,v)}^{-1}(Y)}{\rm ord}_{z}\phi_{(x,v)}^{*}D\leq \varepsilon \int_{\BD_{R}}\phi_{(x,v)}^{*}c_{1}(L,\left\|\cdot\right\|)
$$
for all $R>R_{0}$.
\end{proposition}
To prove the proposition, we need two lemmas. The first consider the (semi)continuity of counting functions as follows.
\begin{lemma}\label{banlianxu}
Let $X$ be a projective variety over $\CC$. Let $Y$ be a closed subscheme of $X$. Let $R>0$ and let $f_{n}\colon\BD_{R}\to X$ be a sequence of holomorphic maps. Assume that $f_{n}$ converges uniformly to a holomorphic map $f\colon\BD_{R}\to X$ with $f(\BD_{R})\nsubseteq Y$. Then for all but countably many $0<r<R$ (which depends only on $f$ and $Y$, but not on $f_{n}$), we have
$$
\sum_{z\in\BD_{r}}{\rm ord}_{z}f_{n}^{*}Y\leq \sum_{z\in\BD_{r}}{\rm ord}_{z}f^{*}Y
$$
for $n$ large enough.
\end{lemma}
\begin{proof}
Write $Y=D_{1}\cap\dots\cap D_{l}$, where $D_{j}$ are effective Cartier divisors on $X$. For simplicity, we identify $D_{j}$ with its support. Since $f(\BD_{R})\nsubseteq Y$, we have $f(\BD_{R})\nsubseteq D_{j}$ for at least one $j$.

Fix $0<r<R$ such that for each $1\leq j\leq l$, if $f(\BD_{R})\nsubseteq D_{j}$, then we have $f(z)\notin D_{j}$ for all $z\in\partial\BD_{r}$. Note that all but countably many $0<r<R$ satisfy the condition above. We prove the inequality for this $r$.

In fact, for each $z\in\ol{\BD}_{r}$, choose an open disc $U_{z}\subseteq \BD_{R}$ containing $z$ such that for any $1\leq j\leq l$ with $f(\BD_{R})\nsubseteq D_{j}$ and any $z'\in\ol{U_{z}}\backslash\{z\}$, we have $f(z')\notin D_{j}$. Then $U_{z}\,(z\in\BD_{r})$ covers $\ol{\BD}_{r}$, which is compact. Choose a finite subcover $U_{z_{k}}\,(1\leq k\leq m)$ of it. By Rouch\'{e}'s theorem, for each $k$, we have
$$
\sum_{z\in U_{z_{k}}}{\rm ord}_{z}f_{n}^{*}Y\leq\sum_{z\in U_{z_{k}}}{\rm ord}_{z}f_{n}^{*}D_{j}=\sum_{z\in U_{z_{k}}}{\rm ord}_{z}f^{*}D_{j}={\rm ord}_{z_{k}}f^{*}D_{j}
$$
for all $1\leq j\leq l$ with $f(\BD_{R})\nsubseteq D_{j}$ and all $n$ large enough (depending on $k,j$). Note that this inequality also holds (in the obvious sense) for all $1\leq j\leq l$ with $f(\BD_{R})\subseteq D_{j}$. Therefore we have
$$
\sum_{z\in U_{z_{k}}}{\rm ord}_{z}f_{n}^{*}Y\leq \min\limits_{1\leq j\leq l}{\rm ord}_{z_{k}}f^{*}D_{j}={\rm ord}_{z_{k}}f^{*}Y
$$
for all $1\leq k\leq m$ and all $n$ large enough (depending on $k$). In conclusion, we have
$$
\sum_{z\in\BD_{r}}{\rm ord}_{z}f_{n}^{*}Y\leq\sum\limits_{k=1}^{m}\sum_{z\in U_{z_{k}}}{\rm ord}_{z}f_{n}^{*}Y\leq\sum\limits_{k=1}^{m}{\rm ord}_{z_{k}}f^{*}Y=\sum_{z\in\BD_{r}}{\rm ord}_{z}f^{*}Y
$$
for $n$ large enough. Here the last equality follows from our choice of $U_{z_{k}}$ and the assumption that $f(z)\notin Y$ for all $z\in\partial\BD_{r}$.
\end{proof}

\begin{remark}\label{blx}
From the above proof we can deduce the following fact (with same conditions as in Lemma \ref{banlianxu}): For any Euclidean closed subset $Y$ of $X$, any effective divisor $D$ containing $Y$ and any $0<r<R$ with $f(\partial\BD_{r})\cap D=\emptyset$, we have
\begin{align}\label{DY}
\sum\limits_{z\in\BD_{r}\cap f_{n}^{-1}(Y)}{\rm ord}_{z}f_{n}^{*}D\leq \sum\limits_{z\in\BD_{r}\cap f^{-1}(Y)}{\rm ord}_{z}f^{*}D
\end{align}
for $n$ large enough. Moreover, we can prove that there is a Euclidean closed subset $W$ of $D$ which is disjoint from $Y$ such that
\begin{align}\label{WDY}
\sum\limits_{z\in\BD_{r}\cap f_{n}^{-1}(W)}{\rm ord}_{z}f_{n}^{*}D= \sum\limits_{z\in\BD_{r}\cap f^{-1}(W)}{\rm ord}_{z}f^{*}D=\sum\limits_{z\in\BD_{r}\backslash f^{-1}(Y)}{\rm ord}_{z}f^{*}D
\end{align}
for $n$ large enough.

In fact, similarly as above, choose an open disc $U_{z}\subseteq\BD_{R}$ for each $z\in\ol{\BD}_{r}$ such that $f(z')\notin D$ for all $z'\in\ol{U}_{z}\backslash\{z\}$. Moreover, choose a Euclidean closed subset $W_{0}$ of $D$ which is disjoint from $Y$ such that $f(\BD_{r})\cap D\subseteq W_{0}\cup Y$. As $Y$ is also closed, we can choose open neighbourhoods $W',Y'$ of $W_{0},Y$ in $X$ such that $\ol{W'}\cap\ol{Y'}=\emptyset$.

Now if $f(z)\in Y$, shrinking $U_{z}$ if necessary, we may assume that the Euclidean closure of $f({U}_{z})$ in $X$ is contained in $Y'$. Similarly, if $f(z)\in W_{0}$, we may shrink $U_{z}$ such that the Euclidean closure of $f({U}_{z})$ in $X$ is contained in $W'$. Take finitely many distinct discs $U_{z_{1}},\dots,U_{z_{m}}$ which cover $\ol{\BD}_{r}$. Then Rouch\'{e}'s theorem implies that
$$
\sum_{z\in U_{z_{k}}}{\rm ord}_{z}f_{n}^{*}D=\sum_{z\in U_{z_{k}}}{\rm ord}_{z}f^{*}D={\rm ord}_{z_{k}}f^{*}D
$$
for each $k$ and $n$ large enough (depending on $k$), and
$$
\sum_{z\in U_{z_{k}}\cap U_{z_{l}}}{\rm ord}_{z}f_{n}^{*}D=\sum_{z\in U_{z_{k}}\cap U_{z_{l}}}{\rm ord}_{z}f^{*}D=0
$$
for any $1\leq k<l\leq m$ and $n$ large enough (depending on $k,l$). Moreover, if $f(z_{k})\in Y$ (resp. $f(z_{k})\in W_{0}$), then the Euclidean closure of $f(U_{z_{k}})$ in $X$ is contained in $Y'$ (resp. in $ W'$). Thus for $n$ large enough (depending on $r$ and $k$), the Euclidean closure of $f_{n}(U_{z_{k}})$ in $X$ is also contained in $Y'$ (resp. in $W'$). In particular, for $n$ large enough, we have $f_{n}(\BD_{r})\cap D\subseteq W'\sqcup Y'\subseteq \ol{W'}\sqcup\ol{Y'}$. Thus we have
\begin{align*}
\sum\limits_{z\in\BD_{r}\cap f_{n}^{-1}(Y')}{\rm ord}_{z}f_{n}^{*}D &= \sum\limits_{k=1}^{m}\sum\limits_{z\in U_{z_{k}}\cap f_{n}^{-1}(Y')}{\rm ord}_{z}f_{n}^{*}D=\sum\limits_{\substack{1\leq k\leq m\\ f(z_{k})\in Y}}\sum\limits_{z\in U_{z_{k}}\cap f_{n}^{-1}(Y')}{\rm ord}_{z}f_{n}^{*}D\\
&=\sum\limits_{\substack{1\leq k\leq m\\ f(z_{k})\in Y}}\sum\limits_{z\in U_{z_{k}}}{\rm ord}_{z}f_{n}^{*}D=\sum\limits_{\substack{1\leq k\leq m\\ f(z_{k})\in Y}}{\rm ord}_{z_{k}}f^{*}D= \sum\limits_{z\in\BD_{r}\cap f^{-1}(Y)}{\rm ord}_{z}f^{*}D.
\end{align*}
for $n$ large enough. This proves (\ref{DY}) since $Y\subseteq Y'$. Similarly, we have
\begin{align*}
\sum\limits_{z\in\BD_{r}\cap f_{n}^{-1}(W')}{\rm ord}_{z}f_{n}^{*}D = \sum\limits_{z\in\BD_{r}\cap f^{-1}(W_{0})}{\rm ord}_{z}f^{*}D=\sum\limits_{z\in\BD_{r}\backslash f^{-1}(Y)}{\rm ord}_{z}f^{*}D.
\end{align*}
for $n$ large enough. Let $W=D\cap\ol{W'}$, which is Euclidean closed in $D$ disjoint from $Y$. By our assumption, we have $f_{n}^{-1}(W'\cap D)=f_{n}^{-1}(W\cap D)$ for $n$ large enough. Hence the left hand side above is exactly $\sum_{z\in\BD_{r}\cap f_{n}^{-1}(W)}{\rm ord}_{z}f_{n}^{*}D$, this proves (\ref{WDY}).
\end{remark}

The next lemma is a simple application of Besicovitch covering theorem. We will see in \S\ref{fenbu} (Lemma \ref{refine--}) and \S\ref{4.3} (Lemma \ref{refine}) for other refined versions.

\begin{lemma}\label{besicovitch}
There is an absolute constant $C>0$ with the following property: Let $\Omega$ be a bounded subset in $\CC$. Let $\mu_{1},\mu_{2}$ be two measures on an open neighbourhood $U$ of $\Omega$. Assume that for any point $x\in\Omega$, there is an open disc $\BD_{x}\subseteq U$ centred at $x$ such that $\mu_{1}(\BD_{x})\leq\mu_{2}(\BD_{x})$. Then we have $\mu_{1}(\Omega)\leq C\cdot\mu_{2}(U)$.
\end{lemma}
\begin{proof}
Take $C$ to be the Besicovitch constant in dimension 2 (which actually equals 19). Then since $\cup_{x\in\Omega}\BD_{x}$ is a Besicovitch cover of $\Omega$, the Besicovitch covering theorem implies that there is a subset $\Omega'\subseteq\Omega$ such that the union $\cup_{x\in\Omega'}\BD_{x}$ contains $\Omega$, and for each point $y\in U$, the number of $x\in\Omega'$ with $y\in\BD_{x}$ is at most $C$. Therefore we have $\mu_{1}(\Omega)\leq \sum_{x\in\Omega'}\mu_{1}(\BD_{x})\leq \sum_{x\in\Omega'}\mu_{2}(\BD_{x})\leq C\cdot\mu_{2}(U)$.
\end{proof}

\begin{proof}[Proof of Proposition \ref{Y}]
After replacing $Z$ by $Z+G(\CC)$, which is also a Euclidean closed subset of $A(\CC)$ satisfying the conditions in the proposition, we may assume that $Z=Z+G(\CC)$. Note that $G\neq 0$, and thus there always exists an effective divisor $D$ on $A$ whose support contains $Y$ such that $(x+G)\nsubseteq D$ for all $x\in Z$. Take any such $D$. We only need to show the second statement.

First consider the case where $\left\|\cdot\right\|$ is the translation-invariant Hermitian metric on $L$. In this case there is a constant $c>0$ such that
\begin{align*}
\int_{\BD_{r}}\phi_{(x,v)}^{*}c_{1}(L,\left\|\cdot\right\|)= cr^{2}
\end{align*}
for any point $x\in A(\CC)$ and any $r>0$.

Let $C>0$ be the Besicovitch constant in dimension 2. We first claim that there is a constant $r_{0}>0$ such that for any point $x\in Z$, there exists $0<r_{x}<r_{0}$ such that
\begin{align*}
\sum\limits_{z\in\BD_{r_{x}}\cap\phi_{(x,v)}^{-1}(Y)}{\rm ord}_{z}\phi_{(x,v)}^{*}D\leq \frac{\varepsilon}{4C} \int_{\BD_{r_{x}}}\phi_{(x,v)}^{*}c_{1}(L,\left\|\cdot\right\|).
\end{align*}
In fact, for each $x\in Z$, the intersection $(x+G)\cap Y$ has codimension at least 2 in $(x+G)$ and $(x+G)\nsubseteq D$. Thus by Remark \ref{jq}, there exists $r'_{x}>0$ such that
\begin{align*}
\sum\limits_{z\in\BD_{r}\cap\phi_{(x,v)}^{-1}(Y)}{\rm ord}_{z}\phi_{(x,v)}^{*}D\leq \frac{\varepsilon}{8C} \int_{\BD_{r}}\phi_{(x,v)}^{*}c_{1}(L,\left\|\cdot\right\|).
\end{align*}
holds for all $r\geq r'_{x}$. After enlarging $r'_{x}$, we may assume that $r'_{x}$ does not lie in the countably many exceptions in Lemma \ref{banlianxu}. Then by Lemma \ref{banlianxu} and the continuity (in $x$) of the right hand side, there is a Euclidean open neighbourhood $U_{x}$ of $x$ in $A$ such that for any $x'\in U_{x}$, we have
\begin{align*}
\sum\limits_{z\in\BD_{r'_{x}}\cap\phi_{(x',v)}^{-1}(Y)}{\rm ord}_{z}\phi_{(x',v)}^{*}D\leq \frac{\varepsilon}{4C} \int_{\BD_{r'_{x}}}\phi_{(x',v)}^{*}c_{1}(L,\left\|\cdot\right\|).
\end{align*}
Now all $U_{x}\,(x\in Z)$ cover $Z$, which is compact. Thus there are finitely many points $x_{1},\dots,x_{m}\in Z$ such that $\cup_{i=1}^{m}U_{x_{i}}$ containing $Z$. Now take $r_{0}=\max_{1\leq i\leq m}r'_{x_{i}}+1$; the claim is true.

Note that for any $x\in Z$ and any $z_{0}\in\CC$, as $x'=\phi_{(x,v)}(z_{0})\in x+G\subseteq Z$ and $\phi_{(x,v)}(z_{0}+z)=\phi_{(x',v)}(z)$ for all $z\in\CC$, by our claim above, there exists $0<r_{x,z_{0}}<r_{0}$ such that
\begin{align*}
&\sum\limits_{z\in\BD(z_{0},r_{x,z_{0}})\cap\phi_{(x,v)}^{-1}(Y)}{\rm ord}_{z}\phi_{(x,v)}^{*}D=\sum\limits_{z\in\BD_{r_{x,z_{0}}}\cap\phi_{(x',v)}^{-1}(Y)}{\rm ord}_{z}\phi_{(x',v)}^{*}D\\
&\leq \frac{\varepsilon}{4C}\int_{\BD_{r_{x,z_{0}}}}\phi_{(x',v)}^{*}c_{1}(L,\left\|\cdot\right\|)= \frac{\varepsilon}{4C} \int_{\BD(z_{0},r_{x,z_{0}})}\phi_{(x,v)}^{*}c_{1}(L,\left\|\cdot\right\|).
\end{align*}
Now for any point $x\in Z$ and any $R>r_{0}$, the union $\cup_{z_{0}\in\BD_{R}}\BD(z_{0},r_{x,z_{0}})$ is contained in $\BD_{2R}$. Set
$$
\mu_{1,R}\colonequals \sum\limits_{z\in\BD_{2R}\cap\phi_{(x,v)}^{-1}(Y)}{\rm ord}_{z}\phi_{(x,v)}^{*}D\cdot\delta_{z}
$$
and
$$
\mu_{2,R}\colonequals \frac{\varepsilon}{4C}\phi_{(x,v)}^{*}c_{1}(L,\left\|\cdot\right\|)\Big|_{\BD_{2R}},
$$
which are measures on $\BD_{2R}$. Applying Lemma \ref{besicovitch} to $\mu_{1,R},\mu_{2,R}$, $\Omega=\BD_{R}$ and $U=\BD_{2R}$, we get
\begin{align*}
\sum\limits_{z\in\BD_{R}\cap\phi_{(x,v)}^{-1}(Y)}{\rm ord}_{z}\phi_{(x,v)}^{*}D\leq C\frac{\varepsilon}{4C}\int_{\BD_{2R}}\phi_{(x,v)}^{*}c_{1}(L,\left\|\cdot\right\|)&\leq C\frac{\varepsilon}{4C}(2R)^{2}=c\varepsilon R^{2}\\
&=\varepsilon\int_{\BD_{R}}\phi_{(x,v)}^{*}c_{1}(L,\left\|\cdot\right\|).
\end{align*}
This completes the proof when $\left\|\cdot\right\|$ is invariant under translations. For general Hermitian metric, just note that there is a constant $c'>0$ such that
$$
cr^{2}-c'r\leq \int_{\BD_{r}}\phi_{(x,v)}^{*}c_{1}(L,\left\|\cdot\right\|)\leq cr^{2}+c'r
$$
for all $x\in A(\CC)$ and $r>0$ (see the argument at the end of the proof of Proposition \ref{rm-}). Therefore the right hand side is sufficiently close to $cr^{2}$ when $r$ is large enough. From this we can deduce the general case from the special case above.
\end{proof}

Now we summarized the main estimates in this subsection as follows. Part 1 will be used in \S\ref{fenbu} and \S\ref{4.3}, part 2 will be used in \S\ref{4.3} and part 3 will be used in \S\ref{4.4}.

\begin{proposition}\label{conclusion}
Let $A$ be an abelian variety over $\CC$, which is also a closed subvariety of some projective variety $X$ over $\CC$. Let $D$ be an effective Cartier divisor on $X$ whose support does not contain $A$ such that the restriction of $D$ to $A$ is ample. Let $v\in{\rm Lie}(A)$ be a nonzero element. Let $Z$ be a (Euclidean) closed subset of $A(\CC)$ such that for any point $x\in Z$, the linear leaf $\CF(x,v)\subseteq A$ is not contained in ${\rm Supp}\, D$. Fix a distance function $d$ on $X(\CC)$, a smooth Hermitian metric $\left\|\cdot\right\|$ on $L=\CO_{X}(D)$ and a constant $\varepsilon>0$. Then there are constants $R>0, \delta>0$ satisfying the following condition: For any point $x\in Z$ and any holomorphic map $\phi\colon \BD_{R+1}\to X$, if $d(\phi(z), \phi_{(x,v)}(z))<\delta$ for all $z\in\BD_{R+1}$, then
$$
\sum\limits_{z\in\BD_{R}}{\rm ord}_{z}\phi^{*}D\geq(1-\varepsilon)\int_{\BD_{R}}\phi^{*}c_{1}(L,\left\|\cdot\right\|).
$$
If moreover we have a closed subscheme $Y$ of $X$ such that for all $x\in Z$, the intersection $Y\cap \ol{\CF(x,v)}^{\rm zar}$ has codimension at least $2$ in $\ol{\CF(x,v)}^{\rm zar}$, then for any $\varepsilon'>0$ we can further ensure the inequality
$$
\sum\limits_{z\in\BD_{R}}{\rm ord}_{z}\phi^{*}Y\leq\varepsilon'\int_{\BD_{R}}\phi^{*}c_{1}(L,\left\|\cdot\right\|).
$$
Finally, if $Y$ is contained in ${\rm Supp}\,D$, then we can further choose a Euclidean closed subset $W\subseteq {\rm Supp}\,D$ which is disjoint from $Y$ such that
$$
\sum\limits_{z\in\BD_{R}\cap\phi^{-1}(W)}{\rm ord}_{z}\phi^{*}D\geq(1-\varepsilon)\int_{\BD_{R}}\phi^{*}c_{1}(L,\left\|\cdot\right\|).
$$
for all $\phi\colon\BD_{R+1}\to X$ above.
\end{proposition}
\begin{proof}
For simplicity, we only prove the last statement. The proof of the first two statements is similar and even easier. Now fix a closed subscheme $Y\subseteq {\rm Supp}\,D$ such that for all $x\in Z$, the intersection $Y\cap \ol{\CF(x,v)}^{\rm zar}$ has codimension at least $2$ in $\ol{\CF(x,v)}^{\rm zar}$. By Corollary \ref{rm-} and Proposition \ref{Y}, there is a constant $R>0$ such that for each point $x\in Z$, we have
\begin{align*}
\sum\limits_{z\in\BD_{R}}{\rm ord}_{z}\phi_{(x,v)}^{*}D\geq (1-\frac{\varepsilon}{4})\int_{\BD_{R}}\phi_{(x,v)}^{*}c_{1}(L,\left\|\cdot\right\|).
\end{align*}
and
\begin{align*}
\sum\limits_{z\in\BD_{R}\cap\phi_{(x,v)}^{-1}(Y)}{\rm ord}_{z}\phi_{(x,v)}^{*}D\leq \frac{\varepsilon}{4}\int_{\BD_{R}}\phi_{(x,v)}^{*}c_{1}(L,\left\|\cdot\right\|).
\end{align*}
Taking the difference of the two inequalities, we get
\begin{align}\label{conclusion-1.5}
\sum\limits_{z\in\BD_{R}\backslash\phi_{(x,v)}^{-1}(Y)}{\rm ord}_{z}\phi_{(x,v)}^{*}D\geq (1-\frac{\varepsilon}{2})\int_{\BD_{R}}\phi_{(x,v)}^{*}c_{1}(L,\left\|\cdot\right\|).
\end{align}
for all $x\in Z$.

Fix this $R$. Now we prove the existence of $\delta>0$ in the first statement by contradiction. Take an increasing sequence $W_{1}\subseteq W_{2}\subseteq\dots\subseteq W_{n}\subseteq\dots$ of Euclidean closed subset of $D$ such that each $W_{n}$ is contained in the interior of $W_{n+1}$ in $(D\backslash Y)(\CC)$, and the union $\cup_{n\geq1}W_{n}$ equals $(D\backslash Y)(\CC)$. Here we identify $D$ with ${\rm Supp}\,D$. If the conclusion does not hold, there is a point $x_{n}\in Z$ and a holomorphic map $\phi_{n}\colon\BD_{R+1}\to X$ for each $n\geq1$ satisfying $d(\phi_{n}(z),\phi_{(x_{n},v)}(z))<\frac{1}{n}\,(\forall z\in\BD_{R+1})$ such that
\begin{align}\label{conclusion-2}
\sum\limits_{z\in\BD_{R}\cap\phi_{n}^{-1}(W_{n})}{\rm ord}_{z}\phi_{n}^{*}D<(1-\varepsilon)\int_{\BD_{R}}\phi_{n}^{*}c_{1}(L,\left\|\cdot\right\|).
\end{align}
After taking a subsequence, we may assume that $x_{n}\to x_{\infty}\in Z$ in the Euclidean topology. Then $\phi_{n}$ converges to the linear entire curve $\phi_{(x_{\infty},v)}$ uniformly on $\BD_{R+1}$. By Remark \ref{blx}, there is a Euclidean closed subset $W$ of $D$ which is disjoint from $Y$ such that
\begin{align*}
\sum\limits_{z\in\BD_{R}\cap\phi_{n}^{-1}(W)}{\rm ord}_{z}\phi_{n}^{*}D = \sum\limits_{z\in\BD_{R}\backslash\phi_{(x_{\infty},v)}^{-1}(Y)}{\rm ord}_{z}\phi_{(x_{\infty},v)}^{*}D\geq (1-\frac{\varepsilon}{2})\int_{\BD_{R}}\phi_{(x_{\infty},v)}^{*}c_{1}(L,\left\|\cdot\right\|).
\end{align*}
for $n$ large enough. Here the right-hand inequality follows from (\ref{conclusion-1.5}). By continuity, we have
$$
\int_{\BD_{R}}\phi_{n}^{*}c_{1}(L,\left\|\cdot\right\|)\to\int_{\BD_{R}}\phi_{(x_{\infty},v)}^{*}c_{1}(L,\left\|\cdot\right\|)
$$
as $n\to\infty$. Combining with the inequality above, we have
\begin{align*}
\sum\limits_{z\in\BD_{R}\cap\phi_{n}^{-1}(W)}{\rm ord}_{z}\phi_{n}^{*}D \geq (1-\varepsilon)\int_{\BD_{R}}\phi_{n}^{*}c_{1}(L,\left\|\cdot\right\|).
\end{align*}
This contradicts (\ref{conclusion-2}) since $W\subseteq W_{n}$ for $n$ large enough.
\end{proof}
\section{Height inequalities on abelian varieties}\label{4}
In this section, we prove our main theorems. The core is to prove Proposition \ref{V}. We will do some basic reductions and prove a key lemma in \S\ref{4.1}, and then prove (1), (3), (2) of Proposition \ref{V} in \S\ref{fenbu}, \S\ref{4.3} and \S\ref{4.4} respectively. These correspond to the asymptotic distribution of intersections of sections with a divisor, the GCD and abc inequalities. The proofs of these inequalities follow the same pattern, but the technical difficulties increase gradually. Then in \S\ref{4.5}, we deduce the main theorems from Proposition \ref{V}. Finally in \S\ref{4.6}, we further show the $K$-independence of the exceptional set in the main theorems.

Throughout this section, we fix a smooth projective curve $B$ over $\CC$ as well as a K\"{a}hler form $\omega_{B}$ on $B(\CC)$. Then $\omega_{B}$ induces a measure $\mu_{B}$ and a distance function $d_{B}$ on $B(\CC)$.

Before proving Proposition \ref{V}, we make the following observation: Let $P$ be a property for some points in a variety $X$. If for any generic sequence $\{x_{n}\}_{n\geq1}$ in $X$, there are infinitely many $n$ such that $P$ holds for $x_{n}$. Then for all such sequences, the property $P$ holds for $x_{n}$ when $n$ large enough. In fact, if not, we can take an infinite subsequence $\{y_{n}\}_{n\geq1}$ of some generic sequence $\{x_{n}\}_{n\geq1}$ in $X$ such that the property $P$ does not hold for all $y_{n}$, this contradicts our assumption since $\{y_{n}\}_{n\geq1}$ is also generic in $X$. Therefore we only need to prove that all inequalities in Proposition \ref{V} hold for infinitely many $n$. In particular, we can freely take infinite subsequences as long as we need.
\subsection{Basic reductions and a key lemma}\label{4.1}
We begin to prove Proposition \ref{V}. First we can assume that the height of the image of $\{s_{n}\}_{n\geq1}$ in any nontrivial quotient abelian variety tends to infinity as $n\to\infty$. In fact, assume that there is a nontrivial quotient abelian variety $A_{0}=A/A_{1}$ of $A$ such that the height of the image of $\{s_{n}\}_{n\geq1}$ in $A_{0}(K)$ does not tend to infinity. After taking a subsequence (which does not change the genericity), we may assume that the heights of these images are bounded. By Lang--N\'{e}ron theorem, $V(A_{0},K)=(A_{0}(K)/A_{0}^{(K/\CC)}(\CC))_{\RR}$ is a finite-dimensional $\RR$-vector space, and the canonical height on $A_{0}(K)$ (with respect to an ample symmetric line bundle on $A_{0}$) induces a positive definite quadratic form on $V(A_{0},K)$. As a consequence, there are only finitely many elements in $V(A_{0},K)$ with a fixed bounded height. Hence the image of $\{s_{n}\}_{n\geq1}$ in $A_{0}(K)$ is contained in a finite union of translations of $A_{0}^{(K/\CC)}(\CC)$. As $\{s_{n}\}_{n\geq1}$ is generic in $A$, we conclude that $A_{0}$ is isomorphic to $(A_{0}^{(K/\CC)})_{K}$. After taking a subsequence, we assume that the image of $s_{n}$ in $A_{0}$ lies in a single translation $a+A_{0}^{(K/\CC)}(\CC)$.

Now take an isogeny $i\colon A_{1}\times A_{0}\to A$ such that its composition with the quotient map $A\to A_{0}$ equals $(a_{1},a_{0})\mapsto ma_{0}$ for some integer $m$. Consider the composition $t_{a}\circ i\colon A_{1}\times A_{0}\overset{i}\to A\overset{+a}\to A$. By (geometric) Chevalley--Weil theorem (see \cite[Lemma 2.1]{arXiv:2507.21468}), there is a finite field extension $K'/K$ such that $A(K)\subseteq (t_{a}\circ i)(A_{1}(K')\times A_{0}(K'))$. Choose $(s_{n,1},s_{n,0})\in (A_{1}(K')\times A_{0}(K'))\cap (t_{a}\circ i)^{-1}(s_{n})=i^{-1}(s_{n}-a)$ for each $n\geq1$. Then $s_{n,0}\in A_{0}^{(K/\CC)}(\CC)$ for each $n\geq1$. After base change everything from $K$ to $K'$ (which does not affect all height functions up to $O(1)$), we may assume that $s_{n,1}\in A_{1}(K)$ for all $n\geq1$.

Consider the finite \'{e}tale morphism $(t_{a}\circ i)\times{\rm id}_{T_{K}}\colon A_{1}\times A_{0}\times T_{K}\to A\times T_{K}$. Let $D'=((t_{a}\circ i)\times{\rm id}_{T_{K}})^{*}D$ and $Y'=((t_{a}\circ i)\times{\rm id}_{T_{K}})^{*}Y$. Then $D'$ is also relatively ample over $T_{K}$. Moreover, if $D$ is reduced, then $D'$ is also reduced. Let $B'\to B$ be the finite surjective map corresponding to the extension $K'/K$ and let $U'$ be the inverse image of $U\subseteq B(\CC)$ in $B'(\CC)$. By projection formula, we have
$$
h_{D'}((s_{n,1},s_{n,0},t_{n}))=h_{D}((s_{n},t_{n}))+O(1)
$$
for all $n\geq1$, and similarly for $h_{Y'},m_{D',U'},N^{(1)}_{D',B(\CC)\backslash U'}$.

Now consider the abelian scheme $A_{1}\times A_{0}\times T_{K}\to A_{0}\times T_{K}$. Let $(D')^{\rm sym}$ be the symmetrization of $D'$ over $A_{0}\times T_{K}$. By Lemma \ref{bijiao}, we have
$$
\hat{h}_{(D')^{\rm sym},U'}((s_{n,1},s_{n,0},t_{n}))=\hat{h}_{D^{\rm sym},U}((s_{n},t_{n}))+O\left(1+\sqrt{\hat{h}_{D^{\rm sym},U}((s_{n},t_{n}))}\right)
$$
for all $n\geq1$.

The equalities above show that we only need to prove Proposition \ref{V} for $A_{1}\times (A_{0}^{(K/\CC)}\times_{\CC} T)_{K},D',Y',U',\{(s_{n,1},(s_{n,0},t_{n}))\}_{n\geq1}$. Replacing $A$ by $(A_{1})_{K'}$ and $T$ by $A_{0}^{(K/\CC)}\times_{\CC}T$ (and thus $s_{n}$ changes to $s_{n,1}$ and $t_{n}$ changes to $(s_{n,0},t_{n})$), then we have $\dim A_{1}<\dim A$. Repeating this process finitely many times, we can eventually assume that the height of the image of $\{s_{n}\}_{n\geq1}$ in any nontrivial quotient abelian variety tends to infinity as $n\to\infty$.

Take an integral model $\CA$ of $A$ over $B$. Then $\CA\times_{\CC}T$ is an integral model of $A\times T_{K}$ over $B$. View $(s_{n},t_{n})$ as a section $B\to\CA\times_{\CC}T$ of this integral model. Take $\CD$ to be an effective Cartier divisor on $\CA\times_{\CC}T$ extending $D$, and take $\CY$ to be a closed subscheme of $\CA\times_{\CC}T$ extending $Y$. Identifying $\CD$ with its support, we can also view it as a closed subscheme of $\CA\times_{\CC}T$. Consider the compositions $\CD\hookrightarrow \CA\times_{\CC}T\to T$ and $\CY\hookrightarrow \CA\times_{\CC}T\to T$. We first reduce to the case when they are both flat.

In fact, by Raynaud--Gruson's flattening theorem (cf. \cite[Th\'{e}or\`{e}me 5.2.2]{zbMATH03360302}), there is a proper birational map $g\colon T'\to T$ such that the strict transforms $\CD',\CY'$ of $\CD,\CY$ under ${\rm id}_{\CA}\times g\colon\CA\times_{\CC}T'\to\CA\times_{\CC}T$ are both flat over $T'$. Let $t'_{n}\in T'(\CC)$ be a lift of $t_{n}$ under $g$. As $({\rm id}_{\CA}\times g)^{*}\CD-\CD'$ is vertical over $T'$ and $\{t_{n}\}_{n\geq1}$ is generic in $T$, we see that $\{t'_{n}\}_{n\geq1}$ is generic in $T'$ and
$$
{\rm ord}_{b}(s_{n},t_{n})^{*}\CD={\rm ord}_{b}(s_{n},t'_{n})^{*}\CD'
$$
for $n$ large enough and for all $b\in B(\CC)$, and similarly for $\CY$. This means that all the counting functions in Proposition \ref{V} don't change in this process for $n$ large enough. Moreover, let $D'$ be the generic fibre of $\CD'\to B$, which is an effective Cartier divisor on $A\times T'_{K}$. As $g$ is birational, the restriction of $D'$ to general fibre $A\times\{t'\}\,(t'\in T'_{K})$ is ample. Since the restriction of $D'$ to all fibres $A\times\{t'\}\,(t'\in T'_{K})$ are numerically equivalent, we see that $D'$ is still relatively ample over $T'_{K}$. Therefore after replacing $T,\{t_{n}\}_{n\geq1},\CD,\CY$ by $T',\{t'_{n}\}_{n\geq1},\CD',\CY'$, we can assume that both $\CD$ and $\CY$ are flat over $T$. In particular, the union $\CY\cup\CD$ does not contain any fibre $\CA\times\{t\}\,(t\in T)$. We always assume this in the following.

Now we prove the following key lemma in this subsection. The letters in the lemma are independent of those in Proposition \ref{V} (although the settings are similar).

\begin{lemma}\label{juli}
Let $A$ be an abelian variety over $K$, and let $T$ be a projective variety over $\CC$. Let $\{(s_{n},t_{n})\}_{n\geq1}\subseteq A(K)\times T_{K}(K)$ be a generic sequence of rational points on $A\times T_{K}$ such that $t_{n}\in{\rm Im}(T(\CC)\to T_{K}(K))$ for all $n$. Moreover, for any nontrivial quotient abelian variety $A'$ of $A$, the height of the image of $s_{n}$ in $A'$ (with respect to an ample line bundle on $A'$) tends to infinity as $n\to\infty$.

Let $\CA$ be an integral model of $A$ over $B$. Let $\CZ$ be a closed subscheme of $\CA\times_{\CC}T$ which does not contain any fibre $\CA\times\{t\}\,(t\in T)$. View $(s_{n},t_{n})$ as a section $B\to\CA\times_{\CC}T$ for each $n\geq1$. Then for any $\delta>0$, there is an (Euclidean) open neighbourhood $\CU$ of $\CZ$ such that for $n$ large enough, the subset $(s_{n},t_{n})^{-1}(\CU)$ of $B(\CC)$ has $\mu_{B}$-measure less than $\delta$.
\end{lemma}
As $\{(s_{n},t_{n})\}_{n\geq1}$ is generic in $A\times T_{K}$, the subset $(s_{n},t_{n})^{-1}(\CZ)$ is finite for $n$ large enough, and thus has measure 0. The key of the lemma is that we can replace $\CZ$ by a small (Euclidean) neighbourhood $\CU$ of it, and then bound the measure of $(s_{n},t_{n})^{-1}(\CU)$. Note that the condition ``$\CZ$ does not contain any fibre $\CA\times\{t\}\,(t\in T)$'' is necessary: If $\CZ$ contains some fibre $\CA\times\{t\}$ and $t_{n}\to t$ under Euclidean topology of $T(\CC)$, then any open neighbourhood $\CU$ of $\CZ$ contains $\CA\times\{t_{n}\}$ for $n$ large enough, which implies that $(s_{n},t_{n})^{-1}(\CU)=B(\CC)$ for $n$ large enough.
\begin{proof}
First note that we only need to show that for any $\delta>0$, there is a Euclidean open neighbourhood $\CU$ of $\CZ$ such that $\mu_{B}((s_{n},t_{n})^{-1}(\CU))<\delta$ for infinitely many $n\geq1$. We call this the `weak statement'. In fact, choose a fundamental system of Euclidean open neighbourhoods $\{\CU_{k}\}_{k\geq1}$ of $\CZ$. If the lemma is not true, there is a sequence $\{n_{k}\}_{k\geq1}$ with $n_{k}\to\infty$ such that $\mu_{B}((s_{n_{k}},t_{n_{k}})^{-1}(\CU_{k}))\geq\delta$ for each $k$. Using the weak statement above for the sequence $\{(s_{n_{k}},t_{n_{k}})\}_{k\geq1}$, we then get a contradiction.

We prove the weak statement by induction on the number of simple isogeny factors of $A$. When $A$ is trivial, then $\CA=B$ and $\CA\times_{\CC}T=B\times_{\CC}T$. By our assumption, $\CZ$ does not contain any fibre $B\times\{t\}$, and thus the intersection $(B\times\{t_{\infty}\})\cap\CZ$ is finite. After taking a subsequence, we may assume $t_{n}\to t_{\infty}$ in $T(\CC)$. Take $\CU_{0}$ to be a Euclidean open neighbourhood of $\CZ$ such that $\{b\in B\big| (b,t_{\infty})\in \CU_{0}\}$ has $\mu_{B}$-measure less than $\delta$. Let $\CU$ be a Euclidean open neighbourhood of $\CZ$ such that the Euclidean closure $\ol{\CU}$ is contained in $\CU_{0}$. Then for $n$ large enough, if $(b,t_{n})\in \CU$ for some $b\in B$, then $(b,t_{\infty})\in \CU_{0}$. This means that the set $(s_{n},t_{n})^{-1}(\CU)=\{b\in B\big|(b,t_{n})\in \CU\}$ is contained in $\{b\in B\big|(b,t_{\infty})\in \CU_{0}\}$, and thus has $\mu_{B}$-measure less than $\delta$.

Now we consider general $A$. After taking a subsequence, we may assume $t_{n}\to t_{\infty}$ in $T(\CC)$, and there is a sequence $\{\ell_{n}\}_{n\geq1}$ of positive real numbers such that $\ell_{n}^{-1}s_{n}\to s_{\infty}\neq0$ in $V(A,K)$. Let $G=G(s_{\infty})$ be the unique minimal abelian subvariety of $A$ such that $s_{\infty}\in V(G,K)$. Let $\CG$ be the Zariski closure of $G$ in $\CA$ and let $\CA/\CG$ be an integral model of $A/G$ over $B$. Then the quotient map $A\to A/G$ extends to a map $q\colon\CA_{B_{0}}\to(\CA/\CG)_{B_{0}}$ over some Zariski dense open subset $B_{0}$ of $B$. Denote by $p$ the map $q\times{\rm id}_{T}\colon \CA_{B_{0}}\times_{\CC}T\to(\CA/\CG)_{B_{0}}\times_{\CC}T$.

Choose an ample effective Cartier divisor $\CE$ of $\CA\times_{\CC}T$ (which can also be viewed as a closed subscheme of $\CA\times_{\CC}T$) containing $\CZ$ which does not contain any fibre $\CA\times_{\CC}\{t\}\,(t\in T(\CC))$. Then we have $\ell_{n}\sim h(s_{n})^{1/2}\sim (\deg (s_{n},t_{n})^{*}\CE)^{1/2}$. Choose a constant $C_{1}>0$ such that $\ell_{n}^{-2}\cdot\deg (s_{n},t_{n})^{*}\CE<C_{1}$ for all $n\geq1$. We only need to find a Euclidean neighbourhood $\CU$ of $\CE$ with $\mu_{B}((s_{n},t_{n})^{-1}(\CU))<\delta$ for $n$ large enough.

Let $\CZ'_{0}$ be the set of all $(x,t)\in(\CA/\CG)_{B_{0}}\times_{\CC}T$ such that $\CE$ contains the whole fibre $p^{-1}((x,t))$. Let $\CZ'$ be the Zariski closure of $\CZ'_{0}$ in $(\CA/\CG)\times_{\CC}T$. As $\CE$ does not contain any fibre $\CA\times\{t\}$, we deduce that $\CZ'_{0}$ does not contain any fibre $(\CA/\CG)_{B_{0}}\times\{t\}$, hence $\CZ'$ does not contain any fibre $(\CA/\CG)\times\{t\}$.

Let $\ol{s}_{n}$ be the image of $s_{n}$ in $A/G$. Applying the induction hypothesis to the sequence\\
$\{(\ol{s}_{n},t_{n})\}_{n\geq1}\subseteq (A/G)(K)\times T(\CC)$ and the closed subscheme $\CZ'$ of $(\CA/\CG)\times_{\CC}T$, we see that for any $\delta>0$, there is a Euclidean open neighbourhood $\CU'$ of $\CZ'$ such that for $n$ large enough, the subset $(\ol{s}_{n},t_{n})^{-1}(\CU')$ of $B(\CC)$ has $\mu_{B}$-measure less than $\frac{\delta}{2}$.

Now let $\CW=(\CA\times_{\CC}T)\backslash p^{-1}(\CU')$. Note that $(s_{n},t_{n})^{-1}(p^{-1}(\CU'))=(\ol{s}_{n},t_{n})^{-1}(\CU')$, which has measure less than $\frac{\delta}{2}$ for $n$ large enough. Therefore we only need to show that there is a Euclidean open neighbourhood $\CU$ of $\CE$ such that $(s_{n},t_{n})^{-1}(\CW\cap \CU)$ has $\mu_{B}$-measure less than $\frac{\delta}{2}$ for infinitely many $n$.

Fix distance functions $d_{\CA}$ on $\CA$ and $d_{T}$ on $T$. Then $d=d_{\CA}\times d_{T}$ is a distance function on $\CA\times_{\CC}T$. For a Euclidean closed subset $E$ of $\CA\times_{\CC}T$ and a real number $c>0$, denote by $\BB(E,c)\colonequals\{(x,t)\in\CA\times_{\CC}T\big|\,d((x,t),E)<c\}$. Recall (at the beginning of this section) that both the distance function $d_{B}$ and the measure $\mu_{B}$ on $B(\CC)$ are induced by a fixed K\"{a}hler form $\omega_{B}$. Hence by enlarging the constant $C_{1}>0$ above, we can ensure that $\mu_{B}(\BB(b,r))\leq C_{1}r^{2}$ for any $r>0$ and any $b\in B(\CC)$.

By Proposition \ref{closure}, there is a subset $S\subseteq B(\CC)$ of measure 0 such that for all $b\in B(\CC)\backslash S$, the Zariski closure of the linear leaf $\CF(0,\delta(s_{\infty},v_{b}))$ in $\CA_{b}$ equals $\CG_{b}$. Here $v_{b}\in T_{b}B$ is a nonzero vector. Enlarging $S$ if necessary, we can assume $(B\backslash B_{0})\subseteq S$.

We claim that for each $b\in B(\CC)\backslash S$, there is an open disc $\BD_{b}\subseteq B(\CC)$ containing $b$, a constant $\varepsilon_{b}>0$ and an integer $n_{b}\geq1$ such that for all $n>n_{b}$ and all $b'\in \BD_{b}$, if $(s_{n}(b'),t_{n})\in \CW$ and $d((s_{n}(b'),t_{n}),\CE)<\varepsilon_{b}$, then there is a point $b''\in(s_{n},t_{n})^{-1}(\CE)$ with $d_{B}(b',b'')<(2C_{1}\ell_{n})^{-1}\delta^{1/2}$.

If the claim is true, we can choose finitely many triples $(\BD_{b_{i}},\varepsilon_{b_{i}},n_{b_{i}})\,(1\leq i\leq m)$ as above such that the union of all $\BD_{b_{i}}$ in $B(\CC)$ has measure greater than $\mu_{B}(B)-\frac{\delta}{4}$. Let $\varepsilon=\min_{1\leq i\leq m}\varepsilon_{b_{i}}$ and $n_{0}=\max_{1\leq i\leq m}n_{b_{i}}$. For each point $b'$ in this union, there exists some $\BD_{b_{i}}$ containing $b'$. By our claim, if $n>n_{0}, (s_{n}(b'),t_{n})\in\CW$ and $d((s_{n}(b'),t_{n}),\CE)<\varepsilon$, then there is a point $b''\in(s_{n},t_{n})^{-1}(\CE)$ with $d_{B}(b',b'')<(2C_{1}\ell_{n})^{-1}\delta^{1/2}$. Therefore the set
$$
(s_{n},t_{n})^{-1}(\CW\cap \BB(\CE,\varepsilon))\subseteq (s_{n},t_{n})^{-1}(\CW\cap \BB(\CE,\varepsilon))
$$
is contained in the union of the complement $B(\CC)\backslash(\cup_{1\leq i\leq m}\BD_{b_{i}})$ and the open balls
$$
\BB(b,(2C_{1}\ell_{n})^{-1}\delta^{1/2}),\,b\in (s_{n},t_{n})^{-1}(\CE).
$$
By our choice of $C_{1}$, each ball has $\mu_{B}$-measure at most  $C_{1}((2C_{1}\ell_{n})^{-1}\delta^{1/2})^{2}=\frac{\ell_{n}^{-2}\delta}{4C_{1}}$. Hence the union has $\mu_{B}$-measure at most $\frac{\delta}{4}+\frac{\ell_{n}^{-2}\delta}{4C_{1}}\deg(s_{n},t_{n})^{*}\CE<\frac{\delta}{2}$.

Now we prove the claim. Fix a point $b\in B(\CC)\backslash S$. Choose any open unit disc $\BD$ centred at $b$. Then there is a constant $C_{2}>0$ such that for any $b',b''\in\BD_{0.9}$, we have $d_{B}(b',b'')\leq C_{2}|b'-b''|$. Denote $v=\delta(s_{\infty},v_{b})$. Choose a constant $\varepsilon_{1}>0$ such that the Euclidean closure $\ol{\CV}$ of $\CV\colonequals \BB(\CW,\varepsilon_{1})$ is disjoint from $p^{-1}(\CZ')$. Then if $x\in\CA_{b}$ such that $(x,t_{\infty})\in\CV$, the subset $(x+\CG_{b})\times\{t_{\infty}\}$, which is the Zariski closure of the linear leaf $\CF(x,v)\times\{t_{\infty}\}$, is not contained in $\CE$.

Applying Proposition \ref{rongchi} to the abelian subvariety $\CA_{b}\times\{t_{\infty}\}$ of $\CA\times_{\CC}T$, the divisor $\CE$, the vector $v=\delta(s_{\infty},v_{b})$, the Euclidean closed subset $\ol{\CV}_{b,t_{\infty}}$ of $\CA_{b}\times\{t_{\infty}\}$ and the constant $c=\frac{\delta^{1/2}}{4C_{1}C_{2}}$, we see that there is a constant $\delta_{c}>0$ such that for any $r>0$, any $(x,t_{\infty})\in\CV_{b,t_{\infty}}$ and any holomorphic map $\phi\colon\BD_{r+c}\to\CA\times_{\CC}T$, if $d(\phi(z),(\phi_{(x,v)},t_{\infty}))<\delta_{c}$ for all $z\in\BD_{r+c}$, then
\begin{align}\label{juli-1}
\sum\limits_{z\in\BD_{r+c}}{\rm ord}_{z}\phi^{*}\CE\geq \sum\limits_{z\in\BD_{r}}{\rm ord}_{z}(\phi_{(x,v)}\times\{t_{\infty}\})^{*}\CE.
\end{align}
By Lemma \ref{linear}, there are constants $0<\varepsilon'<\min\{1,\varepsilon_{1}\}$ and $n_{0}>0$ such that for each $b'\in\BD_{\varepsilon'}$, each $z\in\BD_{2c}$ and each $x\in\CA_{b}=\CA_{0}$, if $n>n_{0}$ and $d_{\CA}(s_{n}(b'),x)\leq\varepsilon'$, then $b'+\ell_{n}^{-1}z\in\BD$ and $d_{\CA}(s_{n}(b'+\ell_{n}^{-1}z),\phi_{(x,v)}(z))<\frac{\delta_{c}}{2}$.

Since $\CE$ is closed in the compact space $\CA$, there is a constant $0<\varepsilon''<\varepsilon'$ such that for any $b'\in\BD_{\varepsilon''}$ and any $x'\in\CA_{b'}$, if $d((x',t_{\infty}),\CE)<\varepsilon''$, then there is a point $(x,t_{\infty})\in\CE_{b,t_{\infty}}$ with $d_{\CA}(x,x')<\frac{\varepsilon'}{2}$. Finally take $n_{b}>n_{0}$ such that $d_{T}(t_{n},t_{\infty})<\min\{\frac{\varepsilon''}{2},\frac{\delta_{c}}{2}\}$ for any $n>n_{b}$, take $\varepsilon_{b}=\frac{\varepsilon''}{2}$ and take $\BD_{b}=\BD_{\varepsilon''}$.

Then for any $b'\in\BD_{b}$ and any $n>n_{b}$, if $(s_{n}(b'),t_{n})\in\CW$ and $d((s_{n}(b'),t_{n}),\CE)<\varepsilon_{b}$, we have $d((s_{n}(b'),t_{\infty}),\CE)<\frac{\varepsilon''}{2}+\frac{\varepsilon''}{2}=\varepsilon''$. Thus there is a point $(x,t_{\infty})\in\CE_{b,t_{\infty}}$ with $d_{\CA}(x,s_{n}(b'))<\frac{\varepsilon'}{2}$. This implies $d((x,t_{\infty}),(s_{n}(b'),t_{n}))<\frac{\varepsilon'}{2}+\frac{\varepsilon''}{2}<\varepsilon'<\varepsilon_{1}$. In particular, $(x,t_{\infty})\in\CV$ and we can apply (\ref{juli-1}) to $(x,t_{\infty})$. More precisely, consider $\phi\colon\BD_{2c}\to\CA\times_{\CC}T$ sending $z\in\BD_{2c}$ to $(s_{n}(b'+\ell_{n}^{-1}z),t_{n})$. For any $z\in\BD_{2c}$ and $n$ large enough, we have
\begin{align*}
d(\phi(z),(\phi_{(x,v)}(z),t_{\infty}))=d_{\CA}(s_{n}(b'+\ell_{n}^{-1}z),\phi_{(x,v)}(z))+d_{T}(t_{n},t_{\infty})< \frac{\delta_{c}}{2}+\frac{\delta_{c}}{2}=\delta_{c}.
\end{align*}
As $(x,t_{\infty})\in\CE$, the inequality (\ref{juli-1}) with $r=c$ implies that $\phi(\BD_{2c})\cap\CE\neq\emptyset$. This means that there exists $z\in\BD_{2c}$ with $(s_{n}(b'+\ell_{n}^{-1}(z),t_{n}))\in\CE$. Take $b''=b'+\ell_{n}^{-1}z$, then $d_{B}(b',b'')\leq 2C_{2}\ell_{n}^{-1}c=(2C_{1}\ell_{n})^{-1}\delta^{1/2}$. This proves the claim.
\end{proof}

We will use Lemma \ref{juli} for $\CD,\CY$ and some other related subschemes in the subsequent subsections. In the end of this subsection, we fix some notations which will be used in the remainder of this section.

Fix a Zariski dense open subset $B_{0}$ of $B$ such that $\CA_{B_{0}}\to B_{0}$ is an abelian scheme. Consider the abelian scheme $\CA_{B_{0}}\times_{\CC}T\to B_{0}\times_{\CC}T$. Let $\omega$ be the Betti form on $\CA_{B_{0}}\times_{\CC}T$ associated to the divisor $D$. Then it is also the Betti form associated to $D^{\rm sym}$.

Take $\ell_{n}\to\infty$ such that $\{\ell_{n}^{-1}s_{n}\}_{n\geq1}$ has a subsequence tending to a nonzero element $s_{\infty}$ in $V(A,K)$. After passing to the subsequence, we may assume that $\ell_{n}^{-1}s_{n}\to s_{\infty}$ as $n\to\infty$. Let $G=G(s_{\infty})$ be the unique minimal abelian subvariety of $A$ such that $s_{\infty}\in V(G,K)$. Let $\CG$ be the Zariski closure of $G$ in $\CA$. By Proposition \ref{closure}, there is a subset $S(A)$ of $B(\CC)$ of measure 0 containing $B\backslash B_{0}$ such that for any $b\in B(\CC)\backslash S(A)$, the Zariski closure of the linear leaf $\CF(0,\delta(s_{\infty},v_{b}))$ in $\CA_{b}$ equals $\CG_{b}$. In particular, we have $\delta(s_{\infty},v_{b})\neq 0$ for all $b\in B(\CC)\backslash S(A)$.

Similarly as in the proof of Lemma \ref{juli}, in the following, we fix distance functions $d_{\CA}$ on $\CA$ and $d_{T}$ on $T$. Then $d=d_{\CA}\times d_{T}$ is a distance function on $\CA\times_{\CC}T$.
\subsection{Asymptotic distributions of intersections of sections with a divisor}\label{fenbu}
In this subsection we prove the first part of Proposition \ref{V}. We first establish a local estimate in the following Proposition \ref{key-}, and then combine it with Proposition \ref{jiben-} and Lemma \ref{juli} to get the global estimate.

The proof of the local estimate also contains two parts: We first give a ``microscopic local'' estimate (see (\ref{key--5})), this is mainly based on \S\ref{3}. And then we prove a lemma (see Lemma \ref{refine--}) to glue these microscopic local inequalities to a local inequality.

For $z\in\CC$ and $r>0$, denote by $\BD(z,r)$ the open disc of radius $r$ with centre $z$.
\begin{proposition}\label{key-}
There is a constant $c>0$ depending only on $\mu_{B}$ satisfying the following property: For a Euclidean open neighbourhood $\CU$ of $\CD$, denote $B_{n}\colonequals (s_{n},t_{n})^{-1}(\CU)$. Then for any point $b\in B(\CC)\backslash S(A)$ and any $\varepsilon>0$, there is an open disc $\BD_{b}\subseteq B(\CC)$ (depending on $U$) containing $b$ such that for any open disc $\BD'_{b}$ contained in $\BD_{b}$, any $\varepsilon'>0$ and any $n$ large enough (depending on $\BD'_{b}$ and $\varepsilon'$), there exists a subset $\Sigma_{n,b}\subseteq \BD'_{b}$ of measure less than $\frac{c}{\varepsilon}\mu_{B}(B_{n}\cap\BD'_{b})+\varepsilon'$ with
\begin{align*}
\sum\limits_{t\in\BD'_{b}}{\rm ord}_{t}(s_{n},t_{n})^{*}\CD\geq (1-\varepsilon)\int_{\BD'_{b}\backslash\Sigma_{n,b}}(s_{n},t_{n})^{*}\omega.
\end{align*}
\end{proposition}

\begin{proof}
Take a finite cover of $B(\CC)$ by open discs. For any disc $\BD$ in this cover, fix a standard Lebesgue measure $\mu_{0,\BD}$ on $\BD$. Then there is a constant $c_{\BD}>1$ such that $c_{\BD}^{-1}\mu_{0,\BD}(\Omega)\leq\mu_{B}(\Omega)\leq c_{\BD}\mu_{0,\BD}(\Omega)$ for any Lebesgue measurable subset $\Omega\subseteq \BD$. Take $c=2\max_{\BD}\{c_{\BD}^{2}\}>0$, where $\BD$ is taken over all open discs in this finite cover. We show that $c$ satisfies the conditions in the proposition.

Fix a point $b\in B(\CC)\backslash S(A)$, a Euclidean open neighbourhood $\CU$ of $\CD$ and a constant $\varepsilon>0$. Then $\delta(s_{\infty},v_{b})\neq0$. Applying Proposition \ref{conclusion} to the abelian subvariety $\CA_{b}\times\{t_{\infty}\}\subseteq\CA\times_{\CC}T$, the nonzero vector $v=\delta(s_{\infty},v_{b})$ and the Euclidean closed subset $Z=\CA_{b}\times\{t_{\infty}\}\backslash \CU_{b,t_{\infty}}$, we get constants $R>0, \delta>0$ satisfying the following condition: For any point $(x,t_{\infty})\in Z$ and any holomorphic map $\phi\colon \BD_{R+1}\to \CA\times_{\CC}T$, if $d(\phi(z), (\phi_{(x,v)}(z),t_{\infty}))<\delta$ for all $z\in\BD_{R+1}$, then
\begin{align}\label{key--2}
\sum\limits_{z\in\BD_{R}}{\rm ord}_{z}\phi^{*}\CD\geq(1-\frac{\varepsilon}{2})\int_{\BD_{R}}\phi^{*}\omega.
\end{align}

Take any open unit disc $\BD$ with centre $b$ which is contained in some open disc of the finite cover chosen at the beginning of the proof. Let $z$ be the standard coordinate of $\BD$. By Lemma \ref{linear}, there are constants $0<\delta'<1$ and $n_{0}>0$ such that for each $b'\in\BD_{\delta'}$, each $z\in\BD_{R+1}$ and each $x\in\CA_{b}=\CA_{0}$, if $n>n_{0}$ and $d(s_{n}(b'),x)<\delta'$, then $b'+\ell_{n}^{-1}z\in\BD$ and
\begin{align}\label{key--3}
d_{\CA}(s_{n}(b'+\ell_{n}^{-1}z),\phi_{(x,v)}(z))<\delta/2
\end{align}
Enlarging $n_{0}$ and replacing $\delta'$ by a smaller positive number if necessary, we can assume that for any $n>n_{0}$ and any point $(x',t_{n})\in(\CA_{\BD_{\delta'}}\times\{t_{n}\})\backslash \CU$, there is a point $(x,t_{\infty})\in Z$ with $d((x',t_{n}),(x,t_{\infty}))<\delta'/2$. In particular, we have $d_{T}(t_{n},t_{\infty})<\delta'/2<\delta/2$ for all $n>n_{0}$.

Now take $\BD_{b}=\BD_{\delta'}$. We prove that it satisfies the conditions in the Proposition.

In fact, for any open disc $\BD'_{b}\subseteq \BD_{b}$, assume $\BD'_{b}=\BD(b',r')$. For $b''\in\BD'_{b}$ and $n>n_{0}$, if $b''\notin B_{n}$ and $\BD(b'',\ell_{n}^{-1}R)\subseteq\BD'_{b}$, then there is a point $(x,t_{\infty})\in Z$ with $d((s_{n}(b''),t_{n}),(x,t_{\infty}))<\delta'/2$. By (\ref{key--3}), we have
\begin{align*}
d((s_{n}(b''+\ell_{n}^{-1}z),t_{n}),(\phi_{(x,v)}(z),t_{\infty}))<\delta/2+\delta/2=\delta
\end{align*}
for all $z\in\BD_{R+1}$ and $n>n_{0}$. Now applying (\ref{key--2}) to $\phi(z)=(s_{n}(b''+\ell_{n}^{-1}z),t_{n})$, we get
\begin{align}\label{key--5}
\sum\limits_{t\in\BD(b'',\ell_{n}^{-1}R)}{\rm ord}_{t}(s_{n},t_{n})^{*}\CD=\sum\limits_{z\in\BD_{R}}{\rm ord}_{z}\phi^{*}\CD\geq (1-\frac{\varepsilon}{2})\int_{\BD_{R}}\phi^{*}\omega= (1-\frac{\varepsilon}{2})\int_{\BD(b'',\ell_{n}^{-1}R)}(s_{n},t_{n})^{*}\omega
\end{align}
for all $n>n_{0}$ and all $b''\in\BD'_{b}\backslash B_{n}$ with $\BD(b'',\ell_{n}^{-1}R)\subseteq\BD'_{b}$.

Now let
$$
\mu_{1,n}\colonequals\sum_{b''\in\BD'_{b}}{\rm ord}_{b''}(s_{n},t_{n})^{*}\CD\cdot\delta_{b''},
$$
where $\delta_{b''}$ is the Dirac measure on $\BD'_{b}$ at $b''$. And let $\mu_{2,n}$ be the measure on $\BD'_{b}$ induced by $(1-\frac{\varepsilon}{2})(s_{n},t_{n})^{*}\omega$. Let $\mu_{0}$ be a standard Lebesgue measure on $\BD'_{b}$. By Lemma \ref{refine--} below, for each $n$ there is a subset $\Sigma_{n,b}\subseteq \BD'_{b}$ with $\mu_{0}(\Sigma_{n,b})\leq\frac{\mu_{0}(B_{n}\cap\BD'_{b})+6\ell_{n}^{-1}R\mu_{0}(\BD'_{b})}{\varepsilon/2}$ such that $\mu_{1,n}(\BD'_{b})\geq (1-\frac{\varepsilon}{2})\mu_{2,n}(\BD'_{b}\backslash \Sigma_{n,b})\geq (1-\varepsilon)\int_{\BD'_{b}\backslash \Sigma_{n,b}}(s_{n},t_{n})^{*}\omega$, which is exactly the inequality in the proposition.

It remains to prove that $\mu_{B}(\Sigma_{n,b})\leq\frac{c}{\varepsilon}\mu_{B}(B_{n}\cap\BD'_{b})+\varepsilon'$ for $n$ large enough. By our choice of $c$ at the beginning, we have (after multiplying $\mu_{0}$ by a positive constant)
\begin{align*}
\mu_{B}(\Sigma_{n,b})\leq \sqrt{\frac{c}{2}}\cdot\mu_{0}(\Sigma_{n,b})&\leq \sqrt{\frac{c}{2}}\cdot\frac{2\mu_{0}(B_{n}\cap\BD'_{b})+12\ell_{n}^{-1}R\mu_{0}(\BD'_{b})}{\varepsilon}\\
&\leq \frac{c}{2}\cdot\frac{2\mu_{B}(B_{n}\cap\BD'_{b})+12\ell_{n}^{-1}R\mu_{B}(\BD'_{b})}{\varepsilon}.
\end{align*}
The proposition then follows from the fact that $\frac{6c\ell_{n}^{-1}R\mu_{B}(\BD'_{b})}{\varepsilon}<\varepsilon'$ for $n$ large enough.
\end{proof}

\begin{lemma}\label{refine--}
Let $\BD\subseteq\CC$ be a unit disc with a Lebesgue measure $\mu_{0}$. Let
$S$ be a Lebesgue measurable subset of $\BD$ and let $r>0$ be a positive number. Assume that there are two measures $\mu_{1}$ and $\mu_{2}$ on $\BD$ such that for any $x\in\BD_{1-r}\backslash S$, we have $\mu_{1}(\BD(x,r))\geq\mu_{2}(\BD(x,r))$. Then for any $\varepsilon>0$, there is a subset $\Sigma\subseteq\BD$ with $\mu_{0}(\Sigma)\leq\frac{\mu_{0}(S)+6r\mu_{0}(\BD)}{\varepsilon}$ such that $\mu_{1}(\BD)\geq (1-\varepsilon)\mu_{2}(\BD\backslash\Sigma)$.
\end{lemma}
\begin{proof}
Since the Lebesgue measure on $\BD$ is unique up to a constant multiple, we may assume that $\mu_{0}(\BD)=1$. Then $\mu_{0}(x,r)=r^{2}$ for any open disc $\BD(x,r)\subseteq\BD$. The lemma is trivial for $\varepsilon\geq 1$, so we assume $0<\varepsilon<1$. We also assume $r<\frac{1}{3}$, or otherwise we can just take $\Sigma=\BD$ and the inequality is also trivial.

Let $\Sigma'$ be the set of all $z\in\BD_{1-2r}$ such that $\mu_{0}(\BD(z,r)\backslash S)<(1-\varepsilon)r^{2}$. Then

\begin{align*}
r^{2}((1-2r)^{2}-\mu_{0}(S))&\leq \int_{\BD_{1-r}}\mu_{0}(\BD(x,r)\backslash S){\rm d}\mu_{0,x}\leq \mu_{0}(\Sigma')(1-\varepsilon)r^{2}+((1-r)^{2}-\mu_{0}(\Sigma'))r^{2}
\end{align*}
This implies that $\mu_{0}(\Sigma')\leq \frac{\mu_{0}(S)+2r}{\varepsilon}$. Let $\Sigma=\Sigma'\cup(\BD\backslash\BD_{1-2r})$. Then $\mu_{0}(\Sigma)\leq\mu_{0}(\Sigma')+1-(1-2r)^{2}\leq \frac{\mu_{0}(S)+2r}{\varepsilon}+4r<\frac{\mu_{0}(S)+6r}{\varepsilon}$. Thus we have
\begin{align*}
r^{2}\mu_{1}(\BD)&\geq \int_{\BD_{1-r}\backslash S}\mu_{1}(\BD(x,r)){\rm d}\mu_{0,x}\geq \int_{\BD_{1-r}\backslash S}\mu_{2}(\BD(x,r)){\rm d}\mu_{0,x}\\
&\geq\int_{\BD_{1-2r}}\mu_{0}(\BD(z,r)\backslash S){\rm d}\mu_{2,z}\geq\int_{\BD_{1-2r}\backslash\Sigma'}\mu_{0}(\BD(z,r)\backslash S){\rm d}\mu_{2,z}\\
&\geq (1-\varepsilon)r^{2}\mu_{2}(\BD_{1-2r}\backslash\Sigma')=(1-\varepsilon)r^{2}\mu_{2}(\BD\backslash\Sigma)
\end{align*}
This proves the lemma.
\end{proof}

\begin{proof}[Proof of Proposition \ref{V}(1)]
We first prove the left-hand inequality in Proposition \ref{V}(1). Fix a subset $U\subseteq B(\CC)$ such that either $\mu_{B}(\partial U)=0$ or $U$ is (Euclidean) open. If $\mu_{B}(U)=0$, then the Xie--Yuan height $\hat{h}_{D^{\rm sym},U}((s_{n},t_{n}))$ is always zero and the proximity function $m_{D,U}((s_{n},t_{n}))$ is nonnegative. Hence the inequality is trivial. Now assume that $\mu_{B}(U)>0$. By Proposition \ref{feituihua}, there is a constant $\varepsilon_{U}>0$ such that $h_{\omega,U}\geq \varepsilon_{U}\hat{h}_{D^{\rm sym}}$ on $A(K)\times T(\CC)$. By Proposition \ref{jiben-}, there is $\delta>0$ such that for any subset $S\subseteq B(\CC)$ with $\mu_{B}(S)<\delta$, we have $h_{\omega,S}\leq \frac{\varepsilon\cdot\varepsilon_{U}}{2} \hat{h}_{D^{\rm sym}}$ on $A(K)\times T(\CC)$. Let $c>0$ be the constant in Proposition \ref{key-}. By Lemma \ref{juli}, there is a Euclidean open neighbourhood $\CU$ of $\CD$ in $\CA\times_{\CC}T$ such that for $n$ large enough, the set $B_{n}=(s_{n},t_{n})^{-1}(\CU)\subseteq B(\CC)$ has measure less than $\frac{\varepsilon\delta}{8c}$. For each $b\in B(\CC)\backslash S(A)$, take an open disc $\BD_{b}\subseteq B(\CC)$ satisfying Proposition \ref{key-} (with $\varepsilon$ replaced by $\frac{\varepsilon}{2}$). Let $U^{\circ}$ be the interior of $U$ (in the Euclidean topology). Then since $\mu_{B}(U^{\circ})=\mu_{B}(U\backslash\partial U)=\mu_{B}(U)$, we can choose finitely many points $b_{1},\dots,b_{k}\in U^{\circ}\backslash S(A)$ such that $U^{\circ}\backslash(\cup_{i=1}^{k}\BD_{b_{i}})$ has $\mu_{B}$-measure less than $\frac{\delta}{2}$. Now choose finitely many pairwise disjoint open discs $D_{j}\, (j\in J)$ in $U^{\circ}$, each disc is contained in some $\BD_{b_{i}}$, and $U^{\circ}\backslash(\cup_{j\in J}D_{j})$ has $\mu_{B}$-measure less than $\frac{\delta}{2}$.

Now choose $\varepsilon_{j}>0$ for each $j\in J$ such that $\sum_{j\in J}\varepsilon_{j}<\frac{\delta}{4}$. By Proposition \ref{key-} and our choice in the previous paragraph, for each $j\in J$ and $n$ large enough, there is a subset $\Sigma_{n,j}\subseteq D_{j}$ of measure less than $\frac{c}{\varepsilon/2}\mu_{B}(B_{n}\cap D_{j})+\varepsilon_{j}$ such that
\begin{align*}
\sum\limits_{t\in D_{j}}{\rm ord}_{t}(s_{n},t_{n})^{*}\CD\geq (1-\frac{\varepsilon}{2})\int_{D_{j}\backslash\Sigma_{n,j}}(s_{n},t_{n})^{*}\omega.
\end{align*}
Denote $V_{0}=U^{\circ}\backslash(\cup_{j\in J}D_{j})$. Let $\Sigma_{n}=V_{0}\cup(\cup_{j\in J}\Sigma_{n,j})$. Then we have
\begin{align}\label{final--0}
\sum\limits_{t\in U}{\rm ord}_{t}(s_{n},t_{n})^{*}\CD\geq\sum\limits_{t\in U^{\circ}\backslash V_{0}}{\rm ord}_{t}(s_{n},t_{n})^{*}\CD\geq (1-\frac{\varepsilon}{2})\int_{U^{\circ}\backslash\Sigma_{n}}(s_{n},t_{n})^{*}\omega.
\end{align}
Moreover, we have
\begin{align*}
\mu_{B}(\Sigma_{n})\leq\mu_{B}(V_{0})+\frac{c}{\varepsilon/2}\sum\limits_{j\in J}\mu_{B}(B_{n}\cap D_{j})+\sum\limits_{j\in J}\varepsilon_{j}\leq \frac{\delta}{2}+\frac{2c}{\varepsilon}\mu_{B}(B_{n})+\frac{\delta}{4}\leq &\frac{\delta}{2}+\frac{2c}{\varepsilon}\cdot\frac{\varepsilon\delta}{8c}+\frac{\delta}{4}\\
&=\delta.
\end{align*}
Thus by our assumption, the right hand side of (\ref{final--0}) is
\begin{align*}
(1-\frac{\varepsilon}{2})(h_{\omega,U}((s_{n},t_{n}))-h_{\omega,\Sigma_{n}}((s_{n},t_{n})))&\geq (1-\frac{\varepsilon}{2})(h_{\omega,U}((s_{n},t_{n}))-\frac{\varepsilon\cdot\varepsilon_{U}}{2} \hat{h}_{D^{\rm sym}}((s_{n},t_{n})))\\
&\geq(1-\frac{\varepsilon}{2})(h_{\omega,U}((s_{n},t_{n}))-\frac{\varepsilon}{2} h_{\omega,U}((s_{n},t_{n})))\\
&\geq (1-\varepsilon)h_{\omega,U}((s_{n},t_{n})).
\end{align*}
This proves the left-hand inequality in Proposition \ref{V}(1). For the right-hand one, fix a subset $U\subseteq B(\CC)$ such that either $\mu_{B}(\partial U)=0$ or $U$ is (Euclidean) closed. Let $U'=B(\CC)\backslash U$. Then either $\partial U'=\partial U$ has measure 0, or $U'$ is open. Using the left-hand inequality for $U'$ (with $\varepsilon$ replaced by $\frac{\varepsilon}{2}$), we see that
\begin{align*}
\sum\limits_{t\in U'}{\rm ord}_{t}(s_{n},t_{n})^{*}\CD\geq (1-\frac{\varepsilon}{2})h_{\omega,U'}((s_{n},t_{n}))
\end{align*}
holds for $n$ large enough. Now we have
\begin{align*}
\sum\limits_{t\in U}{\rm ord}_{t}(s_{n},t_{n})^{*}\CD&=\deg_{B}(s_{n},t_{n})^{*}\CD-\sum\limits_{t\in U'}{\rm ord}_{t}(s_{n},t_{n})^{*}\CD\\
&\leq \deg_{B}(s_{n},t_{n})^{*}\CD-(1-\frac{\varepsilon}{2})h_{\omega,U'}((s_{n},t_{n}))\\
&=\hat{h}_{D^{\rm sym}}((s_{n},t_{n}))+O\left(1+\sqrt{\hat{h}_{D^{\rm sym}}((s_{n},t_{n}))}\right)-(1-\frac{\varepsilon}{2})h_{\omega,U'}((s_{n},t_{n}))\\
&= h_{\omega,U}((s_{n},t_{n}))+O\left(1+\sqrt{\hat{h}_{D^{\rm sym}}((s_{n},t_{n}))}\right)+\frac{\varepsilon}{2}h_{\omega,U'}((s_{n},t_{n}))\\
&\leq h_{\omega,U}((s_{n},t_{n}))+\varepsilon \hat{h}_{D^{\rm sym}}((s_{n},t_{n}))
\end{align*}
for $n$ large enough. Here the third line follows from Proposition \ref{bijiao-}. The last inequality follows from the fact that $h_{\omega,U'}((s_{n},t_{n}))\leq h_{\omega}((s_{n},t_{n}))=\hat{h}_{D^{\rm sym}}((s_{n},t_{n}))\to\infty$.

When $\mu_{B}(U)>0$, by Proposition \ref{feituihua}, there is a constant $\varepsilon_{U}>0$ such that $h_{\omega,U}\geq \varepsilon_{U}\hat{h}_{D^{\rm sym}}$ on $A(K)\times T(\CC)$. Hence the term $\varepsilon \hat{h}_{D^{\rm sym}}((s_{n},t_{n}))$ can be bounded above by $\frac{\varepsilon}{\varepsilon_{U}}\hat{h}_{D^{\rm sym}}((s_{n},t_{n}))$. Since $\varepsilon>0$ is arbitrary, the conclusion holds.
\end{proof}

\begin{remark}\label{jianjin}
We can also interpret Proposition \ref{V}(1) from the perspective of measure theory. Explicitly, let
$$
\sum_{b\in B(\CC)}({\rm ord}_{b}(s_{n},t_{n})^{*}\CD)\cdot\delta_{b}\quad {\rm and}\quad (s_{n},t_{n})^{*}\omega
$$
be the measures on $B(\CC)$. Consider their normalized probability measures
$$
\mu_{n}\colonequals\frac{1}{\deg_{B}(s_{n},t_{n})^{*}\CD}\sum_{b\in B(\CC)}({\rm ord}_{b}(s_{n},t_{n})^{*}\CD)\cdot\delta_{b}\quad{\rm and}\quad \mu_{\omega,n}\colonequals \frac{1}{h_{\omega}((s_{n},t_{n}))}(s_{n},t_{n})^{*}\omega
$$
on $B(\CC)$. Note that by Proposition \ref{bijiao-}, we have
$$
\deg_{B}(s_{n},t_{n})^{*}\CD=\hat{h}_{D^{\rm sym}}((s_{n},t_{n}))+O\left( 1+\sqrt{\hat{h}_{D^{\rm sym}}((s_{n},t_{n}))}\right)\sim\int_{B(\CC)}(s_{n},t_{n})^{*}\omega\to\infty.
$$
Hence Proposition \ref{V}(1) just says that for any subset $U\subseteq B(\CC)$ with $\mu_{B}(\partial U)=0$ and any $\varepsilon>0$, the inequality $(1-\varepsilon)\mu_{\omega,n}(U)\leq \mu_{n}(U)\leq \mu_{\omega,n}(U)+\varepsilon$ holds for $n$ large enough.

On the other hand, we can compute the (strong) limit of the measures $\mu_{\omega,n}$ by using properties of canonical Xie--Yuan heights. Namely, by quadraticity, we have $h_{\omega,S}(ms,t)=m^{2}h_{\omega,S}(s,t)$ for all Lebesgue measurable subsets $S\subseteq B(\CC)$, all integers $m$ and all $(s,t)\in V(A,K)\times T(\CC)$. Recall that we have assumed $\ell_{n}^{-1}s_{n}\to s_{\infty}\neq0$ in $V(A,K)$. Let
$$
\mu_{\omega,\infty}\colonequals \frac{1}{h_{\omega}((s_{\infty},t_{\infty}))}(s_{\infty},t_{\infty})^{*}\omega.
$$
Then for each Lebesgue measurable subset $S\subseteq B(\CC)$, we have
\begin{align*}
\mu_{\omega,n}(S)=\frac{h_{\omega,S}((s_{n},t_{n}))}{h_{\omega}((s_{n},t_{n}))}=\frac{h_{\omega,S}((\ell_{n}^{-1}s_{n},t_{n}))}{h_{\omega}((\ell_{n}^{-1}s_{n},t_{n}))}
\end{align*}
which converges to $\frac{h_{\omega,S}((s_{\infty},t_{\infty}))}{h_{\omega}((s_{\infty},t_{\infty}))}=\mu_{\omega,\infty}(S)$ by the continuity of the canonical Xie--Yuan heights. This implies that $\mu_{\omega,n}$ converges setwise (or strongly) to $\mu_{\omega,\infty}$.

Now combining with the inequality above, we see that $\mu_{n}(U)$ converges to $\mu_{\omega,\infty}(U)$ for every subset $U\subseteq B(\CC)$ with $\mu_{B}(\partial U)=0$ (which is equivalent to $\mu_{\omega,\infty}(\partial U)=0$). Therefore by the Portmanteau theorem, the measure $\mu_{n}$ converges weakly to $\mu_{\omega,\infty}$.
\end{remark}
\subsection{Upper bound for the GCD counting functions}\label{4.3}
Now we begin to prove Proposition \ref{V}(3), which gives an upper bound for the GCD height function. The method is similar to the previous subsection: we first give ``microscopic local'' estimates (see (\ref{key-5}) and (\ref{key-6})), and then use a covering lemma to obtain the ``local'' estimates (see Proposition \ref{key} below), which imply the global estimate by combining with Lemma \ref{juli} and Proposition \ref{jiben-}.

The new problem is that, when we deduce the local estimate from microscopic local estimate, the Lemma \ref{refine--} before only implies that the sum of local GCD counting functions outside a small subset (which we call the ``bad'' locus) is very small compared to the Xie--Yuan height function (see the first inequality in Proposition \ref{key} below). For the sum of local GCD counting functions inside this bad locus (which we call the ``bad sum''), however, it is hard to give an upper bound directly. To treat this, we instead prove a converse. We choose a divisor $\CD$ containing the higher-codimensional subscheme $\CY$ in $\CA\times_{\CC}T$. Then the bad sum above is bounded by the sum of local counting functions associated to $\CD$ inside this bad locus, which also equals the height minus the sum of local height functions associated to $\CD$ outside the bad locus. The upper bound of the last sum was given in Proposition \ref{V}(1), which was proved in the previous subsection.

However, there is still a problem of the strategy above. In fact, the bad locus mentioned above depends on $n$, and the upper bound of sum over a given subset of local height functions associated to $\CD$ only holds for $n$ large enough. Therefore, if the bad locus changes as $n$ grows, we can not apply the upper bound in (1). The solution of this is that, rather than using (1) directly, we prove the upper bound of good GCD function and the lower bound of bad counting function associated to $\CD$ simultaneously. See the two inequalities in Proposition \ref{key} below. The key is that the bad locus ``$B_{n}$'' in two inequalities are the same. The choice of $B_{n}$ is based on a refined version of Lemma \ref{refine--}; see Lemma \ref{refine} below. \medskip

\noindent Now we start the proof. Recall that $\ell_{n}^{-1}s_{n}\to s_{\infty}\neq0$ in $V(A,K)$ and $G=G(s_{\infty})$. Let $\CG$ be the Zariski closure of $G$ in $\CA$. By Lemma \ref{closure}, there is a subset $S(A)$ of $B(\CC)$ of measure 0 such that for any $b\in B(\CC)\backslash S(A)$, the Zariski closure of the linear leaf $\CF(0,\delta(s_{\infty},v_{b}))$ in $\CA_{b}$ equals $\CG_{b}$.

Let $Y'$ be the set of all points $(a,t)\in A\times T_{K}$ such that $((a+G)\times\{t\})\cap Y$ has codimension at most 1. Since $Y$ has codimension at least 2 in $A\times T_{K}$, $Y'$ is a proper Zariski closed subset of $A\times T_{K}$. Let $\CY'$ be the Zariski closure of $Y'$ in $\CA\times_{\CC}T$. Then for all $b\in B$ outside a finite subset $S_{1}\subseteq B(\CC)$, the set of all points $(x,t)\in\CA_{b}\times_{\CC} T$ such that $((x+\CG_{b})\times\{t\})\cap\CY$ has codimension at most 1 is exactly $\CY'_{b}$. Let $S'$ be the union of $S(A),S_{1}$, which has measure 0.

Note that we have assumed in \S\ref{4.1} that $\CY$ is flat over $T$, and thus each fibre $\CY_{t}$ has codimension at least 2 in $\CA\times\{t\}$. Combining with the definition of $\CY'$, we see that $\CY'$ does not contain any fibre $\CA\times\{t\}$.

\begin{proposition}\label{key}
There is a constant $c>0$ depending only on $\mu_{B}$ satisfying the following property: For a Euclidean open neighbourhood $\CU$ of $\CD\cup\CY'$, denote $B_{n}\colonequals (s_{n},t_{n})^{-1}(\CU)$. Then for any point $b\in B(\CC)\backslash S'$ and any $\varepsilon>0$, there is an open disc $\BD_{b}\subseteq B(\CC)$ (depending on $U$) containing $b$ such that for any open disc $\BD'_{b}$ contained in $\BD_{b}$, any $\varepsilon'>0$ and any $n$ large enough (depending on $\BD'_{b}$ and $\varepsilon'$), there exist subsets $B'_{n,b},\Sigma_{n,b}\subseteq \BD'_{b}$ with $\mu_{B}(\Sigma_{n,b})\leq\frac{c}{\varepsilon}\mu_{B}(B_{n}\cap\BD'_{b})+\varepsilon'$ satisfying
\begin{align*}
\sum\limits_{t\in\BD'_{b}\backslash B'_{n,b}}{\rm ord}_{t}(s_{n},t_{n})^{*}\CY\leq \varepsilon\int_{\BD'_{b}}(s_{n},t_{n})^{*}\omega
\end{align*}
and
\begin{align*}
\sum\limits_{t\in\BD'_{b}\backslash B'_{n,b}}{\rm ord}_{t}(s_{n},t_{n})^{*}\CD\geq (1-\varepsilon)\int_{\BD'_{b}\backslash\Sigma_{n,b}}(s_{n},t_{n})^{*}\omega.
\end{align*}
\end{proposition}

\begin{proof}
Take a finite cover of $B(\CC)$ by open discs. For any disc $\BD$ in this cover, fix a standard Lebesgue measure $\mu_{0,\BD}$ on $\BD$. Then there is a constant $c_{\BD}>1$ such that $c_{\BD}^{-1}\mu_{0,\BD}(\Omega)\leq\mu_{B}(\Omega)\leq c_{\BD}\mu_{0,\BD}(\Omega)$ for any Lebesgue measurable subset $\Omega\subseteq \BD$. Take $c=2C_{2}\max_{\BD}\{c_{\BD}^{2}\}>0$, where $\BD$ is taken over all open discs in this finite cover, and $C_{2}$ is the absolute constant in Lemma \ref{refine}(2) below. We show that $c$ satisfies the conditions in the proposition.

Fix a point $b\in B(\CC)\backslash S'$, a Euclidean open neighbourhood $\CU$ of $\CD\cup\CY'$ and a constant $\varepsilon>0$. Then $\delta(s_{\infty},v_{b})\neq0$. Applying Proposition \ref{conclusion} to the abelian subvariety $\CA_{b}\times\{t_{\infty}\}\subseteq\CA\times_{\CC}T$, the nonzero vector $v=\delta(s_{\infty},v_{b})$ and the Euclidean closed subset $Z=\CA_{b}\times\{t_{\infty}\}\backslash \CU_{b,t_{\infty}}$, we get constants $R>0, \delta>0$ satisfying the following condition: For any point $(x,t_{\infty})\in Z$ and any holomorphic map $\phi\colon \BD_{R+1}\to \CA\times_{\CC}T$, if $d(\phi(z), (\phi_{(x,v)}(z),t_{\infty}))<\delta$ for all $z\in\BD_{R+1}$, then
\begin{align}\label{key-1}
\sum\limits_{z\in\BD_{R}}{\rm ord}_{z}\phi^{*}\CY\leq \frac{\varepsilon}{C_{1}}\int_{\BD_{R}}\phi^{*}\omega
\end{align}
and
\begin{align}\label{key-2}
\sum\limits_{z\in\BD_{R}}{\rm ord}_{z}\phi^{*}\CD\geq(1-\frac{\varepsilon}{2})\int_{\BD_{R}}\phi^{*}\omega.
\end{align}
Here $C_{1}>0$ is the absolute constant in Lemma \ref{refine}(1).

Take any open unit disc $\BD$ with centre $b$ which is contained in some open disc of the finite cover chosen at the beginning of the proof. Let $z$ be the standard coordinate of $\BD$. By Lemma \ref{linear}, there are constants $0<\delta'<1$ and $n_{0}>0$ such that for each $b'\in\BD_{\delta'}$, each $z\in\BD_{R+1}$ and each $x\in\CA_{b}=\CA_{0}$, if $n>n_{0}$ and $d(s_{n}(b'),x)<\delta'$, then $b'+\ell_{n}^{-1}z\in\BD$ and
\begin{align}\label{key-3}
d_{\CA}(s_{n}(b'+\ell_{n}^{-1}z),\phi_{(x,v)}(z))<\delta/2
\end{align}
Enlarging $n_{0}$ and replacing $\delta'$ by a smaller positive number if necessary, we can assume that for any $n>n_{0}$ and any point $(x',t_{n})\in(\CA_{\BD_{\delta'}}\times\{t_{n}\})\backslash \CU$, there is a point $(x,t_{\infty})\in Z$ with $d((x',t_{n}),(x,t_{\infty}))<\delta'/2$. In particular, we have $d_{T}(t_{n},t_{\infty})<\delta'/2<\delta/2$ for all $n>n_{0}$.

Now take $\BD_{b}=\BD_{\delta'}$. We prove that it satisfies the conditions in the Proposition.

In fact, for any open disc $\BD'_{b}\subseteq \BD_{b}$, assume $\BD'_{b}=\BD(b',r')$. For $b''\in\BD'_{b}$ and $n>n_{0}$, if $b''\notin B_{n}$ and $\BD(b'',\ell_{n}^{-1}R)\subseteq\BD'_{b}$, then there is a point $(x,t_{\infty})\in Z$ with $d((s_{n}(b''),t_{n}),(x,t_{\infty}))<\delta'/2$. By (\ref{key-3}), we have
\begin{align*}
d((s_{n}(b''+\ell_{n}^{-1}z),t_{n}),(\phi_{(x,v)}(z),t_{\infty}))<\delta/2+\delta/2=\delta
\end{align*}
for all $z\in\BD_{R+1}$ and $n>n_{0}$. Now applying (\ref{key-1}) and (\ref{key-2}) to $\phi(z)=(s_{n}(b''+\ell_{n}^{-1}z),t_{n})$, we get
\begin{align}\label{key-5}
\sum\limits_{t\in\BD(b'',\ell_{n}^{-1}R)}{\rm ord}_{t}(s_{n},t_{n})^{*}\CY\leq \frac{\varepsilon}{C_{1}}\int_{\BD(b'',\ell_{n}^{-1}R)}(s_{n},t_{n})^{*}\omega
\end{align}
and
\begin{align}\label{key-6}
\sum\limits_{t\in\BD(b'',\ell_{n}^{-1}R)}{\rm ord}_{t}(s_{n},t_{n})^{*}\CD\geq (1-\frac{\varepsilon}{2})\int_{\BD(b'',\ell_{n}^{-1}R)}(s_{n},t_{n})^{*}\omega
\end{align}
for all $n>n_{0}$ and all $b''\in\BD'_{b}\backslash B_{n}$ with $\BD(b'',\ell_{n}^{-1}R)\subseteq\BD'_{b}$.

Now let
$$
\mu_{1,n}\colonequals\sum_{b''\in\BD'_{b}}{\rm ord}_{b''}(s_{n},t_{n})^{*}\CY\cdot\delta_{b''},
$$
where $\delta_{b''}$ is the Dirac measure on $\BD'_{b}$ at $b''$. Similarly, let
$$
\mu'_{1,n}\colonequals\sum_{b''\in\BD'_{b}}{\rm ord}_{b''}(s_{n},t_{n})^{*}\CD\cdot\delta_{b''}.
$$
Finally, let $\mu_{2,n}$ be the measure on $\BD'_{b}$ induced by $(s_{n},t_{n})^{*}\omega$. Let $\mu_{0}$ be a standard Lebesgue measure on $\BD'_{b}$. By Lemma \ref{refine} below, for each $n$ there are subsets $B'_{n,b},\Sigma_{n,b}\subseteq \BD'_{b}$ with $\mu_{0}(\Sigma_{n,b})\leq\frac{C_{2}\mu_{0}(B_{n}\cap\BD'_{b})+10\ell_{n}^{-1}R\mu_{0}(\BD'_{b})}{\varepsilon/2}$ such that

\begin{enumerate}
\item[(1)] $\mu_{1,n}(\BD'_{b}\backslash B'_{n,b})\leq C_{1}\cdot\frac{\varepsilon}{C_{1}}\mu_{2,n}(\BD'_{b})=\varepsilon\cdot\mu_{2,n}(\BD'_{b})$ and
\item[(2)] $\mu'_{1,n}(\BD'_{b}\backslash B'_{n,b})\geq (1-\frac{\varepsilon}{2})\cdot(1-\frac{\varepsilon}{2})\mu_{2,n}(\BD'_{b}\backslash \Sigma_{n,b})\geq (1-\varepsilon)\mu_{2,n}(\BD'_{b}\backslash \Sigma_{n,b})$,
 \end{enumerate}
which are exactly the inequalities in the proposition.

It remains to prove that $\mu_{B}(\Sigma_{n,b})\leq\frac{c}{\varepsilon}\mu_{B}(B_{n}\cap\BD'_{b})+\varepsilon'$ for $n$ large enough. By our choice of $c$, we have (after multiplying $\mu_{0}$ by a positive constant)
\begin{align*}
\mu_{B}(\Sigma_{n,b})\leq \sqrt{\frac{c}{2C_{2}}}\cdot\mu_{0}(\Sigma_{n,b})&\leq \sqrt{\frac{c}{2C_{2}}}\cdot\frac{2C_{2}\mu_{0}(B_{n}\cap\BD'_{b})+20\ell_{n}^{-1}R\mu_{0}(\BD'_{b})}{\varepsilon}\\
&\leq \frac{c}{2C_{2}}\cdot\frac{2C_{2}\mu_{B}(B_{n}\cap\BD'_{b})+20\ell_{n}^{-1}R\mu_{B}(\BD'_{b})}{\varepsilon}.
\end{align*}
The proposition then follows from the fact that $\frac{10c\ell_{n}^{-1}R\mu_{B}(\BD'_{b})}{C_{2}\varepsilon}<\varepsilon'$ for $n$ large enough.
\end{proof}
In the above proof, we have used the following lemma, which is a combination and a refinement of Lemma \ref{besicovitch} and Lemma \ref{refine--}. The point is that in both of the estimates below, we can use the same set $S_{r}$ on the left hand side.
\begin{lemma}\label{refine}
Let $\BD\subseteq\CC$ be a unit disc with a Lebesgue measure $\mu_{0}$. For any $0<r<1$ and any measurable subset $S\subseteq\BD$, denote by $S_{r}\colonequals\{x\in \BD\big|\,\BD(x,\frac{r}{2})\cap\BD_{1-r}\subseteq S\}\subseteq \BD$. Then
\begin{enumerate}[label=\normalfont(\arabic*),leftmargin=*]
\item There is an absolute constant $C_{1}>0$ satisfying the following property: If there are two measures $\mu_{1}$ and $\mu_{2}$ on $\BD$ such that for any $x\in\BD_{1-r}\backslash S$, we have $\mu_{1}(\BD(x,r))\leq\mu_{2}(\BD(x,r))$, then $\mu_{1}(\BD\backslash S_{r})\leq C_{1}\cdot\mu_{2}(\BD)$.

\item There is an absolute constant $C_{2}>0$ satisfying the following property: If there are two measures $\mu_{1}$ and $\mu_{2}$ on $\BD$ such that for any $x\in\BD_{1-r}\backslash S$, we have $\mu_{1}(\BD(x,r))\geq\mu_{2}(\BD(x,r))$. Then for any $\varepsilon>0$, there is a subset $\Sigma\subseteq\BD$ with $\mu_{0}(\Sigma)\leq\frac{C_{2}\mu_{0}(S)+10r\mu_{0}(\BD)}{\varepsilon}$ such that $\mu_{1}(\BD\backslash S_{r})\geq (1-\varepsilon)\mu_{2}(\BD\backslash\Sigma)$.
\end{enumerate}
\end{lemma}

\begin{proof}
Similarly as in the proof of Lemma \ref{refine--}, we may assume that $\mu_{0}(\BD)=1, 0<\varepsilon<1$ and $r<\frac{1}{3}$. We first prove (1). Note that $\cup_{x\in\BD_{1-r}\backslash S}\BD(x,\frac{r}{2})$ is a Besicovitch cover of $\BD_{1-r}\backslash S$. Let $C$ be the Besicovitch constant in dimension 2. Then there is a subset $\Omega$ of $\BD_{1-r}\backslash S$ such that the union $\cup_{x\in\Omega}\BD(x,\frac{r}{2})$ contains $\BD_{1-r}\backslash S$, and for each point $y$ in this union, the number of $x\in\Omega$ with $y\in\BD(x,\frac{r}{2})$ is at most $C$.

Now for each $x\in\BD\backslash S_{r}$, by definition, the intersection $\BD(x,\frac{r}{2})\cap(\BD_{1-r}\backslash S)$ is nonempty. Thus there is a point $x'\in\BD_{1-r}\backslash S$ with $x\in\BD(x',\frac{r}{2})$. By our choice of $\Omega$, there is a point $x''\in\Omega\subseteq\BD_{1-r}\backslash S$ with $x'\in\BD(x'',\frac{r}{2})$. Therefore we have $x\in\BD(x'',r)$. In other words, the union $\cup_{x\in\Omega}\BD(x,r)$ covers $\BD\backslash S_{r}$. Moreover, this union is contained in $\BD$ since $\Omega\subseteq\BD_{1-r}$. Note also that an open ball of radius $r$ can be covered by 7 open balls of radius $\frac{r}{2}$. As each open ball of radius $\frac{r}{2}$ contains at most $C$ points in $\Omega$, we conclude that $\BD(y,r)$ contains at most $7C$ points in $\Omega$, hence the inequality
$$
\mu_{1}(\BD\backslash S_{r})\leq\sum\limits_{x\in\Omega}\mu_{1}(\BD(x,r))\leq\sum\limits_{x\in\Omega}\mu_{2}(\BD(x,r))\leq 7C\mu_{2}(\BD).
$$

Now we come to (2). Denote by $S_{r}'\colonequals (\cup_{x\in S_{r}}\BD(x,r))\cap\BD$. We first prove that $\mu_{0}(S'_{r})\leq 4r+9C\mu_{0}(S)$. In fact, note that $S_{r}\subseteq (\BD\backslash\BD_{1-r})\cup (\BD_{1-r}\cap S_{r})$. Therefore $S'_{r}\subseteq (\BD\backslash\BD_{1-2r})\cup(\cup_{x\in\BD_{1-r}\cap S_{r}}\BD(x,r))$. Use the Besicovitch covering theorem again, we can choose a subset $\Omega'\subseteq \BD_{1-r}\cap S_{r}$ such that the union $\cup_{x\in\Omega'}\BD(x,\frac{r}{2})$ covers $\BD_{1-r}\cap S_{r}$ and for each point $y$ in this union, there are at most $C$ points $x\in\Omega'$ with $y\in\BD(x,\frac{r}{2})$. By definition, for each point $x\in\BD_{1-r}\cap S_{r}$, we have $\BD(x,\frac{r}{2})\subseteq S$, therefore we have $\cup_{x\in\Omega'}\BD(x,\frac{r}{2})\subseteq S$. In particular, we have the inequality
\begin{align}\label{refine-1}
\sum\limits_{x\in\Omega'}\mu_{0}(\BD(x,\frac{r}{2}))\leq C\mu_{0}(S)
\end{align}
Now for each point $x\in\BD_{1-r}\cap S_{r}$, there is a point $x'\in\Omega'$ with $x\in\BD(x',\frac{r}{2})$. Hence we have $\BD(x,r)\subseteq\BD(x',\frac{3r}{2})$. This implies that
\begin{align}\label{refine-2}
\mu_{0}(\cup_{x\in\BD_{1-r}\cap S_{r}}\BD(x,r))\leq \mu_{0}(\cup_{x\in\Omega'}\BD(x,\frac{3r}{2}))\leq\sum\limits_{x\in\Omega'}\mu_{0}(\BD(x,\frac{3r}{2}))=9\sum\limits_{x\in\Omega'}\mu_{0}(\BD(x,\frac{r}{2}))
\end{align}
Combining (\ref{refine-1}) and (\ref{refine-2}), we see that $\mu_{0}(S'_{r})\leq 1-(1-2r)^{2}+9C\mu_{0}(S)<4r+9C\mu_{0}(S)$.

Now let $\Sigma'\subseteq\BD$ be the set of all points $z\in\BD_{1-2r}$ with $\mu_{0}(\BD(z,r)\backslash S'_{r})<(1-\varepsilon)r^{2}$. Then $\mu_{0}(\Sigma')\leq \frac{\mu_{0}(S'_{r})+2r}{\varepsilon}\leq\frac{9C\mu_{0}(S)+6r}{\varepsilon}$. Let $\Sigma=\Sigma'\cup(\BD\backslash\BD_{1-2r})$. Then we have $\mu_{0}(\Sigma)=\mu_{0}(\Sigma')+\mu_{0}(\BD\backslash\BD_{1-2r})\leq \frac{9C\mu_{0}(S)+6r}{\varepsilon}+1-(1-2r)^{2}\leq\frac{9C\mu_{0}(S)+10r}{\varepsilon}$, and

\begin{align*}
r^{2}\mu_{1}(\BD\backslash S_{r})&\geq r^{2}\mu_{1}(\BD\backslash S'_{r})\geq \int_{\BD_{1-r}\backslash S'_{r}}\mu_{1}(\BD(x,r)){\rm d}\mu_{0,x}\geq \int_{\BD_{1-r}\backslash S'_{r}}\mu_{2}(\BD(x,r)){\rm d}\mu_{0,x}\\
&\geq\int_{\BD_{1-2r}\backslash\Sigma'}\mu_{0}(\BD(z,r)\backslash S'_{r}){\rm d}\mu_{2,z}\geq (1-\varepsilon)r^{2}\mu_{2}(\BD_{1-2r}\backslash\Sigma')=(1-\varepsilon)r^{2}\mu_{2}(\BD\backslash\Sigma).
\end{align*}
This completes the proof by taking $C_{2}=9C$.
\end{proof}

\begin{proof}[Proof of Proposition \ref{V}(3)]
Note that the truth of the statement does not depend on the choice of the $T_{K}$-relatively ample divisor $D$ and the height functions. Therefore, after replacing $D$ by an ample symmetric effective divisor on $A\times T_{K}$ whose support contains $Y$, we may assume that ${\rm Supp}\,D$ itself is symmetric (over $T_{K}$) and contains $Y$. Then $\CD$ also contains $\CY$. Multiplying $D$ by a positive constant if necessary, we may assume that the ideal sheaf $\CI_{\CD}$ of $\CD$ in $\CA\times_{\CC}T$ is contained in the ideal sheaf $\CI_{\CY}$. In particular, we have ${\rm ord}_{b}(s,t)^{*}\CD\geq{\rm ord}_{b}(s,t)^{*}\CY$ for any $b\in B(\CC)$ and any section $(s,t)\colon B\to\CA\times_{\CC}T$ which is not contained in $\CD$. We fix the height functions
$$
h_{Y}((s,t))=\sum\limits_{b\in B(\CC)}{\rm ord}_{b}(s,t)^{*}\CY\quad{\rm and}\quad h_{D}((s,t))=\sum\limits_{b\in B(\CC)}{\rm ord}_{b}(s,t)^{*}\CD
$$
in the following.

By Proposition \ref{jiben-}, there is $\delta>0$ such that for any subset $S\subseteq B(\CC)$ with $\mu_{B}(S)<\delta$, we have $h_{\omega,S}\leq \frac{\varepsilon}{4} \hat{h}_{D}$ on $A(K)\times T(\CC)$. Let $c>0$ be the constant in Proposition \ref{key}. By Lemma \ref{juli}, there is a Euclidean open neighbourhood $\CU$ of $\CD\cup\CY'$ in $\CA\times_{\CC}T$ such that for $n$ large enough, the set $B_{n}=(s_{n},t_{n})^{-1}(\CU)\subseteq B(\CC)$ has measure less than $\frac{\varepsilon\delta}{16c}$. For each $b\in B(\CC)\backslash S'$, take an open disc $\BD_{b}\subseteq B(\CC)$ satisfying Proposition \ref{key} (with $\varepsilon$ replaced by $\frac{\varepsilon}{4}$). Then we can choose finitely many points $b_{1},\dots,b_{k}\in B(\CC)\backslash S'$ such that $B(\CC)\backslash(\cup_{i=1}^{k}\BD_{b_{i}})$ has measure less than $\frac{\delta}{2}$. Now choose finitely many pairwise disjoint open discs $D_{j}\, (j\in J)$ in $B(\CC)$, each disc is contained in some $\BD_{b_{i}}$, and $B(\CC)\backslash(\cup_{j\in J}D_{j})$ has measure less than $\frac{\delta}{2}$.

Now choose $\varepsilon_{j}>0$ for each $j\in J$ such that $\sum_{j\in J}\varepsilon_{j}<\frac{\delta}{4}$. By Proposition \ref{key} and our choice in the previous paragraph, for each $j\in J$ and $n$ large enough, there are subsets $B_{n,j},\Sigma_{n,j}\subseteq D_{j}$ of measure less than $\frac{c}{\varepsilon/4}\mu_{B}(B_{n}\cap D_{j})+\varepsilon_{j}$ such that
\begin{align*}
\sum\limits_{t\in D_{j}\backslash B_{n,j}}{\rm ord}_{t}(s_{n},t_{n})^{*}\CY\leq \frac{\varepsilon}{4}\int_{D_{j}}(s_{n},t_{n})^{*}\omega
\end{align*}
and
\begin{align*}
\sum\limits_{t\in D_{j}\backslash B_{n,j}}{\rm ord}_{t}(s_{n},t_{n})^{*}\CD\geq (1-\frac{\varepsilon}{4})\int_{D_{j}\backslash\Sigma_{n,j}}(s_{n},t_{n})^{*}\omega.
\end{align*}
Denote $V_{0}=B(\CC)\backslash(\cup_{j\in J}D_{j})$. Let $V_{n}=V_{0}\cup(\cup_{j\in J}B_{n,j})$ and let $\Sigma_{n}=V_{0}\cup(\cup_{j\in J}\Sigma_{n,j})$. Then we have
\begin{align*}
\sum\limits_{t\in B(\CC)\backslash V_{n}}{\rm ord}_{t}(s_{n},t_{n})^{*}\CY\leq \frac{\varepsilon}{4}\int_{B(\CC)}(s_{n},t_{n})^{*}\omega=\frac{\varepsilon}{4} \hat{h}_{D}((s_{n},t_{n})).
\end{align*}
and
\begin{align}\label{final-0}
\sum\limits_{t\in B(\CC)\backslash V_{n}}{\rm ord}_{t}(s_{n},t_{n})^{*}\CD\geq (1-\frac{\varepsilon}{4})\int_{B(\CC)\backslash\Sigma_{n}}(s_{n},t_{n})^{*}\omega.
\end{align}
Moreover, we have
\begin{align*}
\mu_{B}(\Sigma_{n})\leq\mu_{B}(V_{0})+\frac{c}{\varepsilon/4}\sum\limits_{j\in J}\mu_{B}(B_{n}\cap D_{j})+\sum\limits_{j\in J}\varepsilon_{j}\leq & \frac{\delta}{2}+\frac{4c}{\varepsilon}\mu_{B}(B_{n})+\frac{\delta}{4}\\
&\leq\frac{\delta}{2}+\frac{4c}{\varepsilon}\cdot\frac{\varepsilon\delta}{16c}+\frac{\delta}{4}=\delta.
\end{align*}
Thus by our assumption, the right hand side of (\ref{final-0}) is at least $(1-\frac{\varepsilon}{4})^{2}\hat{h}_{D}((s_{n},t_{n}))\geq(1-\frac{\varepsilon}{2})\hat{h}_{D}((s_{n},t_{n}))$, and we have
\begin{align*}
h_{Y}((s_{n},t_{n}))&=\sum\limits_{b\in B(\CC)}{\rm ord}_{b}(s_{n},t_{n})^{*}\CY\\
&=\sum\limits_{b\in B(\CC)\backslash V_{n}}{\rm ord}_{b}(s_{n},t_{n})^{*}\CY+\sum\limits_{b\in V_{n}}{\rm ord}_{b}(s_{n},t_{n})^{*}\CY\\
&\leq \frac{\varepsilon}{4}\hat{h}_{D}((s_{n},t_{n}))+\sum\limits_{b\in V_{n}}{\rm ord}_{b}(s_{n},t_{n})^{*}\CD\\
&= \frac{\varepsilon}{4}\hat{h}_{D}((s_{n},t_{n}))+h_{D}((s_{n},t_{n}))-\sum\limits_{b\in B(\CC)\backslash V_{n}}{\rm ord}_{b}(s_{n},t_{n})^{*}\CD\\
&\leq \frac{\varepsilon}{4}\hat{h}_{D}((s_{n},t_{n}))+h_{D}((s_{n},t_{n}))-(1-\frac{\varepsilon}{2})\hat{h}_{D}((s_{n},t_{n}))\\
&=\frac{3\varepsilon}{4}\hat{h}_{D}((s_{n},t_{n}))+O(1)\leq\varepsilon h_{D}((s_{n},t_{n}))
\end{align*}
for $n$ large enough. Here the last line is due to $h_{D}((s_{n},t_{n}))-\hat{h}_{D}((s_{n},t_{n}))=O(1)$, which follows from Proposition \ref{uniform}. This completes the proof of Proposition \ref{V}(3).
\end{proof}

\subsection{Lower bound for the truncated counting functions}\label{4.4}
In this subsection we prove Proposition \ref{V}(2), which replaces the ordinary counting function in (1) by the truncated version when $D$ is reduced. The proof is still similar to that of (1) in \S\ref{fenbu}, but with two new technical issues.

First, we need to find a way to show that $(s_{n},t_{n})(B)$ and $\CD$ have many transversal intersections. The idea is similar as before: the sections $(s_{n},t_{n})(B)$ can be (locally) approximated by some linear entire curves in some fibre $\CA_{b}\times\{t_{\infty}\}$, and Yamanoi's second main theorem then shows that this linear entire curve has many transversal intersections with any reduced divisor on it. Therefore we may deduce the truncated version when the approximation is close enough. However, even if $D$ is reduced and flat over $T_{K}$, the fibre $D_{t_{\infty}}$ may not be reduced, and thus $\CD_{b,t_{\infty}}$ may not be reduced. To overcome this, we use a limit tangent space lemma proved in \cite[Lemma 3.7]{arXiv:2511.17010}, which says that even though $\CD_{b,t_{\infty}}$ may not be reduced, the limit of tangent bundles of $\CD_{b,t_{n}}$ is almost the tangent space of $(\CD_{b,t_{\infty}})_{\rm red}$.

We mention that \cite{UU,CDMZ} also bounded the intersection multiplicities of sections of elliptic surfaces with torsion sections. Here we use a different approach to treat general intersection multiplicities on abelian schemes.

Second, as in the proof in \S\ref{fenbu}, we need to use Lemma \ref{juli} for some closed subscheme of $\CA\times_{\CC}T$. However, in the proof of truncated version, we need to vary this closed subscheme as $b\in B(\CC)$ varies. See the statement of Proposition \ref{key-abc} below, where we mention a subset $Y_{b}$ for each $b$ and the union $\cup_{b\in B(\CC)}Y_{b}$ may not be algebraic in general. This leads to a difficulty to derive the global estimate from the local one. To treat this, we make a refinement of Lemma \ref{juli} (see Lemma \ref{juli-}). This refinement not only gives an upper bound of the measure of the whole inverse image $(s_{n},t_{n})^{-1}(\CU)$ for $U$ small enough, but also gives a uniform estimate of $\mu_{B}(\BD\cap(s_{n},t_{n})^{-1}(\CU))$ for all open discs $\BD\subseteq B(\CC)$. The proof is essentially the same as Lemma \ref{juli}. That is, use induction to show that the inverse image is contained in the union of some open discs with small radius. In this refined version we need to further estimate the number of points in $\BD\cap(s_{n},t_{n})^{-1}(\CD)$, which has been done in Proposition \ref{V}(1).\medskip\medskip

\noindent Now we start the proof. As $D$ is reduced and $\{t_{n}\}_{n\geq1}$ is generic in $T$, by the generic reducedness, we see that $\CD_{t_{n}}$ is reduced for $n$ large enough. By \cite[Lemma 3.7, Remark 3.9]{arXiv:2511.17010}, there is a Zariski open subset $\CD^{\circ}_{t_{\infty}}$ of $\CD_{t_{\infty}}$ such that for any sequence of points $x_{n}\in\CD_{t_{n}}\,(n\geq1)$ and any sequence of tangent vectors $v_{n}\in T_{x_{n}}\CD_{t_{n}}\,(n\geq1)$, if $\{v_{n}\}_{n\geq1}$ has a limit point $(x_{\infty},v_{\infty})$ in the tangent bundle $T\CD$ with $x_{\infty}\in\CD^{\circ}_{t_{\infty}}$, then we have $v_{\infty}\in T_{x_{\infty}}(\CD_{t_{\infty}})_{\rm red}$.

Let $S_{\rm nr}$ be the finite set of all points $b\in B(\CC)$ such that the fibre over $b$ of the natural map $(\CD_{t_{\infty}})_{\rm red}\to B$ is non-reduced.

Recall that $S(A)$ is a subset of $B(\CC)$ of measure 0 such that for any $b\in B(\CC)\backslash S(A)$, the Zariski closure of the linear leaf $\CF(0,\delta(s_{\infty},v_{b}))$ in $\CA_{b}$ equals $\CG_{b}$.

Fix a point $b\in B(\CC)\backslash S(A)$. Let $W_{0,b}\subseteq\CD_{b,t_{\infty}}$ be the Zariski closed subset consisting of all points $(x,t_{\infty})$ in $\CD_{b,t_{\infty}}$ such that the linear leaf $\CF(x,\delta(s_{\infty},v_{b}))\times\{t_{\infty}\}$ is tangent to $(\CD_{b,t_{\infty}})_{\rm red}$ at $(x,t_{\infty})$. This includes the case where the leaf is entirely contained in $(\CD_{b,t_{\infty}})_{\rm red}$, so that every point of the leaf belongs to $W_{0,b}$.

Let $W_{0,b}^{\rm bad}$ be the union of the irreducible components of $W_{0,b}$ which are also irreducible components of $\CD_{b,t_{\infty}}$; equivalently, these are precisely the components of $\CD_{b,t_{\infty}}$ invariant under translation by $\CG_{b}\times\{t_{\infty}\}$. Let $W_{0,b}^{\rm good}$ be the union of the remaining irreducible components of $W_{0,b}$. Note that $W_{0,b}^{\rm bad}$ and $W_{0,b}^{\rm good}$ may have nonempty intersection. By definition, the subset $W_{0,b}^{\rm good}$ is closed of codimension at least 1 in $\CD_{b,t_{\infty}}$.

Finally, let $W_{b}=W_{0,b}^{\rm good}\cup(\CD_{b,t_{\infty}}\backslash\CD^{\circ}_{t_{\infty}})$. Then $W_{b}$ is of codimension at least 1 in $\CD_{b,t_{\infty}}$, hence of codimension at least 2 in $\CA_{b}\times\{t_{\infty}\}$.

Let $Y_{b}\subseteq\CA_{b}\times\{t_{\infty}\}$ be the proper closed subset consisting of all the points $(x,t_{\infty})\in\CA_{b}\times\{t_{\infty}\}$ such that $((x+\CG_{b})\times\{t_{\infty}\})\cap W_{b}$ has codimension at most 1 in $(x+\CG_{b})\times\{t_{\infty}\}$.

\begin{proposition}\label{key-abc}
There is a constant $c>0$ depending only on $\mu_{B}$ satisfying the following property: For a point $b\in B(\CC)\backslash (S(A)\cup S_{\rm nr})$ and a Euclidean open neighbourhood $\CU$ of $Y_{b}\cup\CD$, denote $B_{n,b}\colonequals (s_{n},t_{n})^{-1}(\CU)$. Then for any $\varepsilon>0$, there is an open disc $\BD_{b}\subseteq B(\CC)$ (depending on $\CU$) containing $b$ such that for any open disc $\BD'_{b}$ contained in $\BD_{b}$, any $\varepsilon'>0$ and any $n$ large enough (depending on $\BD'_{b}$ and $\varepsilon'$), there exists a subset $\Sigma_{n,b}\subseteq \BD'_{b}$ of $\mu_{B}$-measure less than $\frac{c}{\varepsilon}\mu_{B}(B_{n,b}\cap\BD'_{b})+\varepsilon'$ with
\begin{align*}
\sum\limits_{t\in\BD'_{b}}\min\{1,{\rm ord}_{t}(s_{n},t_{n})^{*}\CD\}\geq (1-\varepsilon)\int_{\BD'_{b}\backslash\Sigma_{n,b}}(s_{n},t_{n})^{*}\omega.
\end{align*}
\end{proposition}

\begin{proof}
Take a finite cover of $B(\CC)$ by open discs. For any disc $\BD$ in this cover, fix a standard Lebesgue measure $\mu_{0,\BD}$ on $\BD$. Then there is a constant $c_{\BD}>1$ such that $c_{\BD}^{-1}\mu_{0,\BD}(\Omega)\leq\mu_{B}(\Omega)\leq c_{\BD}\mu_{0,\BD}(\Omega)$ for any Lebesgue measurable subset $\Omega\subseteq \BD$. Take $c=2\max_{\BD}\{c_{\BD}^{2}\}>0$, where $\BD$ is taken over all open discs in this finite cover. We show that $c$ satisfies the conditions in the proposition.

Fix a point $b\in B(\CC)\backslash (S(A)\cup S_{\rm nr})$, a Euclidean open neighbourhood $\CU$ of $Y_{b}\cup\CD$ in $\CA\times_{\CC}T$ and a constant $0<\varepsilon<1$. We need to find the open disc $\BD_{b}$ in the proposition. If $\CD$ contains the whole fibre $\CA_{b}\times\{t_{\infty}\}$, then $\CU$ contains the fibre of the map $\CA\times_{\CC}T\to B\times_{\CC}T$ over a Euclidean open neighbourhood of $(b,t_{\infty})$. In particular, there is an open disc $\BD_{b}$ containing $b$ such that $\CA_{\BD_{b}}\times\{t_{n}\}\subseteq \CU$ for $n$ large enough. In this case $B_{n,b}\cap\BD'_{b}=\BD'_{b}$ for any disc $\BD'_{b}\subseteq\BD_{b}$ and $n$ large enough, and thus the inequality is trivial by taking $\Sigma_{n,b}=\BD'_{b}$ (note that $c>1>\varepsilon$). Now we assume $\CD$ does not contain the whole fibre $\CA_{b}\times\{t_{\infty}\}$.

By the definition of $Y_{b}$, for each point $(x,t_{\infty})\in\CA_{b}\times\{t_{\infty}\}\backslash Y_{b}$, the intersection $((x+\CG_{b})\times\{t_{\infty}\})\cap W_{b}$ has codimension at least 2 in $(x+\CG_{b})\times\{t_{\infty}\}$. Applying Proposition \ref{conclusion} to the abelian subvariety $\CA_{b}\times\{t_{\infty}\}\subseteq\CA\times_{\CC}T$, the vector $v=\delta(s_{\infty},v_{b})$ and the Euclidean closed subset $Z=\CA_{b}\times\{t_{\infty}\}\backslash \CU_{b,t_{\infty}}$, we get constants $R>0, \delta>0$ and a Euclidean closed subset $W'_{b}$ of $\CD$ which is disjoint from $W_{b}$ satisfying the following condition: For any point $(x,t_{\infty})\in Z$ and any holomorphic map $\phi\colon \BD_{R+1}\to \CA\times_{\CC}T$, if $d(\phi(z), (\phi_{(x,v)}(z),t_{\infty}))<\delta$ for all $z\in\BD_{R+1}$, then
\begin{align}\label{abc-2}
\sum\limits_{z\in\BD_{R}\cap\phi^{-1}(W_{b}')}{\rm ord}_{z}\phi^{*}\CD\geq(1-\frac{\varepsilon}{2})\int_{\BD_{R}}\phi^{*}\omega.
\end{align}

Recall that $\CG_{b}\times\{t_{\infty}\}$ is the Zariski closure of the linear leaf $\CF(x,\delta(s_{\infty},v_{b}))\times\{t_{\infty}\}$ in $\CA_{b}\times\{t_{\infty}\}$. Let $Z'\colonequals Z+\CG_{b}\times\{t_{\infty}\}$, which is also a Euclidean closed subset of $\CA_{b}\times\{t_{\infty}\}$.

We claim that after replacing $\delta$ by a smaller positive number, there is a Euclidean open neighbourhood $U_{b}$ of $b$ such that for $n$ large enough and $b'\in (s_{n},t_{n})^{-1}(W'_{b})\cap U_{b}$ with $d((s_{n}(b'),t_{n}),Z')<\delta$, we have ${\rm ord}_{b'}(s_{n},t_{n})^{*}\CD=1$. In fact, if this is not true, there is a sequence $b_{k}\in B(\CC)$ tending to $b$ and a sequence $n_{k}\to\infty$ such that $(s_{n_{k}},t_{n_{k}})$ is tangent to $\CD$ at $b_{k}$, $(s_{n_{k}}(b_{k}),t_{n_{k}})\in W'_{b}$ and $d((s_{n_{k}}(b_{k}),t_{n_{k}}),Z')<\frac{1}{k}$. Let $x\in\CA_{b}$ be a limit point of the sequence $\{s_{n_{k}}(b_{k})\}_{k\geq1}$. Then $(x,t_{\infty})$ is a limit point of $\{(s_{n_{k}}(b_{k}),t_{n_{k}})\}_{k\geq1}$, and thus belongs to $W'_{b}\cap Z'$ by our assumption. In particular, as $W_{b}$ is disjoint from $W'_{b}$, we have $(x,t_{\infty})\in\CD_{t_{\infty}}^{\circ}$. Choose any open disc $\BD\subseteq B(\CC)$ with centre $b$.  Recall that $\ell_{n_{k}}^{-1}s_{n_{k}}\to s_{\infty}$ in $V(A,K)$ as $k\to\infty$. Choose $r_{n_{k}}\to\infty$ such that $b_{k}+\ell_{n_{k}}^{-1}z\in\BD$ for all $k$ large enough and all $z\in\BD_{r_{n_{k}}}$. Consider the re-parametrization
\begin{align*}
\varphi_{k}\colon \BD_{r_{n_{k}}}&\to \CA\times_{\CC}T \\
                      z&\mapsto (s_{n_{k}}(b_{k}+\ell_{n_{k}}^{-1}z),t_{n_{k}})
\end{align*}
By Lemma \ref{linear}, $\varphi_{k}$ converges (uniformly on any compact subset of $\CC$) to the linear entire curve $\phi_{\infty}\colonequals\phi_{(x, \delta(s_{\infty},v_{b}))}\times\{t_{\infty}\}$. In particular, we have
$$
({\rm d}\phi_{k})_{0}(v_{\rm st})\to ({\rm d}\phi_{\infty})_{0}(v_{\rm st})
$$
inside the tangent bundle of $\CA\times_{\CC}T$, where $({\rm d}\phi_{k})_{0}$ is the tangent map of $\phi_{k}$ at $0$, and $v_{\rm st}=\frac{{\rm d}}{{\rm d}z}\in T_{0}\CC$ is the standard tangent vector. By our assumption, the tangent vector $({\rm d}\phi_{k})_{0}(v_{\rm st})$ belongs to $T_{(s_{n_{k}}(b_{k}),t_{n_{k}})}\CD$. As the image of $\phi_{k}$ is contained in $\CA\times\{t_{n_{k}}\}$, this vector also belongs to $T_{(s_{n_{k}}(b_{k}),t_{n_{k}})}\CD_{t_{n_{k}}}$. Since $(x,t_{\infty})\in\CD_{t_{\infty}}^{\circ}$, by our choice of $\CD_{t_{\infty}}^{\circ}$, the vector $({\rm d}\phi_{\infty})_{0}(v_{\rm st})$ belongs to $T_{(x,t_{\infty})}(\CD_{t_{\infty}})_{\rm red}$. As the image of $\phi_{\infty}$ is contained in $\CA_{b}\times\{t_{\infty}\}$, this vector also belongs to the tangent space of $((\CD_{t_{\infty}})_{\rm red})_{b}$ at $(x,t_{\infty})$. Since $b\notin S_{\rm nr}$, the last scheme is reduced and equals $(\CD_{b,t_{\infty}})_{\rm red}$.
In conclusion, the linear leaf $\CF(x,\delta(s_{\infty},v_{b}))\times\{t_{\infty}\}$ (which is not contained in $\CD_{b,t_{\infty}}$ since $(x,t_{\infty})\in Z'$) is tangent to $(\CD_{b,t_{\infty}})_{\rm red}$ at $(x,t_{\infty})$, which contradicts the fact that $(x,t_{\infty})\notin W_{b}$. This proves the claim.

Take any open unit disc $\BD\subseteq U_{b}$ with centre $b$ which is contained in some open disc of the finite cover chosen at the beginning of the proof. Let $z$ be the standard coordinate of $\BD$. By Lemma \ref{linear}, there are constants $0<\delta'<1$ and $n_{0}>0$ such that for each $b'\in\BD_{\delta'}$, each $z\in\BD_{R+1}$ and each $x\in\CA_{b}=\CA_{0}$, if $n>n_{0}$ and $d_{\CA}(s_{n}(b'),x)<\delta'$, then $b'+\ell_{n}^{-1}z\in\BD$ and
\begin{align}\label{abc-3}
d_{\CA}(s_{n}(b'+\ell_{n}^{-1}z),\phi_{(x,v)}(z))<\delta/2
\end{align}
Enlarging $n_{0}$ and replacing $\delta'$ by a smaller positive number if necessary, we can assume that for any $n>n_{0}$ and any point $(x',t_{n})\in(\CA_{\BD_{\delta'}}\times\{t_{n}\})\backslash \CU$, there is a point $(x,t_{\infty})\in Z$ with $d((x',t_{n}),(x,t_{\infty}))<\delta'/2$. In particular, we have $d_{T}(t_{n},t_{\infty})<\delta'/2<\delta/2$ for all $n>n_{0}$.

Now take $\BD_{b}=\BD_{\delta'}$. We prove that it satisfies the conditions in the Proposition.

In fact, for any open disc $\BD'_{b}\subseteq \BD_{b}$, assume $\BD'_{b}=\BD(b',r')$. For $b''\in\BD'_{b}$ and $n>n_{0}$, if $b''\notin B_{n,b}$ and $\BD(b'',\ell_{n}^{-1}R)\subseteq\BD'_{b}$, then there is a point $(x,t_{\infty})\in Z$ with $d((s_{n}(b''),t_{n}),(x,t_{\infty}))<\delta'/2$. By (\ref{abc-3}), we have
\begin{align*}
d((s_{n}(b''+\ell_{n}^{-1}z),t_{n}),(\phi_{(x,v)}(z),t_{\infty}))<\delta/2+\delta/2=\delta
\end{align*}
for all $z\in\BD_{R+1}$ and $n>n_{0}$. Note that $(\phi_{(x,v)}(z),t_{\infty})=(x+\phi_{(0,v)}(z),t_{\infty})\in Z+\CG_{b}\times\{t_{\infty}\}=Z'$, and thus $d((s_{n}(b''+\ell_{n}^{-1}z),t_{n}),Z')<\delta$ for all $z\in\BD_{R+1}$ and $n>n_{0}$. Now applying (\ref{abc-2}) to $\phi(z)=(s_{n}(b''+\ell_{n}^{-1}z),t_{n})$, we get
\begin{align}\label{abc-5}
\sum\limits_{z\in\BD_{R}\cap\phi^{-1}(W'_{b})}{\rm ord}_{z}\phi^{*}\CD\geq (1-\frac{\varepsilon}{2})\int_{\BD_{R}}\phi^{*}\omega= (1-\frac{\varepsilon}{2})\int_{\BD(b'',\ell_{n}^{-1}R)}(s_{n},t_{n})^{*}\omega
\end{align}
for all $n>n_{0}$ and $b''\in\BD'_{b}\backslash B_{n,b}$ with $\BD(b'',\ell_{n}^{-1}R)\subseteq\BD'_{b}$. Note that the condition $z\in\phi^{-1}(W'_{b})$ is equivalent to $b''+\ell_{n}^{-1}z\in (s_{n},t_{n})^{-1}(W'_{b})$. Since $b''+\ell_{n}^{-1}z\in \BD'_{b}\subseteq U_{b}$, the claim above then implies that ${\rm ord}_{z}\phi^{*}\CD={\rm ord}_{(b''+\ell_{n}^{-1}z)}(s_{n},t_{n})^{*}\CD=1$ for $n$ large enough. In particular, we have

\begin{align*}
\sum\limits_{z\in\BD(b'',\ell_{n}^{-1}R)}\min\{1,{\rm ord}_{z}(s_{n},t_{n})^{*}\CD\}&\geq\sum\limits_{z\in\BD(b'',\ell_{n}^{-1}R)\cap(s_{n},t_{n})^{-1}(W'_{b})}\min\{1,{\rm ord}_{z}(s_{n},t_{n})^{*}\CD\}\\
&=\sum\limits_{z\in\BD(b'',\ell_{n}^{-1}R)\cap(s_{n},t_{n})^{-1}(W'_{b})}{\rm ord}_{z}(s_{n},t_{n})^{*}\CD\\
&=\sum\limits_{z\in\BD_{R}\cap\phi^{-1}(W'_{b})}{\rm ord}_{z}\phi^{*}\CD\\
&\geq(1-\frac{\varepsilon}{2})\int_{\BD(b'',\ell_{n}^{-1}R)}(s_{n},t_{n})^{*}\omega.
\end{align*}
for all $n$ large enough and all $b''\in\BD'_{b}\backslash B_{n,b}$ with $\BD(b'',\ell_{n}^{-1}R)\subseteq\BD'_{b}$.

The proposition then follows from Lemma \ref{refine--}; the deduction is completely similar to the final argument in the proof of Proposition \ref{key-} and we omit it here.
\end{proof}
To derive Proposition \ref{V}(2) from Proposition \ref{key-abc}, we need a refined version of Lemma \ref{juli} as follows.
\begin{lemma}\label{juli-}
There exists a subset $S_{0}\subseteq B(\CC)$ with $\mu_{B}(S_{0})=0$ and a sequence of probability measures $\{\mu_{n}\}_{n\geq1}$ on $B(\CC)$, all depending only on the sequence $\{(s_{n},t_{n})\}_{n\geq1}$, satisfying the following property: For a closed subscheme $\CZ$ of $\CA\times_{\CC}T$ which does not contain the fibre $\CA\times\{t_{\infty}\}$, let $S_{\CZ}\subseteq B(\CC)$ be the set of all $b\in B(\CC)$ such that $\CZ\supseteq\CA_{b}\times\{t_{\infty}\}$. Then for any $\delta>0$ and any point $b\in B(\CC)\backslash(S_{0}\cup S_{\CZ})$, there is an open disc $\BD$ containing $b$ and a Euclidean open neighbourhood $\CU$ of $\CZ$ such that for any open disc $\BD^{(1)}\subseteq \BD$ and any open disc $\BD^{(2)}\subseteq \BD$ containing $\ol{\BD^{(1)}}$, the inequality $\mu_{B}(\BD^{(1)}\cap(s_{n},t_{n})^{-1}(\CU))\leq\delta\mu_{n}(\BD^{(2)})$ holds for $n$ large enough.
\end{lemma}
\begin{proof}
Similarly to Lemma \ref{juli}, we do induction on the number of simple isogeny factors of $A$. When $A$ is trivial, then $\CA=B$ and $\CA\times_{\CC}T=B\times_{\CC}T$. We prove that we can take $S_{0}=\emptyset$ and $\mu_{n}$ to be any probability measure on $B(\CC)$. In fact, fix a closed subscheme $\CZ\subseteq B\times_{\CC}T$ and a point $b\in B(\CC)\backslash S_{\CZ}$. By our assumption, the subset $\{b\}\times \{t_{\infty}\}\subseteq B(\CC)\times T(\CC)$ is disjoint from $\CZ$. Hence there is an open disc $\BD\subseteq B(\CC)$ containing $b$ and a Euclidean open neighbourhood $\CU$ of $\CZ$ such that the subset $\BD\times \{t_{n}\}\subseteq B(\CC)\times T(\CC)$ is disjoint from $\CU$ for $n$ large enough. In this case $\BD\cap(s_{n},t_{n})^{-1}(\CU)$ is empty and the inequality is trivial.

Now we consider the general case. Let $G=G(s_{\infty})$ be the unique minimal abelian subvariety of $A$ such that $s_{\infty}\in V(G,K)$. Let $\CG$ be the Zariski closure of $G$ in $\CA$. By Proposition \ref{closure}, there is a subset $S(A)\subseteq B(\CC)$ with $\mu_{B}(S(A))=0$ such that for all $b\in B(\CC)\backslash S(A)$, the Zariski closure of the linear leaf $\CF(0,\delta(s_{\infty},v_{b}))$ in $\CA_{b}$ equals  $\CG_{b}$. Here $v_{b}\in T_{b}B$ is a nonzero vector.

Let $\CA/\CG$ be an integral model of $A/G$ over $B$.  Then the quotient map $A\to A/G$ extends to a map $p\colon\CA_{B_{0}}\to(\CA/\CG)_{B_{0}}$ over some Zariski dense open subset $B_{0}$ of $B$.

Let $\ol{s}_{n}$ be the image of $s_{n}$ in $A/G$. Applying the induction hypothesis to the sequence\\  $\{(\ol{s}_{n},t_{n})\}_{n\geq1}\subseteq (A/G)(K)\times T(\CC)$, we get a subset $S'_{0}\subseteq B(\CC)$ of $\mu_{B}$-measure zero and a sequence of probability measures $\{\mu'_{n}\}_{n\geq1}$ on $B(\CC)$ satisfying the conditions in the lemma. Now let $S_{0}=S'_{0}\cup S(A)\cup(B\backslash B_{0})(\CC)$. Fix an ample symmetric line bundle $\CM$ on $\CA\times_{\CC}T$. Let $\omega$ be the Betti form associated to $\CM$ on $\CA_{B_{0}}\times_{\CC}T$ and let $\mu_{\omega,n}$ be the normalized probability measure associated to $\omega$ and $(s_{n},t_{n})$, i.e.
$$
\mu_{\omega,n}(S)\colonequals\frac{\int_{S\cap B_{0}(\CC)}(s_{n},t_{n})^{*}\omega}{\int_{B_{0}(\CC)}(s_{n},t_{n})^{*}\omega}
$$
for any measurable subset $S$ of $B(\CC)$. Finally, let $\mu_{n}=\frac{1}{2}(\mu'_{n}+\mu_{\omega,n})$, which is a probability measure on $B(\CC)$ for all $n\geq1$. We prove that $S_{0}$ and $\{\mu_{n}\}_{n\geq1}$ satisfy the conditions.

Fix a closed subscheme $\CZ$ of $\CA\times_{\CC}T$ which does not contain the fibre $\CA\times\{t_{\infty}\}$, a constant $\delta>0$ and a point $b\in B(\CC)\backslash(S_{0}\cup S_{\CZ})$. By definition, the fibre $\CA_{b}\times\{t_{\infty}\}$ is not contained in $\CZ$. Choose a section of some multiple $m\CM$ of $\CM$ such that the support of its corresponding divisor $\CE$ contains $\CZ$ but does not contain the fibre $\CA_{b}\times\{t_{\infty}\}$. Identify $\CE$ with its support. let $\CZ'_{0}$ be the set of all $(x,t)\in(\CA/\CG)_{B_{0}}\times_{\CC}T$ such that $\CE$ contains the whole fibre $(p\times{\rm id}_{T})^{-1}((x,t))$. Let $\CZ'$ be the Zariski closure of $\CZ'_{0}$ in $(\CA/\CG)\times_{\CC}T$. As $\CE$ does not contain the fibre $\CA_{b}\times\{t_{\infty}\}\subseteq\CA\times\{t_{\infty}\}$, we deduce that $\CZ'_{0}$ does not contain the fibre $(\CA/\CG)_{b}\times\{t_{\infty}\}\subseteq(\CA/\CG)_{B_{0}}\times\{t_{\infty}\}$, hence $\CZ'$ does not contain the fibre $(\CA/\CG)\times\{t_{\infty}\}$. Moreover, let $S_{\CZ'}\subseteq B(\CC)$ be the set of all points $b'\in B(\CC)$ such that $\CZ'\supseteq(\CA/\CG)_{b'}\times\{t_{\infty}\}$. Then we have $b\notin S'_{0}\cup S_{\CZ'}$.

Now Applying our induction hypothesis to $\CZ'\subseteq(\CA/\CG)\times_{\CC}T$, $\frac{\delta}{2}$ and $b$, we see that there is an open disc $\BD'\subseteq B(\CC)$ containing $b$ and a Euclidean open neighbourhood $\CU'$ of $\CZ'$ such that for any open disc $\BD^{(1)}\subseteq \BD'$ and any open disc $\BD^{(2)}$ containing $\ol{\BD^{(1)}}$, the inequality $\mu_{B}(\BD^{(1)}\cap(\ol{s}_{n},t_{n})^{-1}(\CU'))\leq\frac{\delta}{2}\mu'_{n}(\BD^{(2)})$ for $n$ large enough.

Now let $\CW=(\CA\times_{\CC}T)\backslash p^{-1}(\CU')$. Fix the open discs $\BD^{(1)},\BD^{(2)}$ above. Note that $\BD^{(1)}\cap(s_{n},t_{n})^{-1}(p^{-1}(\CU'))=\BD^{(1)}\cap(\ol{s}_{n},t_{n})^{-1}(\CU')$, which has $\mu_{B}$-measure less than $\frac{\delta}{2}\mu'_{n}(\BD^{(2)})$ for $n$ large enough. Therefore we only need to show that there is an open disc $\BD\subseteq \BD'$ containing $b$ and a Euclidean open neighbourhood $\CU$ of $\CE$ such that for any $\BD^{(1)},\BD^{(2)}$ above, if $\BD^{(2)}$ is contained in $\BD$, then the subset $\BD^{(1)}\cap(s_{n},t_{n})^{-1}(\CW\cap \CU)$ has $\mu_{B}$-measure less than $\frac{\delta}{2}\mu_{\omega,n}(\BD^{(2)})$ for $n$ large enough.

As $\ell_{n}\sim \hat{h}_{\CE}((s_{n},t_{n}))^{1/2}$, there is a constant $C>0$ such that $\ell_{n}^{-2}\cdot\hat{h}_{\CE}((s_{n},t_{n}))<C$ for all $n\geq1$.

After enlarging $C$, we may assume that $\mu_{B}(\BB(b',r))\leq Cr^{2}$ for all $b'\in B(\CC)$ and $r>0$. Here $\BB(b',r)$ is the open ball $\{b''\in B(\CC)\big|\,d_{B}(b',b'')<r\}$.

Now consider the point $b\in B(\CC)\backslash (S_{0}\cup S_{\CZ})$ which we have fixed. Recall that $\CE$ does not contain the fibre $\CA_{b}\times\{t_{\infty}\}$. By the claim in the proof of Lemma \ref{juli}, there is an open disc $\BD\subseteq B(\CC)$ containing $b$, a constant $\varepsilon_{b}>0$ and an integer $n_{b}\geq1$ such that for all $n>n_{b}$ and all $b'\in \BD$, if $(s_{n}(b'),t_{n})\in \CW$ and $d((s_{n}(b'),t_{n}),\CE)<\varepsilon_{b}$, then there is a point $b''\in(s_{n},t_{n})^{-1}(\CE)$ with
$$
d_{B}(b',b'')<(2C\ell_{n})^{-1}\delta^{1/2}.
$$
Shrinking $\BD$ if necessary, we assume that $\BD$ is contained in $\BD'$. Take $\CU$ to be $\BB(\CE,\varepsilon_{b})\colonequals\{(x,t)\in\CA\times_{\CC}T\big|\,d((x,t),\CE)<\varepsilon_{b}\}$. We prove that $\BD$ and $\CU$ above satisfy the conditions in the lemma.

Fix two open discs $\BD^{(1)},\BD^{(2)}$ contained in $\BD$ such that $\ol{\BD^{(1)}}\subseteq\BD^{(2)}$. By the claim above, for $n$ large enough, the set $\BD^{(1)}\cap(s_{n},t_{n})^{-1}(\CW\cap \CU)$ is contained in the union of the open balls
$$
\BB(b'',(2C\ell_{n})^{-1}\delta^{1/2}),\, b''\in (s_{n},t_{n})^{-1}(\CE)\cap\BD^{(2)}.
$$
By our choice of $C$, each ball has $\mu_{B}$-measure at most $C_{1}((2C\ell_{n})^{-1}\delta^{1/2})^{2}=\frac{\ell_{n}^{-2}\delta}{4C}$. Moreover, by Proposition \ref{V}(1), the inequality
$$
\sharp (\BD^{(2)}\cap(s_{n},t_{n})^{-1}(\CE))\leq\sum\limits_{t\in\BD^{(2)}}{\rm ord}_{t}(s_{n},t_{n})^{*}\CE\leq 2\int_{\BD^{(2)}}(s_{n},t_{n})^{*}(m\omega)
$$
holds for $n$ large enough. Thus we have
\begin{align*}
\mu_{B}(\BD^{(1)}\cap(s_{n},t_{n})^{-1}(\CW\cap \CU))&\leq 2\int_{\BD^{(2)}}(s_{n},t_{n})^{*}(m\omega)\cdot \frac{\ell_{n}^{-2}\delta}{4C}\\
&\leq\frac{\delta\ell_{n}^{-2}h_{m\omega}((s_{n},t_{n}))}{2C}\cdot\mu_{\omega,n}(\BD^{(2)})\\
&<\frac{\delta}{2}\mu_{\omega,n}(\BD^{(2)}).
\end{align*}
for $n$ large enough. This ends the proof.
\end{proof}

\begin{proof}[Proof of Proposition \ref{V}(2)]
We only need to prove the left-hand inequality in Proposition \ref{V}(2), as the right-hand is weaker than that of Proposition \ref{V}(1). Fix a subset $U\subseteq B(\CC)$ such that either $\mu_{B}(\partial U)=0$ or $U$ is (Euclidean) open. If $\mu_{B}(U)=0$, then the Xie--Yuan height $\hat{h}_{D^{\rm sym},U}((s_{n},t_{n}))$ is always zero and the truncated counting function is nonnegative. Hence the inequality is trivial. Now assume that $\mu_{B}(U)>0$. By Proposition \ref{feituihua}, there is a constant $\varepsilon_{U}>0$ such that $h_{\omega,U}\geq \varepsilon_{U}\hat{h}_{D^{\rm sym}}$ on $A(K)\times T(\CC)$. By Proposition \ref{jiben-}, there is $\delta>0$ such that for any subset $S\subseteq B(\CC)$ with $\mu_{B}(S)<\delta$, we have $h_{\omega,S}\leq \frac{\varepsilon\cdot\varepsilon_{U}}{2} \hat{h}_{D^{\rm sym}}$ on $A(K)\times T(\CC)$. Let $c>0$ be the constant in Proposition \ref{key-abc}.

Apply Lemma \ref{juli-} to the sequence $\{(s_{n},t_{n})\}_{n\geq1}$, we get a subset $S_{0}\subseteq B(\CC)$ with $\mu_{B}(S_{0})=0$ and a sequence of probability measures $\{\mu_{n}\}_{n\geq1}$ on $B(\CC)$ satisfying the properties of the lemma. Recall that $S_{\CD}\subseteq B(\CC)$ is the set of all $b\in B(\CC)$ such that $\CD$ contains $\CA_{b}\times\{t_{\infty}\}$. For each $b\in B(\CC)\backslash (S_{0}\cup S_{\CD})$, consider the closed subscheme $\CZ(b)\colonequals Y_{b}\cup\CD$ of $\CA\times_{\CC}T$. Since $Y_{b}$ is properly contained in $\CA_{b}\times\{t_{\infty}\}$, we also have $b\in B(\CC)\backslash (S_{0}\cup S_{\CZ(b)})$. Therefore by Lemma \ref{juli-}, there is an open disc $\BD(b)$ containing $b$ and a Euclidean open neighbourhood $\CU(b)$ of $\CZ(b)$ in $\CA\times_{\CC}T$ such that for any open disc $\BD'\subseteq\BD(b)$ and any open disc $\BD''\subseteq B(\CC)$ containing $\ol{\BD}'$, the inequality
$$
\mu_{B}(\BD'\cap (s_{n},t_{n})^{-1}(\CU(b)))\leq\frac{\varepsilon\delta}{8c}\mu_{n}(\BD'')
$$
holds for $n$ large enough. Denote $B_{n,b}=(s_{n},t_{n})^{-1}(\CU(b))\subseteq B(\CC)$.

For each $b\in B(\CC)\backslash (S_{0}\cup S_{\CD})$, take an open disc $\BD_{b}\subseteq B(\CC)$ satisfying Proposition \ref{key-abc} (with $\varepsilon$ replaced by $\frac{\varepsilon}{2}$). Shrinking $\BD_{b}$ if necessary, we may assume that $\BD_{b}$ is contained in $\BD(b)$. Let $U^{\circ}$ be the interior of $U$ (in the Euclidean topology). Then since $\mu_{B}(U^{\circ})=\mu_{B}(U\backslash\partial U)=\mu_{B}(U)$, we can choose finitely many points $b_{1},\dots,b_{k}\in U^{\circ}\backslash S_{0}$ such that $U^{\circ}\backslash(\cup_{i=1}^{k}\BD_{b_{i}})$ has $\mu_{B}$-measure less than $\frac{\delta}{2}$. Now choose finitely many pairwise disjoint open discs $D_{j}\, (j\in J)$ in $U^{\circ}$, each disc is contained in some $\BD_{b_{i}}$, and $U^{\circ}\backslash(\cup_{j\in J}D_{j})$ has $\mu_{B}$-measure less than $\frac{\delta}{2}$. We also choose open discs $D'_{j}$ containing $D_{j}$ such that $D'_{j}\,(j\in J)$ are also pairwise disjoint.

Choose $\varepsilon_{j}>0$ for each $j\in J$ such that $\sum_{j\in J}\varepsilon_{j}<\frac{\delta}{4}$. By our choice in the previous paragraph, for each $j\in J$, there is some $1\leq i_{j}\leq k$ with $D_{j}\subseteq \BD_{b_{i_{j}}}$. Thus for $n$ large enough, Proposition \ref{key-abc} implies that there is a subset $\Sigma_{n,j}\subseteq D_{j}$ of measure less than $\frac{c}{\varepsilon/2}\mu_{B}(B_{n,b_{i_{j}}}\cap D_{j})+\varepsilon_{j}$ such that
\begin{align*}
\sum\limits_{t\in D_{j}}\min\{1,{\rm ord}_{t}(s_{n},t_{n})^{*}\CD\}\geq (1-\frac{\varepsilon}{2})\int_{D_{j}\backslash\Sigma_{n,j}}(s_{n},t_{n})^{*}\omega.
\end{align*}
Denote $V_{0}=U^{\circ}\backslash(\cup_{j\in J}D_{j})$. Let $\Sigma_{n}=V_{0}\cup(\cup_{j\in J}\Sigma_{n,j})$. Then we have
\begin{align}\label{final---0}
\sum\limits_{t\in U}\min\{1,{\rm ord}_{t}(s_{n},t_{n})^{*}\CD\}\geq\sum\limits_{t\in U^{\circ}\backslash V_{0}}\min\{1,{\rm ord}_{t}(s_{n},t_{n})^{*}\CD\}\geq (1-\frac{\varepsilon}{2})\int_{U^{\circ}\backslash\Sigma_{n}}(s_{n},t_{n})^{*}\omega.
\end{align}
Moreover, we have
\begin{align*}
\mu_{B}(\Sigma_{n})\leq\mu_{B}(V_{0})+\frac{c}{\varepsilon/2}\sum\limits_{j\in J}\mu_{B}(B_{n,b_{i_{j}}}\cap D_{j})+\sum\limits_{j\in J}\varepsilon_{j}&\leq \frac{\delta}{2}+\frac{c}{\varepsilon/2}\sum\limits_{j\in J}\frac{\varepsilon\delta}{8c}\mu_{n}(D'_{j})+\frac{\delta}{4}\\
&\leq \frac{\delta}{2}+\frac{c}{\varepsilon/2}\cdot\frac{\varepsilon\delta}{8c}+\frac{\delta}{4}=\delta.
\end{align*}
Thus by our assumption, the right hand side of (\ref{final---0}) is
\begin{align*}
(1-\frac{\varepsilon}{2})(h_{\omega,U}((s_{n},t_{n}))-h_{\omega,\Sigma_{n}}((s_{n},t_{n})))&\geq (1-\frac{\varepsilon}{2})(h_{\omega,U}((s_{n},t_{n}))-\frac{\varepsilon\cdot\varepsilon_{U}}{2} \hat{h}_{D^{\rm sym}}((s_{n},t_{n})))\\
&\geq(1-\frac{\varepsilon}{2})(h_{\omega,U}((s_{n},t_{n}))-\frac{\varepsilon}{2} h_{\omega,U}((s_{n},t_{n})))\\
&\geq (1-\varepsilon)h_{\omega,U}((s_{n},t_{n})).
\end{align*}
This completes the proof of Proposition \ref{V}(2).
\end{proof}
\subsection{Proof of the main theorems}\label{4.5}
Now we deduce the main theorems (as well as their corollaries) from Proposition \ref{V}. We begin with two technical lemmas, which allow us to pass from Zariski density to genericity.

\begin{lemma}\label{generic}
Let $X$ be a variety over any field $K$. Let $\{x_{n}\}_{n\geq1}$ be a Zariski dense sequence of points in $X$. Then $\{x_{n}\}_{n\geq1}$ has a generic subsequence.
\end{lemma}
\begin{proof}
When $K$ is countable, we can use the same argument as in \cite[Lemma 4.1]{zbMATH01192411}. Namely, the set of all closed subvarieties of $X$ is countable. Arrange all proper closed subvarieties of $X$ as $Z_{1},Z_{2},\dots$. For each $k\geq1$, take $n_{k}\geq1$ such that $x_{n_{k}}$ does not belong to the union $\cup_{i=1}^{k}Z_{i}$. This can be done since $\{x_{n}\}_{n\geq1}$ is Zariski dense in $X$ and $Z_{i}\subsetneq X$. Then the subsequence  $\{x_{n_{k}}\}_{k\geq1}$ is generic in $X$.

Now we consider the general case. The variety $X$ and all the points $x_{n}(n\geq1)$ are defined over a countably generated field $k\subseteq K$. In particular, $k$ is countable. Let $X=(X_{0})\times_{k}K$, where $X_{0}$ is a $k$-variety. Let $x_{0,n}$ be the image of $x_{n}$ in $X_{0}$ (under the natural projection). By our assumption, $x_{0,n}\in X_{0}(k)$ for all $n\geq1$. Since $\{x_{n}\}_{n\geq1}$ is Zariski dense in $X$, $\{x_{0,n}\}_{n\geq1}$ is also Zariski dense in $X_{0}$. By the previous paragraph, $\{x_{0,n}\}_{n\geq1}$ has a generic subsequence $\{x_{0,n_{k}}\}_{k\geq1}$. We claim that the subsequence $\{x_{n_{k}}\}_{k\geq1}$ is generic in $X$. In fact, for any infinite subsequence $\{y_{n}\}_{n\geq1}$ of $\{x_{n_{k}}\}_{k\geq1}$, let $Y$ be the Zariski closure of $\cup_{n\geq1}\{y_{n}\}$ in $X$. Then $Y$ is also defined over $k$. Let $y_{0,n}$ be the image of $y_{n}$ in $X_{0}$. Then $\{y_{0,n}\}_{n\geq1}$ is an infinite subsequence of $\{x_{0,n_{k}}\}_{k\geq1}$, which is Zariski dense in $X_{0}$ by our assumption. This implies that the image of $Y$ in $X_{0}$ is exactly $X_{0}$. Thus we have $Y=(X_{0})\times_{k}K=X$ and $\{y_{n}\}_{n\geq1}$ is Zariski dense in $X$.
\end{proof}

\begin{lemma}\label{real}
Let $X$ be a variety over a field $K$. If there is a real valued function $h\colon X(K)\to\RR$ such that for any nonempty Zariski open subset $U$ of $X$, the restriction $h\big|_{U(K)}$ is unbounded from the above. Then there is a generic sequence $\{x_{n}\}_{n\geq1}\subseteq X(K)$ in $X$ with $h(x_{n})\to+\infty$ as $n\to\infty$.
\end{lemma}

\begin{proof}
For convenience, we call the function $h$ in the lemma ``height''. We use induction to prove the following claim:

\emph{For any nonempty (Zariski) open subset $U$ of $X$ and any integer $1\leq k\leq\dim X$, there is a sequence of points $\{y_{n}\}_{n\geq1}\subseteq U(K)$ such that $h(y_{n})\to+\infty$ as $n\to\infty$, and the Zariski closure of $\cup_{n\geq1}\{y_{n}\}$ in $X$ is of dimension at least $k$.}\medskip

In fact, by our hypothesis, for any nonempty open subset $U\subseteq X$ there is a sequence of points in $U(K)$ with height tending to infinity. This sequence satisfies the condition above for $k=1$. Now assume that the claim is true for $l\,(1\leq l\leq\dim X-1)$. Then for any nonempty open subset $U\subseteq X$, take a sequence $\{y_{n,0}\}_{n\geq1}\subseteq U(K)$ with height tending to infinity and $\dim\ol{\cup_{n\geq1}\{y_{n,0}\}}^{\rm zar}\geq l$. Set $Z_{0}=\ol{\cup_{n\geq1}\{y_{n,0}\}}^{\rm zar}$. After replacing $Z_{0}$ by one of its irreducible components of maximal dimension and replacing $\{y_{n,0}\}_{n\geq1}$ by a subsequence, we may assume that $Z_{0}$ is irreducible. Now set $U_{1}=U\backslash Z_{0}$.

If $U_{1}$ is empty, the sequence $\{y_{n,0}\}_{n\geq1}$ satisfies the conditions above for all $1\leq k\leq\dim X$. Now assume $U_{1}$ is a nonempty open subset of $X$. Take a sequence $\{y_{n,1}\}_{n\geq1}\subseteq U_{1}(K)$ with height tending to infinity and $\dim\ol{\cup_{n\geq1}\{y_{n,1}\}}^{\rm zar}\geq l$. Set $Z_{1}=\ol{\cup_{n\geq1}\{y_{n,1}\}}^{\rm zar}$. Similarly as above, we may assume that $Z_{1}$ is irreducible. Set $U_{2}=U_{1}\backslash Z_{1}$.

If $U_{2}$ is empty, the sequence $\{y_{n,1}\}_{n\geq1}$ satisfies the conditions above for all $1\leq k\leq\dim X$. Now if it is nonempty, then we can similarly construct a sequence $\{y_{n,2}\}_{n\geq1}\subseteq U_{2}(K)$ with height tending to infinity such that $Z_{2}=\ol{\cup_{n\geq1}\{y_{n,2}\}}^{\rm zar}$ is irreducible of dimension at least $l$. Set  $U_{3}=U_{2}\backslash Z_{2}$.

Repeat this process again and again, if $U_{i}$ is empty for some $i\geq1$, then the sequence $\{y_{n,i-1}\}_{n\geq1}$ satisfies the conditions above for all $1\leq k\leq\dim X$. Now assume that $U_{i}\neq\emptyset$ for all $i\geq1$. Then for all $i\geq1$ we have a sequence $\{y_{n,i}\}_{n\geq1}\subseteq U_{i}(K)$ with height tending to infinity such that $Z_{i}=\ol{\cup_{n\geq1}\{y_{n,i}\}}^{\rm zar}$ is irreducible of dimension at least $l$. Moreover, we have $U_{i+1}=U_{i}\backslash Z_{i}$.

Now for each $i\geq1$, by Lemma \ref{generic}, after taking a subsequence we may further assume that $\{y_{n,i}\}_{n\geq1}$ is generic in $Z_{i}$. By our assumption, we can further replace $\{y_{n,i}\}_{n\geq1}$ by a subsequence to make sure that $h(y_{n,i})>n+i$ for all $n\geq1$. This process does not change the fact that $\{y_{n,i}\}_{n\geq1}$ is generic in $Z_{i}$.

Consider the union $\cup_{i,n\geq1}\{y_{n,i}\}$, which is a countable subset of $U$. Now let $\{y_{n}\}\subseteq U(K)$ be a sequence such that $y_{n}\neq y_{m}$ if $n\neq m$, and the union $\cup_{n\geq1}\{y_{n}\}$ is exactly $\cup_{i,n\geq1}\{y_{n,i}\}$. Then $h(y_{n,i})>n+i$ implies that $h(y_{n})$ tends to infinity as $n\to\infty$. Moreover, by construction, the set $\ol{\cup_{n\geq1}\{y_{n}\}}^{\rm zar}$ contains the union $\cup_{i\geq1}Z_{i}$, which is an infinite union of distinct closed subsets of $X$ of dimension at least $l$, and which is therefore of dimension at least $l+1$. Hence the sequence $\{y_{n}\}$ satisfies the conditions in the claim with $k=l+1$. The claim follows by induction.

Now take $U=X,k=\dim X$ in the claim and apply Lemma \ref{generic} again; we then get the desired generic sequence $\{x_{n}\}_{n\geq1}$.
\end{proof}

\begin{proof}[Proof of Theorem \ref{main}]
Choose any ample effective Cartier divisor $D$ on $A$. We only need to prove that for any $\varepsilon>0$, there is a constant $c>0$ such that the inequality $h_{Y}\leq\varepsilon h_{D}+c$ holds for a general point on $A(K)$.

Assume to the contrary. Then there exists an $\varepsilon>0$ such that the function $h_{Y}-\varepsilon h_{D}$ is unbounded from the above on any $U(K)$ (where $U$ is a nonempty Zariski open subset of $A$). By Lemma \ref{real}, there is a generic sequence $\{s_{n}\}_{n\geq1}\subseteq A(K)$ in $A$ with $h_{Y}(s_{n})-\varepsilon h_{D}(s_{n})\to+\infty$ as $n\to\infty$. In particular, we have $h_{Y}(s_{n})\to+\infty$.

On the other hand, applying Proposition \ref{V}(3) to $T={\rm Spec}\,\CC$ and $D,Y,\{s_{n}\}_{n\geq1}$, we see that $h_{Y}(s_{n})\leq\varepsilon h_{D}(s_{n})$ for $n$ large enough, which contradicts our assumption.
\end{proof}
Corollary \ref{main'} is a direct consequence of Theorem \ref{main} and the Mordell--Lang theorem. We omit the proof here.

\begin{proof}[Proof of Theorem \ref{main1}]
First, we reduce to the case when $D$ is ample, which is the essential case. In fact, let $L=\CO_{A}(D)$ and let $G$ be the identity component of the subset $\{x\in A\big|\,t_{x}^{*}L\cong L\}$ (equipped with reduced induced subscheme structure). Then $G$ is an abelian subvariety of $A$. Let $A'=A/G$ and $q\colon A\to A/G$ be the quotient map. By definition, the restriction of $L$ to any fibre of $q$ is invariant under translation by points in $G$, and thus numerically trivial. Since $D$ is effective, the restriction of $D$ to general fibre of $q$ is still effective. Combining numerical triviality with effectivity, we see that $D$ restricts to zero on general fibre of $q$. In particular, the divisor $D$ is exactly the pullback of an effective divisor $D'$ on $A'$. Moreover, by our definition of $G$, the set $\{a'\in A'\big|\,t_{a'}^{*}D'=D'\}$ is finite. This implies that $D'$ is ample on $A'$. As $D$ is reduced, $D'$ is also reduced. By projection formula, we have $ N_{D,S}^{(1)}(s)=N_{D',S}^{(1)}(p(s))+O(1)$ and $h_{D}(s)=h_{D'}(p(s))+O(1)$ for all $s\in A(K)$. Therefore the theorem for $(A',D')$ implies the theorem for $(A,D)$. This reduces the problem to the ample case.

Now assume $D$ is ample. Let $D^{\rm sym}=\frac{1}{2}(D+[-1]^{*}D)$ be the symmetrization of $D$, which is also ample. Let $\hat{h}_{D},\hat{h}_{D^{\rm sym}}$ be the canonical height function associated to $D,D^{\rm sym}$ respectively. Since $D^{\rm sym}$ is ample and $D-D^{\rm sym}$ is anti-symmetric, we have $h_{D}=\hat{h}_{D}+O(1)=\hat{h}_{D^{\rm sym}}+O(1+\sqrt{\hat{h}_{D^{\rm sym}}})$. Hence we only need to prove the theorem with $h_{D}$ replaced by $\hat{h}_{D^{\rm sym}}$.

Assume to the contrary. Then there is an $\varepsilon>0$ such that the function $(1-\varepsilon)\hat{h}_{D^{\rm sym}}-N_{D,S}^{(1)}$ is unbounded from the above on any $U(K)$ (where $U$ is a nonempty Zariski open subset of $A$). By Lemma \ref{real}, there is a generic sequence $\{s_{n}\}_{n\geq1}\subseteq A(K)$ in $A$ with $(1-\varepsilon)\hat{h}_{D^{\rm sym}}(s_{n})-N_{D,S}^{(1)}(s_{n})\to+\infty$ as $n\to\infty$. In particular, we have $\hat{h}_{D^{\rm sym}}(s_{n})\to+\infty$.

On the other hand, as $S$ is finite, the boundary $\partial(B(\CC)\backslash S)=\partial S=S$ has zero measure. Hence applying Proposition \ref{V}(2) to $T={\rm Spec}\,\CC$ and $D,Y,\{s_{n}\}_{n\geq1},U=B(\CC)\backslash S$, we see that there are infinitely many $n\geq1$ such that $N_{D,S}^{(1)}(s_{n})\geq (1-\varepsilon)\hat{h}_{D^{\rm sym},U}(s_{n})=(1-\varepsilon)\hat{h}_{D^{\rm sym}}(s_{n})$, which contradicts our assumption.
\end{proof}

\begin{remark}\label{jiaqiang}
The above proof only uses the fact that $\partial S=S$ has measure 0, and thus Theorem \ref{main1} is also valid for all (Euclidean) closed subsets $S$ of $B(\CC)$ of measure 0.

More generally, the theorem also holds for any (Euclidean) closed subset $S\subseteq B(\CC)$ of sufficiently small $\mu_{B}$-measure (depending only on $\varepsilon$). Here we do not need $\mu_{B}(\partial S)=0$. In fact, by Proposition \ref{jiben-}, the condition $\mu_{B}(S)<<1$ implies that $\hat{h}_{D^{\rm sym},S}\leq \frac{\varepsilon}{2}\hat{h}_{D^{\rm sym}}(s)$ for all $s\in A(K)$. Fix such a subset $S\subseteq B(\CC)$ which is (Euclidean) closed. Similarly to the above proof, we only need to show the inequality on any generic sequence $\{s_{n}\}_{n\geq1}$ in $A(K)$ with $h_{D^{\rm sym}}(s_{n})\to\infty$. By Proposition \ref{V}(2), we have $N_{D,S}^{(1)}(s_{n})\geq (1-\frac{\varepsilon}{2})\hat{h}_{D^{\rm sym},U}(s_{n})$ for $n$ large enough. Here $U=B(\CC)\backslash S$. We conclude that the inequality $N_{D,S}^{(1)}(s_{n})\geq (1-\frac{\varepsilon}{2})^{2}\hat{h}_{D^{\rm sym}}(s_{n})\geq (1-\varepsilon)\hat{h}_{D^{\rm sym}}(s_{n})$ holds for $n$ large enough.

Similar remarks apply to Proposition \ref{buium} below; we will only prove the original version.
\end{remark}
\medskip

\begin{proof}[Proof of Proposition \ref{buium}]
Using the same argument as in the proof of Theorem \ref{main1}, we may assume that $D$ is ample.

Assume to the contrary. Then for each $n\geq1$, there is a point $x_{n}\in (A\backslash{\rm Supp}\,D)(K)$ such that $N^{(1)}_{D,S}(x_{n})< \frac{1}{n}h_{D}(x_{n})-n$. In particular, we have $h_{D}(x_{n})\to\infty$ as $n\to\infty$.

Let $X$ be the Zariski closure of $\cup_{n\geq1}\{x_{n}\}$ in $A$. In particular, the set $X(K)$ is Zariski dense in $X$. After replacing $X$ by one of its irreducible components and passing to a subsequence, we may assume that $X$ is irreducible. By Lemma \ref{generic}, we can further assume that $\{x_{n}\}_{n\geq1}$ is generic in $X$. Let
$$
A'\colonequals \{a\in A\big|\,a+X\subseteq X\}_{\rm red}^{0}.
$$
This is an abelian subvariety of $A$. Let $X/A'$ be the image of $X$ under the mapping $A\to A/A'$. By Ueno's fibration theorem, $X/A'$ is of general type.

Choose an isogeny $i\colon A'\times(A/A')\to A$ such that the composition of the inclusion $A'\hookrightarrow A'\times(A/A')\colon a'\mapsto (a',0)$ with $i$ is equal to the multiplication $a'\mapsto ma'$ by some integer $m$ on $A'$. Let $X'$ be an irreducible component of the inverse image $i^{-1}(X)\subseteq A'\times(A/A')$. Since $i$ is finite \'{e}tale, so is $X'\to X$. Moreover, we have $X'+(A'\times\{0\})\subseteq i^{-1}(X)$, and thus $X'+ (A'\times\{0\})=X'$ in $A'\times(A/A')$. Let $X'/A'$ be the image of $X'$ under the projection $A'\times(A/A')\to A/A'$, then we have $X'=A'\times(X'/A')$.

Now applying (geometric) Chevalley--Weil theorem (see \cite[Lemma 2.1]{arXiv:2507.21468}) to $i\big|_{X'}\colon X'\to X$, there is a finite field extension $K'/K$ such that $X(K)\subseteq i(X'(K'))$. Thus we can lift all points $x_{n}\in X(K)$ to $x'_{n}\in X'(K')$. Since $X(K)$ is Zariski dense in $X$, we have $X'(K')$ is Zariski dense in $X'$, and thus $(X'/A')(K')$ is also Zariski dense in $X'/A'$. As $X'/A'$ is a subvariety of general type of an abelian variety, the resolved (geometric) Bombieri--Lang conjecture (see \cite[Theorem 1.3]{arXiv:2511.17010}) for $X'/A'$ implies that the normalization of $X'/A'$ is isomorphic to a constant variety $T_{K}$, where $T$ is a projective variety over $\CC$. Moreover, \cite[Theorem 1.3]{arXiv:2511.17010} also concludes that the set $T_{K}(K')\backslash T(\CC)$ is not Zariski dense in $T_{K}(K')$.

Consider the composition $A'\times T_{K}\to A'\times (X'/A')=X'\to X$, which is generically finite. Since $\{x_{n}\}_{n\geq1}$ is generic in $X$, the sequence $\{x'_{n}\}_{n\geq1}$ is generic in $X'$. Hence for $n$ large enough, the point $x'_{n}$ has a unique lift $(s_{n},t_{n})$ on $A'(K')\times T_{K}(K')$, where $t_{n}$ belongs to $T(\CC)$. Moreover, $\{(s_{n},t_{n})\}_{n\geq1}$ is generic in $A'\times T_{K}$.

Since $D$ is an ample effective Cartier divisor on $A$, the pullback $(i^{*}D)\big|_{X'}$ on $X'=A'\times(X'/A')$ is also ample effective. Let $D'$ be the pullback of $(i^{*}D)\big|_{X'}$ to $A'\times T_{K}$, which is relatively ample over $T_{K}$. By the classical theory, the ampleness of an effective divisor on an abelian variety depends only on its support. Therefore the reduced divisor $D'_{\rm red}$ is also relatively ample over $T_{K}$. Let $B'\to B$ be the finite surjective map corresponding to the field extension $K'/K$ and let $S'$ be the inverse image of $S$ in $B'$, which is also finite. Applying Proposition \ref{V}(2) to $\{(s_{n},t_{n})\}_{n\geq1}\subseteq A'(K')\times T(\CC), D'_{\rm red},U=B'(\CC)\backslash S'$, we see that for any $\varepsilon>0$, the inequality
$$
N^{(1)}_{D',S'}((s_{n},t_{n}))=N^{(1)}_{D'_{\rm red},S'}((s_{n},t_{n}))\geq (1-\varepsilon)\hat{h}_{(D'_{\rm red})^{\rm sym}}((s_{n},t_{n}))
$$
holds for $n$ large enough.

By Proposition \ref{bijiao-}, we have $h_{D'_{\rm red}}((s,t))=\hat{h}_{(D'_{\rm red})^{\rm sym}}((s,t))+O\left(1+\sqrt{\hat{h}_{(D'_{\rm red})^{\rm sym}}((s,t))}\right)$ for all $(s,t)\in A'(K')\times T(\CC)$. Thus we also have
\begin{align*}
N^{(1)}_{D',S'}((s_{n},t_{n}))\geq (1-\varepsilon)h_{D'_{\rm red}}((s_{n},t_{n}))
\end{align*}
for $n$ large enough. Choose an integer $N\geq1$ such that $N\cdot D'_{\rm red}\geq D'$. Then the projection formula asserts that
\begin{align*}
N^{(1)}_{D,S}(x_{n})&= N^{(1)}_{D',S'}((s_{n},t_{n}))+O(1)\geq (1-\varepsilon)h_{D'_{\rm red}}((s_{n},t_{n}))+O(1)\\
&\geq (1-\varepsilon)(\frac{1}{N}h_{D'}((s_{n},t_{n}))+O(1))+O(1)\\
&= (1-\varepsilon)(\frac{1}{N}h_{D}(x_{n})+O(1))+O(1)
\end{align*}
for $n$ large enough. Here $O(1)$ are bounded functions. This contradicts our assumption at the beginning when $n$ is large enough.
\end{proof}

\subsection{Complements}\label{4.6}
In this final subsection, we use the main theorems to show that its exceptional set is geometric and hence independent of $K$, as predicted by Vojta's conjecture.

First note that by the Lefschetz principle, the field $\CC$ in the main theorems can be replaced by any field $k$ of characteristic 0. In particular, when $K$ is the function field of some curve $B$ over $k$ and $A$ is an abelian variety defined over $k$, we can deduce the following corollary.
\begin{corollary}\label{lefschetz}
Let $A$ be an abelian variety over a field $k$ of characteristic $0$. Let $L$ be an ample line bundle on $A$ and let $Y$ be a closed subscheme of $A$ of codimension at least $2$. Let $B$ be a smooth projective curve over $k$. Fix $\varepsilon>0$. Then the union in $A_{k(B)}$ of all maps $f\colon B\to A$ (viewed as points in $A(k(B))$) with $f(B)\nsubseteq Y$ and
\begin{align}\label{lef}
\sum\limits_{b\in |B|}{\rm ord}_{b}f^{*}Y>\varepsilon\deg_{B}f^{*}L
\end{align}
is not Zariski dense.
\end{corollary}
\begin{proof}
By Lefschetz principle and Theorem \ref{main}, the inequality
$$
\sum\limits_{b\in |B|}{\rm ord}_{b}f^{*}Y\leq\frac{\varepsilon}{2}\deg_{B}f^{*}L+O(1)
$$
holds for all $f\in A(k(B))$ outside a proper closed subset of $A_{k(B)}$. It suffices to cancel the $O(1)$ term. When $\deg_{B}f^{*}L$ is large enough, the right hand side above is less than $\varepsilon\deg_{B}f^{*}L$, and in this case $f$ doesn't satisfy (\ref{lef}). Now we consider all $f\colon B\to A$ with $\deg_{B}f^{*}L\leq C$ for some $C>0$. We claim that all these $f$ lie in a finite union of cosets of the form $f_{0}+A(k)\,(f_{0}\in A(k(B)))$. In fact, by Lang--N\'{e}ron theorem, the group $A(k(B))/A(k)$ is finitely generated, and the N\'{e}ron--Tate height function associated to an ample symmetric line bundle on $A_{k(B)}$ induces a positive definite quadratic form on $(A(k(B))/A(k))\otimes_{\BZ}\RR$. In particular, the elements in $A(k(B))/A(k)$ with bounded N\'{e}ron--Tate height is finite. Since $\deg_{B}f^{*}L\leq C$ implies that the height of $f\in A(k(B))$ with respect to $L$ (which is ample) is bounded, so is its image in $A(k(B))/A(k)$. This proves our claim above.

Now for any $f_{0}\colon B\to A$, the set of all points $a\in A$ such that $(f_{0}(B)+a)\cap Y\neq\emptyset$ is proper closed in $A$. In fact, this set is exactly the image of the map $B\times Y\to A$ sending $(b,y)$ to $y-f_{0}(b)$, which is closed of codimension at least $1$ in $A$. For all $a\in A(k)$ outside this proper closed subset, we have $(f_{0}(B)+a)$ disjoints from $Y$, and thus  $\sum_{b\in |B|}{\rm ord}_{b}(f_{0}+a)^{*}Y=0\leq \varepsilon\deg_{B}(f_{0}+a)^{*}L$. This means that the set of points $f\in f_{0}+A(k)$ satisfying (\ref{lef}) is not Zariski dense in $A_{k(B)}$. Combining this with the previous paragraph, we get the corollary.
\end{proof}
Now we state the main result in this subsection.
\begin{theorem}\label{kwuguan}
Let $B$ be a smooth projective geometrically integral curve over a field $k$ of characteristic $0$ and let $K=k(B)$ be its function field. Let $A$ be an abelian variety over $K$. Let $L$ be an ample line bundle on $A$ and let $Y\subseteq A$ be a closed subscheme of $A$ of codimension $>1$. Then for any $\varepsilon>0$, there is a proper Zariski closed subset $Z_{\varepsilon}$ of $A$ such that for any finite extension $K'/K$, the inequality
$$
h_{Y}(x)\leq\varepsilon h_{L}(x)+O_{\varepsilon,K'}(1)
$$
holds for all $x\in (A\backslash Z_{\varepsilon})(K')$. Here $O_{\varepsilon,K'}(1)$ is a bounded term depending on $\varepsilon$ and $K'$.
\end{theorem}
We need the following lemma on abelian varieties.
\begin{lemma}\label{quxian}
Let $A$ be an abelian variety over an algebraically closed field $k$ of characteristic $0$. Then there is a smooth projective curve $B$ over $k$ such that for any divisor $D$ on $A$ which is not nef, there exists a non-constant $k$-morphism $f\colon B\to A$ such that $\deg_{B}f^{*}D<0$.
\end{lemma}
\begin{proof}
Set
\[
E_{\QQ}
  =\operatorname{End}(A)\otimes_{\BZ}\QQ,
\qquad
E_{\RR}
  =E_{\QQ}\otimes_{\QQ}\RR,
\qquad
G=E_{\RR}^{\times}.
\]
Fix a polarization on $A$, and denote the associated Rosati involution by $\dagger$. By \cite[Corollary~3.5 and Theorem~4.3]{PS12}, there are identifications
\[
E_{\RR}\simeq\prod_{j=1}^{r}M_{n_j}(F_j),
\qquad
N^1(A)_{\RR}\simeq
V\colonequals\prod_{j=1}^{r}\operatorname{Herm}_{n_j}(F_j),
\qquad
F_j\in\{\RR,\CC,\BH\},
\]
under which $\dagger$ is conjugate transpose. The nef cone ${\rm Nef}(A)\subseteq N^{1}(A)_{\RR}$ corresponds to
\[
\CP
  \colonequals\prod_{j=1}^{r}\operatorname{Herm}_{n_j}^{+}(F_j),
\]
where the superscript $+$ denotes positive semidefinite matrices.

Consider the nondegenerate pairing on $V$
\[
\langle X,Y\rangle
  =\sum_{j=1}^{r}
    \operatorname{Re}\operatorname{tr}(X_jY_j).
\]
The cone $\mathcal P$ is self-dual with respect to this pairing. Thus, identifying $N_1(A)_{\RR}=\bigl(N^1(A)_{\RR}\bigr)^\vee$ with $V$ through this pairing, the duality
\[
\ol{\operatorname{NE}}(A)
  =\operatorname{Nef}(A)^\vee
\]
identifies the closed cone of curves with $\mathcal P$.

The pullback action on divisor classes is
$X\mapsto u^\dagger Xu$. Since $\langle u^\dagger Xu,Y\rangle=\langle X,uYu^\dagger\rangle$, the projection formula shows that the pushforward action on curve classes extends to the action
\[
u_*Y=uYu^\dagger
\qquad (u\in G).
\]
Its orbits on $\CP$ are classified by the ranks in the respective factors. In particular, there are only finitely many orbits.

From each orbit containing the class of an integral curve (that is, an element $[Z]\in \ol{\operatorname{NE}}(A)$, where $Z\subseteq A$ is a curve) on $A$, choose one such curve $Z_i$. Let
\[
\nu_i:C_i\longrightarrow Z_i\hookrightarrow A,
\qquad i=1,\ldots,m,
\]
be the normalizations, and write $c_i=(\nu_i)_*[C_i]$. Fix a smooth projective curve $B$ admitting finite surjective morphisms $p_i:B\to C_i$ for all $i$.

Now let $D$ be a non-nef divisor on $A$. Choose a curve $Z\subset A$ with $D\cdot Z<0$. Then there exist $i$ and $u\in G$ such that
\[
[Z]=u_*c_i.
\]
Since $E_{\QQ}^{\times}$ is dense in $G$, continuity gives
an element $v\in E_{\QQ}^{\times}$ such that
\[
D\cdot v_*c_i<0.
\]
Choose a positive integer $N$ with $\phi=Nv\in\operatorname{End}(A)$. The action on curve classes is quadratic in $v$, so
\[
\deg(\phi\circ\nu_i)^*D
  =D\cdot\phi_*c_i
  =N^2D\cdot v_*c_i<0.
\]
Therefore the morphism $f=\phi\circ\nu_i\circ p_i:B\to A$ satisfies
\[
\deg f^*D
  =\deg(p_i)\,\deg(\phi\circ\nu_i)^*D<0.
\]
This proves the assertion.
\end{proof}
Before proving Theorem \ref{kwuguan}, we generalize the GCD height functions to the high-dimensional setting as follows.

For a morphism $f\colon X'\to X$ between smooth $k$-varieties and a closed subscheme $Y$ of $X$ with $f(X')\nsubseteq Y$, define the \emph{generalized pullback divisor}
$$
[f^{*}Y]_{\rm div}\colonequals \sum\limits_{D'\subseteq X'}({\rm ord}_{D'}f^{*}Y)\cdot D',
$$
where $D'$ runs over all prime divisors on $X'$, and the order ${\rm ord}_{D'}f^{*}Y$ is defined as follows: write $Y$ as a scheme-theoretic intersection $D_{1}\cap\dots\cap D_{l}$ of effective Cartier divisors, and then let ${\rm ord}_{D'}f^{*}Y$ be the minimum of the coefficients of $f^{*}D_{i}\,(1\leq i\leq l)$ in $D'$. This definition is independent of the choice of $D_{i}$. Since $X',X$ are assumed to be smooth, $[f^{*}Y]_{\rm div}$ can be viewed as an effective Cartier divisor on $X'$. We also denote by $[f^{*}Y]_{\rm lb}$ the associated line bundle of $[f^{*}Y]_{\rm div}$, and call it the \emph{generalized pullback line bundle}.

When $X'$ is a smooth projective curve, then $\deg_{X'}[f^{*}Y]_{\rm div}$ is just the GCD height $h_{Y}(f)$ of $f$ (viewed as a rational point in $A(k(X'))$).

\begin{corollary}\label{bl}
Let $A$ be an abelian variety over an algebraically closed field $k$ of characteristic $0$. Let $L$ be an ample line bundle on $A$ and let $Y$ be a closed subscheme of $A$ of codimension at least $2$. Fix $\varepsilon>0$. Then the union in $A$ of the images of all closed immersions $i\colon A'\hookrightarrow A$ from abelian varieties to $A$ with $i(A')\nsubseteq Y$ and such that $\varepsilon i^{*}L-[i^{*}Y]_{\rm lb}$ is not ample on $A'$ is not Zariski dense in $A$.
\end{corollary}

\begin{proof}
By \cite{LOZY}, up to isomorphism, there are only finitely many abelian varieties $A'$ that admit a closed immersion into $A$. Therefore we only need to show that for any fixed abelian variety $A'$, the union in $A$ of the images of all closed immersions $i\colon A'\hookrightarrow A$ from $A'$ to $A$ with $i(A')\nsubseteq Y$ and $\varepsilon i^{*}L-[i^{*}Y]_{\rm lb}$ is not ample on $A'$ is not Zariski dense in $A$.

Fix an abelian variety $A'$. Applying Lemma \ref{quxian} to $A'$, there is a smooth projective curve $B$ over $k$ such that for any divisor $D'$ on $A'$ which is not nef, there exists a morphism $g\colon B\to A'$ with $\deg_{B}g^{*}D'<0$.

Now for any closed immersion $i\colon A'\hookrightarrow A$ such that $i(A')\nsubseteq Y$ and $\varepsilon i^{*}L-[i^{*}Y]_{\rm lb}$ is not ample on $A'$, the divisor $D(i)\colonequals \frac{\varepsilon}{2} i^{*}L-[i^{*}Y]_{\rm lb}$ is not nef on $A'$. Thus there is a morphism $g\colon B\to A'$ such that $\deg_{B}g^{*}D(i)<0$. Moreover, since composing $g$ with a (general) translation of $A'$ does not change the numerical class of $g^{*}D(i)$, we conclude that the union in $A'(k(B))$ of all possible $g$ (viewed as a point in $A'(k(B))$) above is Zariski dense in $A'_{k(B)}$. Let $f=i\circ g\colon B\to A$ be the composition. Then we have $\deg_{B}(\frac{\varepsilon}{2}f^{*}L-[f^{*}Y]_{\rm lb})<0$. Viewing $f$ as a rational point of the constant abelian variety $A_{k(B)}$ over $k(B)$, we see that this just means that $h_{Y}(f)>\frac{\varepsilon}{2}h_{L}(f)$. Moreover, the union in $i(A'(k(B)))\subseteq A(k(B))$ of all such $f$ is Zariski dense in $i(A'_{k(B)})$.

Assume to the contrary that the union in $A$ of images of all closed immersions $i\colon A'\hookrightarrow A$ is Zariski dense in $A$. Then by the argument in the previous paragraph, we see that the union of all maps $f\colon B\to A$ (viewed as points in $A(k(B))$) with $h_{Y}(f)>\frac{\varepsilon}{2}h_{L}(f)$ in $A(k(B))$ is Zariski dense. This contradicts Corollary \ref{lefschetz}.
\end{proof}

\begin{proof}[Proof of Theorem \ref{kwuguan}]
Let $Z'_{\varepsilon}$ be the Zariski closure of the union in $A$ of the images of all closed immersions $i\colon A'\hookrightarrow A_{\ol{K}}$ from abelian varieties over $\ol{K}$ to $A_{\ol{K}}$ with $i(A')\nsubseteq Y$ and $\frac{\varepsilon}{2} i^{*}L_{\ol{K}}-[i^{*}Y_{\ol{K}}]_{\rm lb}$ is not ample on $A'_{\ol{K}}$, which is a proper Zariski closed subset of $A_{\ol{K}}$ by Corollary \ref{bl}. It can be shown that $Z'_{\varepsilon}$ is invariant under the action of ${\rm Gal}(\ol{K}/K)$, and thus it is defined over $K$. Let $Z_{\varepsilon}=Z'_{\varepsilon}\cup Y$, which depends only on $\varepsilon$. We show that $Z_{\varepsilon}$ satisfies the theorem.

It suffices to show that
\begin{align}\label{yizhi}
h_{Y}(x)\leq \varepsilon h_{L}(x)+O_{\varepsilon,K}(1)
\end{align}
holds for all $x\in (A\backslash Z_{\varepsilon})(K)$. A similar conclusion holds for any finite extension $K'/K$. By Theorem \ref{main}, there is a proper Zariski closed subset $Z$ of $A$ such that the inequality (\ref{yizhi}) holds for all $x\in (A\backslash Z)(K)$. We need to show that $Z$ can be chosen so that it is contained in $Z_{\varepsilon}$. In fact, replacing $Z$ by the Zariski closure of $Z(K)$ in $A$, we may assume that $Z(K)$ is Zariski dense in $Z$. Let $Z_{0}$ be an irreducible component of $Z$. Then $Z_{0}(K)$ is also Zariski dense in $Z_{0}$. If $Z_{0}\subseteq Y$, then $Z_{0}$ is also contained in $Z_{\varepsilon}$. Now assume that $Z_{0}\nsubseteq Y$.

If there is a proper Zariski closed subset $W_{0}$ of $Z_{0}$ such that $h_{Y}(x)-\varepsilon h_{L}(x)$ is bounded above for all rational points in $(Z_{0}\backslash W_{0})(K)$ , then we can replace $Z_{0}$ by $W_{0}$ and replace $O_{\varepsilon,K}(1)$ by a larger bounded term. Now we assume that $h_{Y}(x)-\varepsilon h_{L}(x)$ is unbounded from above on any $U(K)$ (where $U$ is a nonempty open subset of $Z_{0}$). Then by Lemma \ref{real}, there is a generic sequence $\{x_{n}\}_{n\geq1}$ in $Z_{0}(K)$ with $h_{Y}(x_{n})-\varepsilon h_{L}(x_{n})\to\infty$. In particular, we have $h_{L}(x_{n})\to\infty$. We show that $Z_{0}$ is contained in $Z'_{\varepsilon}$ in this case.

Let $A'$ be the identity component of the set $\{a\in A\big|\, a+Z_{0}\subseteq Z_{0}\}$, which is an abelian subvariety of $A$. Let $Z_{0}/A'$ be the image of $Z_{0}$ in $A/A'$. By Ueno's fibration theorem, $Z_{0}/A'$ is of general type. Moreover, the Zariski density of $Z_{0}(K)$ in $Z_{0}$ implies the Zariski density of $(Z_{0}/A')(K)$ in $Z_{0}/A'$.

Similarly to the proof of Proposition \ref{buium}, choose an isogeny $i\colon A'\times(A/A')\to A$ such that the composition of the inclusion $A'\hookrightarrow A'\times(A/A')$ with $i$ is equal to the multiplication $a'\mapsto ma'$ on $A'$ by some integer $m$. Let $Z'_{0}$ be an irreducible component of the inverse image $i^{-1}(Z_{0})\subseteq A'\times(A/A')$ and let $Z'_{0}/A'$ be the image of $Z'_{0}$ in $A/A'$ under the natural projection. Then we have $Z'_{0}=A'\times(Z'_{0}/A')$. Note that $Z'_{0}/A'$ has a finite map to $Z_{0}/A'$, and thus $Z'_{0}/A'$ is of general type. Applying (geometric) Chevalley--Weil theorem to $i\big|_{Z'_{0}}\colon Z'_{0}\to Z_{0}$, we obtain a finite field extension $K'/K$ such that $Z_{0}(K)\subseteq i(Z'_{0}(K'))$. Thus we can lift each point $x_{n}\in Z_{0}(K)$ to $x'_{n}\in Z'_{0}(K')$. Since $Z_{0}(K)$ is Zariski dense in $Z_{0}$, we have $Z'_{0}(K')$ is Zariski dense in $Z'_{0}$, and thus $(Z'_{0}/A')(K')$ is Zariski dense in $Z'_{0}/A'$. As $Z'_{0}/A'$ is a subvariety of general type of the abelian variety $A/A'$, by the resolved (geometric) Bombieri--Lang conjecture, the normalization of $Z'_{0}/A'$ is isomorphic to a constant variety $T_{K}$, where $T$ is a projective variety over $k$. Moreover, the complement $T_{K}(K')\backslash T(k)$ is not Zariski dense in $T_{K}$. Take a resolution of singularities of $T'\to T$ over $k$. Consider the composition $f\colon A'\times T'_{K}\to A'\times T_{K}\to A'\times (Z'_{0}/A')=Z'_{0}\to Z_{0}$, which is generically finite. Since $\{x_{n}\}_{n\geq1}$ is generic in $Z_{0}$, $\{x'_{n}\}_{n\geq1}$ is also generic in $Z'_{0}$. Hence for $n$ large enough, the point $x'_{n}$ has a unique lift $(s_{n},t'_{n})$ on $A'(K')\times T'_{K}(K')$, where $t'_{n}$ belongs to $T'(k)$.

Let $g\colon A'\times T'_{K}\to A$ be the composition $A'\times T'_{K}\overset{f}\to Z_{0}\hookrightarrow A$. We show that $\frac{\varepsilon}{2} g^{*}L-[g^{*}Y]_{\rm lb}$ is not relatively ample over $T'_{K}$.  In fact, assume to the contrary that it is relatively ample over $T'_{K}$. Then for $n$ large enough, we have
\begin{align*}
\frac{\varepsilon}{2}h_{L}(x_{n})-h_{[g^{*}Y]_{\rm div}}((s_{n},t'_{n}))=\frac{\varepsilon}{2}h_{g^{*}L}((s_{n},t'_{n}))-h_{[g^{*}Y]_{\rm div}}((s_{n},t'_{n}))\geq O(1).
\end{align*}
Let $Y'$ be the (scheme-theoretic) union of all irreducible components of $g^{-1}(Y)$ which are of codimension at least 2 in $A'\times T'_{K}$. Then we have $h_{Y}(x_{n})\leq h_{Y'}((s_{n},t'_{n}))+h_{[g^{*}Y]_{\rm div}}((s_{n},t'_{n}))$ for $n$ large enough. Since $L$ is ample, $g^{*}L$ is relatively ample over $T'_{K}$. By Proposition \ref{V}(3), we have
\begin{align*}
h_{Y'}((s_{n},t'_{n}))\leq \frac{\varepsilon}{2}h_{g^{*}L}((s_{n},t'_{n}))=\frac{\varepsilon}{2}h_{L}(x_{n})
\end{align*}
for $n$ large enough. Combining all inequalities above, we see that $h_{Y}(x_{n})\leq\varepsilon h_{L}(x_{n})+O(1)$ for $n$ large enough. This contradicts our assumption that $h_{Y}(x_{n})-\varepsilon h_{L}(x_{n})\to\infty$.

Now we have $\frac{\varepsilon}{2} g^{*}L-[g^{*}Y]_{\rm lb}$ not relatively ample over $T'_{K}$. This implies that for general $t'\in T'_{K}$, the generalized pullback of $\frac{\varepsilon}{2} L-Y$ to $A'\times\{t'\}$ is not ample. Hence for general $z\in Z_{0}/A'$, the generalized pullback of $\frac{\varepsilon}{2} L-Y$ to $q^{-1}(z)$ is not ample. Here $q\colon A\to A/A'$ is the quotient map. Note that $q^{-1}(z)$ is a translation of $A'$ in $A$. Thus by definition, we have $Z_{0}\subseteq Z'_{\varepsilon}$. This completes the proof.
\end{proof}

By using the same methods, we can prove that the exceptional subset in Theorem \ref{main1} can also be chosen to be independent of $K$. In fact, in Corollary \ref{bl}, instead of considering the ampleness of the generalized pullback $\varepsilon i^{*}L-[i^{*}Y]_{\rm lb}$, we consider the ampleness of the divisor $(i^{*}D)_{\rm red}-(1-\varepsilon)i^{*}D$. All remaining arguments are the same and we omit the detailed proof here.\medskip

\noindent\textbf{Acknowledgements.} I would like to express my sincere gratitude to my advisors Junyi Xie and Xinyi Yuan for their consistent encouragement. I am especially indebted to Junyi Xie for valuable insights shared in our discussions and for his assistance in using ChatGPT to review the proofs. I am grateful to Pietro Corvaja, Umberto Zannier and Francesco Zucconi for sharing their recent related paper \cite{CZZ}, and to Katsutoshi Yamanoi for many helpful communications on an earlier version of this article. The discussion in \S\ref{4.6} concerning the $K$-independence of the exceptional set originated from my email discussions with Pietro Corvaja and Umberto Zannier, and from related questions raised by William Cherry and Riccardo Pengo; I thank them all. I would also like to thank Laura Capuano, Paolo Dolce, Carlo Gasbarri, Ruoyi Guo and Amos Turchet for their helpful comments on this work.

The research was partially completed while the author was visiting the Vietnam Institute for Advanced Study in Mathematics and Westlake University in 2026. I would like to thank these institutions for their hospitality. This work is supported by the National Natural Science Foundation of China Grants No. 12271007, No. 12250004 and No. 12321001.
\medskip\medskip

\noindent\textbf{AI disclosure.} After the first draft, the author has used DeepSeek to check grammar and typos. Junyi Xie has also used ChatGPT 5.6 Pro Sol to assist with the review of the proofs and with searching for relevant literature. All mathematical arguments and the final text were written by the author, who takes full responsibility for the content of the article.

\bibliographystyle{alpha}
\bibliography{refs}

\

\noindent \small{ Peking University, Beijing 100871, China}

\noindent \small{\it Email: gaoguoquan@pku.edu.cn}

\end{document}